\documentclass[reqno,11pt]{amsart}
\usepackage{amsfonts,amssymb,latexsym}
\usepackage{amsmath,amsthm,amsxtra}
\usepackage[latin1]{inputenc}
\usepackage[T1]{fontenc}
\usepackage{lmodern}
\usepackage{a4wide}

\newtheorem{theorem}{Theorem}[section]
\newtheorem{lemma}[theorem]{Lemma}
\newtheorem{proposition}[theorem]{Proposition}

\newtheorem*{theorem*}{Theorem}
\newtheorem*{proposition*}{Proposition}
\theoremstyle{definition}
\newtheorem{definition}[theorem]{Definition}

\theoremstyle{remark}\newtheorem{remark}{Remark}[section]

\numberwithin{equation}{section}

\begin{document}
\title[Minimal mass blow-up for mBO]{Construction of a minimal mass blow up solution of the modified Benjamin-Ono equation}

\author[Y. Martel]{Yvan Martel}
\address{CMLS, \'Ecole Polytechnique, CNRS, Universit\'e Paris-Saclay, 91128 Palaiseau, France}
\email{yvan.martel@polytechnique.edu}
\author[D. Pilod]{Didier Pilod}
\address{Instituto de Matem\'atica, Universidade Federal do Rio de Janeiro, Caixa Postal 68530, CEP: 21945-970, Rio de Janeiro, RJ, Brazil}
\email{didierpilod@gmail.com}

\begin{abstract} 
We construct a minimal mass blow up solution  of the modified Benjamin-Ono equation (mBO)
\[
u_{t}+(u^3-D^1 u)_{x}=0, \eqno{\rm (mBO)}
\] 
which is a standard mass critical dispersive model.
Let $Q\in H^{\frac 12}$, $Q>0$, be the unique ground state solution of $D^1 Q +Q=Q^3$, constructed using variational arguments by Weinstein (Comm. PDE, 12 (1987), J. Diff. Eq., 69 (1987)) and Albert, Bona and Saut (Proc. Royal London Soc., 453 (1997)), and whose uniqueness was recently proved   by Frank and Lenzmann (Acta Math., 210 (2013)).

We show the existence of a solution $S$ of (mBO) satisfying
$\|S \|_{L^2}=\|Q\|_{L^2}$  and
\[
S(t)-\frac1{\lambda^{\frac12}(t)} Q\left(\frac{\cdot - x(t)}{\lambda(t)}\right)\to 0\quad \mbox{ in }\ H^{\frac 12}(\mathbb R) \mbox{ as }\ t\downarrow 0,
\]
where 
\[
\lambda(t)\sim t,\quad x(t) \sim -|\ln t| \quad \hbox{and}\quad
\|S(t)\|_{\dot H^{\frac 12}} \sim  t^{-\frac 12}\|Q\|_{\dot H^{\frac 12}} \quad \hbox{as}\ t\downarrow 0.
\]
This existence result is analogous to the one obtained  by Martel, Merle and Rapha\"el (J. Eur. Math. Soc., 17 (2015)) for the mass critical generalized Korteweg-de Vries equation. However, in contrast with the (gKdV) equation, for which the blow up problem is now well-understood   in a neighborhood of the ground state,   $S$ is the first example of blow up solution for (mBO).

The proof involves the construction of  a blow up profile, energy estimates as well as refined localization arguments, developed in the context of Benjamin-Ono type equations by Kenig, Martel and Robbiano  (Ann. Inst. H. Poincaré, Anal. Non Lin., 28 (2011)).
Due to the lack of information on the (mBO) flow around the ground state,  the energy estimates have to be considerably sharpened in the present paper.
 \end{abstract}

\maketitle

\section{Introduction}
\subsection{Main result}
We consider the modified Benjamin-Ono  equation (mBO)
\begin{equation}  \label{mBO}
u_t+\big(u^3-\mathcal{H}u_x\big)_x=0 \, ,\quad    t \in \mathbb R\,,\ x \in \mathbb R\,,
\end{equation}
where $u(t,x)$ is a real-valued function  and $\mathcal{H}$ denotes the Hilbert transform, defined by
\begin{displaymath} 
\mathcal{H}f(x)=  \frac1\pi \, \text{p.v.}\int_{\mathbb
R}\frac{f(y)}{x-y}dy \, .
\end{displaymath}
Observe that with this convention $\mathcal{H}\partial_x=D^1$, where  $D^\alpha$ is the Riesz potential of order $-\alpha$, defined via Fourier transform by $(D^\alpha f)^{\wedge}(\xi)=|\xi|^\alpha\widehat{f}(\xi)$, for any $\alpha \in \mathbb R$.
We see equation (mBO) as a natural generalization of the classical quadratic Benjamin-Ono equation
\begin{equation}  \label{BO}
u_t+\big(u^2-\mathcal{H}u_x\big)_x=0 \, ,\quad    t \in \mathbb R\,,\ x \in \mathbb R\,,
\end{equation}
 introduced by Benjamin \cite{Ben} and Ono~\cite{On} and intensively studied since then, both mathematically and numerically,  as a model for one-dimensional waves in deep water.
The cubic nonlinearity for the Benjamin-Ono model is also relevant as a long wave model, see e.g. Abdelouhab, Bona,  Felland and~Saut \cite{AbBoFeSa} and Bona and Kalisch~\cite{BoKa}. At first sight, the relation between \eqref{mBO} and \eqref{BO}   seems similar to the one between the (cubic) modified KdV equation and the   Korteweg-de Vries equation, but
     \eqref{mBO} is not completely integrable and no algebraic relation relates these two models. Another difference is that with  dispersion of the Benjamin-Ono type,  a cubic nonlinearity  leads  to instable waves.
  More generally, nonlinear one dimensional models with weak dispersion seem of great physical interest, see e.g. Klein and Saut~\cite{KlSa} and Linares, Pilod and Saut~\cite{LiPiSa}. Equation (mBO) is a typical   model with interesting mathematical properties, which can be seen as an intermediate step between the well-studied generalized (KdV) equations and other relevant models with    weak dispersion.

\medskip

The following quantities are formally invariant by the flow associated to (mBO)
\begin{equation} \label{mass}
M(u)=\frac12 \int_{\mathbb R} u^2dx \quad \hbox{and}\quad 
E(u)=\frac12 \int_{\mathbb R} |D^{\frac12}u|^2dx-\frac14\int_{\mathbb R}u^4dx \, .
\end{equation}
Note the scaling symmetry: if $u(t,x)$ is solution then $u_{\lambda}(t,x)$ defined by
\begin{equation} \label{ulambda}
u_{\lambda}(t,x)=\lambda^{-\frac12}u(\lambda^{-2} t, \lambda^{-1} x)
\end{equation} 
is also solution. Since this transformation leaves
  the $L^2$ norm invariant,   the problem is \emph{mass critical}.
Recall that the Cauchy problem for \eqref{mBO} is locally well-posed in the energy space $H^{\frac12}(\mathbb R)$ by the work of Kenig and Takaoka  \cite{KeTa}: for any $u_0 \in H^{\frac12}(\mathbb R)$, there exists a unique (in a certain sense) maximal solution of \eqref{mBO} in $C([0,T^\star):H^{\frac12}(\mathbb R))$ satisfying $u(0,\cdot)=u_0$.  
Moreover, the flow map data-solution is locally Lipschitz. (See also Tao \cite{Tao}, respectively Molinet and Ribaud \cite{MoRi,MoRi1} and Kenig, Ponce and Vega \cite{KPV1}, for previous related works on the Benjamin-Ono equation, respectively the modified Benjamin-Ono equation.)  For such solutions, the quantities $M(u(t))$ and $E(u(t))$ are   conserved.  Moreover, if $T^\star<+\infty$ then $\lim_{t \uparrow T^\star} \| D^{\frac12}u(t)\|_{L^2}=+\infty$ and more precisely, by a scaling argument, 
$\|D^{\frac12}u(t)\|_{L^2}\gtrsim (T^\star-t)^{-\frac 14}$, for $t<T^\star$ close to $T^\star$.
We refer to Sect.~\ref{S.2.4} for more details.

\medskip

From works of Weinstein \cite{We1, We2} and Albert, Bona and Saut \cite{AlBoSa}, there exists an even \textit{ground state} solution $ Q\in H^{\frac 12}(\mathbb R),$ $Q>0$ of the stationary problem
\begin{equation} \label{E}
D^1Q+Q-Q^3=0 ,  
\end{equation}
related to the best constant in the following  Gagliardo-Nirenberg inequality 
\begin{equation}\label{gn}
\hbox{for all $v\in H^{\frac 12}(\mathbb R)$,}\quad 
\int v^4 \leq 2 \int |D^{\frac 12} v|^2 \left(\frac {\int v^2}{\int Q^2}\right).
\end{equation}
Frank and Lenzmann \cite{FrLe} proved very general uniqueness results of nonlinear ground states for fractional Laplacians in $\mathbb R$ that include the model \eqref{E}. As a consequence, $Q$ is the \emph{unique} ground state solution of \eqref{E} up to the symmetries of the equation. Their work also includes a decisive description of the spectrum of the linearized operator around $Q$. We refer to Sect.~\ref{SW} for more details.

\medskip

Following a classical observation due to Weinstein \cite{We82}, the conservation laws \eqref{mass}, the inequality \eqref{gn}  and the Cauchy theory \cite{KeTa}  imply that any initial data $u_{0}\in H^{\frac 12}(\mathbb R)$ with subcritical mass, i.e. satisfying  $\|u_0\|_{L^2}<\|Q\|_{L^2}$ generates a \emph{global and  bounded} solution in $H^{\frac 12}$.
In this paper, we show that this condition is sharp by constructing a \emph{minimal mass blow up} solution, i.e. 
a solution  of  (mBO) which blows up in finite time in $H^{\frac 12}$ with the threshold mass $\|Q\|_{L^2}$.
Actually, this solution is  the first example of blow up solution for (mBO).

\begin{theorem}\label{th1}
There exist  $T_{0}>0$ and a solution $S\in C((0,T_{0}]\, :\, H^{\frac 12}(\mathbb R))$ to~{\rm (mBO)}
such that 
\begin{equation}\label{mm}
\|S(t)\|_{L^2}=\|Q\|_{L^2}\quad \mbox{for all $t\in (0,T_{0}]$,}
\end{equation}
and
\begin{equation}\label{d:S}
S(t)-\frac1{\lambda^{\frac12}(t)} Q\left(\frac{\cdot - x(t)}{\lambda(t)}\right)\to 0\quad \mbox{in $H^{\frac 12}(\mathbb R)$   as $t\downarrow 0$},
\end{equation}
where  the functions $t\mapsto\lambda(t)$ and $t\mapsto x(t)$ satisfy
\begin{equation}\label{th:param}
\lambda(t)\sim t,\quad x(t) \sim -|\ln t| \quad \hbox{as}\ t\downarrow 0.
\end{equation}
In particular, 
\begin{equation}\label{th:rate}
\|D^{\frac12} S(t)\|_{L^2} \sim   {t^{-\frac 12}} {\|D^{\frac12} Q\|_{L^2}} \quad \mbox{as $t\downarrow 0$.}
\end{equation}
\end{theorem}
Note that Theorem~\ref{th1} implies readily the orbital instability (in $L^2(\mathbb R)$ and thus also  in $H^{\frac 12}(\mathbb R)$) of the   solution $u(t,x)=Q(x-t)$, which also seems to be new for (mBO). Indeed, for any $n>1/T_{0}$, let $t_{n}=1/n$, $T_{n}=(T_{0}-t_{n}) \lambda(t_{n})^{-2}>0$,  and 
\[
u_{0,n}(x) = \lambda^{\frac 12} (t_{n}) S(t_n,\lambda(t_{n}) x + x(t_{n}))\quad \hbox{so that}\quad
\lim_{n\to +\infty}\|u_{0,n}-Q\|_{H^{\frac 12}}=0.
\]
Then, the corresponding solution $u_{n}$ of \eqref{mBO} writes, for $t\in [0,T_{n}]$,  \[
u_{n}(t,x) = \lambda^{\frac 12} (t_{n}) S(t_n+\lambda(t_{n})^2 t,\lambda(t_{n}) x + x(t_{n}))
\]
and thus,
\begin{multline}
\inf\{ \|u_{n}(T_{n}) - \lambda_{1}^{\frac 12} Q(\lambda_{1} \cdot + x_{1})\|_{L^2}\, : \, {x_{1}\in \mathbb R,\lambda_{1}>0}\} \\
=  \inf\{ \|S(T_{0})- \lambda_{1}^{\frac 12} Q(\lambda_{1} \cdot + x_{1})\|_{L^2}\, : \, {x_{1}\in \mathbb R,\lambda_{1}>0}\}
=c_{0} > 0.
\end{multline}
It is also clear that the blow up behavior displayed by the solution $S$ is   unstable  since for any initial data with mass less that $\|Q\|_{L^2}$, the corresponding solution is global and bounded.
From \eqref{th:rate}, we see that $S(t)$ blows up as $t\downarrow 0$ in $H^{\frac 12}(\mathbb R)$ twice as fast as the lower bound given by the Cauchy theory. 

\subsection{Comments and references}
Historically, blow up results for nonlinear dispersive  PDE were first obtained  by global obstruction arguments, such as  the Virial identity for the nonlinear Schr\"odinger equations (NLS) and related models. More rarely,   explicit blow up solutions (the most famous one for (NLS) is reproduced in \eqref{Snls}) would give a description of some special forms of blow up. 
In the 80's and 90's, variational arguments and a refined understanding of the linearized operator  around the ground state led to original blow up constructions and classification results related to rescaled solitary waves, see in particular Weinstein \cite{We,We86}, Merle and Tsutsumi~\cite{MeTs}, Merle \cite{Me,Meduke} and Bourgain-Wang \cite{BW}.
 Numerical experiments were also used to try to predict blow up rates.
 We refer to Cazenave \cite{Ca03} and references therein.
Such directions were more recently systematically studied, in particular for the mass critical generalized KdV equation, for mass critical NLS equations and several other related models. A few sample  results will be reviewed below
 (mainly from \cite{MaMejmpa,MaMe,MaMeRa1,Mejams,MeRagafa,MeRainvent,Pe}). It became clear that a refined study of the flow of the evolution equation around the ground state family was the key to the understanding of the blow up dynamics with one bubble, both for  \emph{stable} and \emph{unstable}  forms of blow up.
Theorem~\ref{th1} above belongs to this category of results and methods, providing a quite explicit   blow up solution. In a situation where very few is known on the flow around the ground state,   considering a ``doubly critical'' situation (both critical exponent and critical mass) is a way to enjoy a lot of structure and rigidity, idea which goes back to \cite{We86} and \cite{Meduke}.

\medskip

Now, we give more details on previous related results, starting with the closest models.

\medskip

For the mass critical generalized Korteweg-de Vries equation (gKdV)
\begin{equation}\label{kdv}
u_t + (u_{xx} + u^5)_x =0, \quad x\in \mathbb R,\ t\in \mathbb R,
\end{equation}
(the energy space for \eqref{kdv} is $H^1(\mathbb R)$),
an existence result similar to Theorem~\ref{th1} was proved by Martel, Merle and Rapha\"el \cite{MaMeRa2}, and then sharpened by
Combet and Martel \cite{CoMa}.
More precisely, let $Q_{\rm KdV}\in H^1(\mathbb R)$, $Q_{\rm KdV}>0$ be the ground state for \eqref{kdv}, i.e. the unique even solution of 
$Q_{\rm KdV}''+Q_{\rm KdV}^5=Q_{\rm KdV}$.
It follows from \cite{MaMeRa2} and \cite{CoMa} that there exists a solution $S_{\rm KdV}$ 
on $(0,+\infty)$ such that $\|S_{\rm KdV}(t)\|_{L^2}=\|Q_{\rm KdV}\|_{L^2}$ and 
\begin{equation}\label{Skdv}\begin{aligned}
& S_{\rm KdV}(t) - \frac 1{t^{\frac 12}} Q_{\rm KdV}\left(\frac{\cdot+\frac1t}t+c_0\right) \to 0 
\quad \mbox{in   $H^1(\mathbb R)$},\\
& \|(S_{\rm KdV})_{x}(t)\|_{L^2}\sim t^{-1} \|Q'\|_{L^2}, \quad  \mbox{ as $t\downarrow 0$},
\end{aligned}\end{equation}
for some constant $c_{0}$. We see that the singularity has the form of a blow up bubble with the same scaling $\lambda(t)\sim t$, as in \eqref{d:S}.
 A main qualitative difference  is the speed of the bubble as $t\downarrow 0$, since for (gKdV), $x_{\rm KdV}(t)\sim -\frac 1t$, whereas $x(t)\sim -|\ln t| $ for (mBO). In this respect, (mBO) seems to be a threshold  case in the family of critical equations \eqref{gBO}.
 
\smallskip

For (mBO), the information obtained in the present paper on the parameters $\lambda(t)$ and $x(t)$ as $t\downarrow 0$ is not sufficient to replace them by their explicit asymptotics in the convergence result \eqref{d:S}; see Remark~\ref{rk.precise} for more details. 
The result \eqref{Skdv}   for (gKdV) is thus more precise. In fact, for (gKdV), the minimal mass blow up is   quite well understood, at least close to the blow up time: in addition to \eqref{Skdv}, sharp asymptotics, both in time (as $t\downarrow 0$) and in space (as $x\to \pm\infty$) were derived in \cite{CoMa}, for any level of derivative of $S_{\rm KdV}$. Importantly, $S_{\rm KdV}$ is also known to be global for $t>0$ and to be the \emph{unique} minimal mass solution of (gKdV), up to the symmetries of the equation (scaling, translations and sign change),  see \cite{MaMeRa2}.
For (mBO), such properties are open problems.

\smallskip

Recall that the existence  and uniqueness of the minimal mass solution $S_{\rm KdV}$ is only a part of the results obtained in \cite{MaMejmpa,MaMe,MaMeNaRa,MaMeRa1,MaMeRa2,MaMeRa3,Mejams} on the description of the blow up phenomenon around the ground state (or soliton) for (gKdV) and, more generally, on the classification of the long time behavior of solutions close to the soliton.
Those works focus on  the case of  \emph{slightly supercritical mass} $H^1$ initial data
\begin{equation} \label{in:1}
\|Q_{\rm KdV}\|_{L^2} \leq \|u_0\|_{L^2} < (1+\delta_0) \|Q_{\rm KdV}\|_{L^2} \quad \mbox{where }\ 0<\delta_0\ll 1.
\end{equation}
In this context, the  main results  can be summarized as follows:
(1) The ground state $Q_{\rm KdV}$ is a universal blow up profile;
(2) General $H^1$ initial data with negative energy lead to blow up in finite or infinite time;
(3) For initial data close to $Q_{\rm KdV}$ in a topology stronger than $H^1$ (based on $L^2$ weighted norm), only three   behaviors are possible: (Blowup) with speed $(T^\star-t)^{-1}$,  (Soliton)  and (Exit). 
It is also proved in~\cite{MaMeNaRa} that the (Soliton) case (solutions that converge in a local sense to a bounded soliton) corresponds
to a codimension one manifold of initial data which separates the (Blowup) and (Exit) cases.
The (Exit) case refers to solutions that eventually leave  any small neighborhood of the soliton. It is expected (but yet an open problem) that such solutions behave as a solution of the linear problem as $t\to +\infty$.
Finally, blow up solutions with various blow up rates (in finite or infinite time) are constructed in~\cite{MaMeRa3}
for initial data arbitrarily close to~$Q$ in the energy space. Concerning the critical and supercritical gKdV equations, we also refer to Klein and Peter \cite{KlPe} and references therein for detailed numerical studies.

\smallskip

In contrast, we recall that Theorem \ref{th1} is the first blow up result for the (mBO) equation.
The difficulties in extending techniques and results from (gKdV) to (mBO) are multiple. First, the slow decay of the soliton $Q(x)$ as $x\to \infty$ (see Proposition~\ref{pr:decay}) creates serious difficulties when trying to construct a relevant blow up profile (see Remark~\ref{rk:profile}), and, more technically, when estimating error terms far from the bubble. Second, an important aspect of the analysis in KdV-type equations consists in considering   localized versions of  basic quantities, such as the energy and the mass. Standard commutator estimates are not enough and suitable localization arguments were developed in this context by Kenig and Martel~\cite{KeMa} and Kenig, Martel and Robbiano~\cite{KeMaRo}. They are decisively used in the present paper (see Sect.~\ref{s.2.5}), but being by nature much more limited than the corresponding ones for the (gKdV) equation, they create error terms that are difficult to handle.
Finally, a decisive point in studying the flow of the critical (gKdV) equation around the soliton is a suitable Virial-type identity, roughly speaking a Liapounov functional on the linearized equation around $Q_{\rm KdV}$. It was first introduced by Martel and Merle~\cite{MaMejmpa} and used intensively in all subsequent works on (gKdV) mentionned above. Such a Virial identity is not available for the linearization of (mBO) around the soliton. In the present paper, to get around the lack of such Liapounov functional, we introduce a new  refined algebra related to the energy functional, extending the approach of Rapha\"el and Szeftel in \cite{RaSz} (see below for more comments). Such approach happens to be successful for the construction of the minimal mass solution, which corresponds to a precise, rigid  regime, but it should not be  sufficient to   study extensively the   blow up around the soliton, in particular the stable blow up.

\medskip

The results mentionned above on (gKdV) and Theorem~\ref{th1} are, to our knowledge, the only available rigorous results on   blow up for KdV-type or BO-type equations. Nevertheless, the history of minimal mass blow up solutions for $L^2$ critical nonlinear dispersive equations is much longer, especially for NLS-type equations. It started with the early derivation of the \emph{explicit minimal mass blow up solution}
for the mass critical nonlinear Schr\"odinger (NLS) in $\mathbb R^d$, $d\geq 1$,
\[
i\partial_t u + \Delta u + |u|^{\frac 4d} u = 0, \quad x\in \mathbb R^d,
\]
using the so-called pseudo-conformal symmetry.
Let, for $t>0$,
\begin{equation}\label{Snls}\begin{aligned}
&S_{\rm NLS}(t,x)=\frac1 {t^{\frac N2}} e^{- i\frac{|x|^2}{4t}- \frac i t} Q_\mathrm{NLS}\left(\frac x t\right), \\ &
\|S_{\rm NLS}(t)\|_{L^2}= \|Q_{\rm NLS}\|_{L^2}, \quad \|\nabla S_{\rm NLS}(t)\|_{L^2}\mathop{\sim}_{t\downarrow 0}\frac C{t},
\end{aligned}\end{equation}
where $Q_{\rm NLS}>0$ is the unique ground state of (NLS). Then $S_{\rm NLS}$ is solution of (NLS); see Weinstein~\cite{We82}, Cazenave~\cite{Ca03} for references.
Also using the  pseudo-conformal symmetry, Merle~\cite{Meduke} proved  that $S_\mathrm{NLS}$ is the unique
(up to the symmetries of the equation) minimal mass blow up solution in the energy space (see also Banica~\cite{Ba} and Hmidi and Keraani~\cite{HmKe}).
We refer to Merle and Rapha\"el~\cite{MeRagafa,MeRainvent} (and references therein) for more recent results notably on the stable ``log-log'' blow up for (NLS) equation.

For the inhomogeneous mass critical (NLS) in dimension~2,
\[
i \partial_t u+\Delta u +k(x)|u|^2 u=0,\quad k(0)=1,\quad \nabla k(0)=0,
\]
while Merle~\cite{MeIHP} derived sufficient conditions on the function $k(x)>0$ to ensure
the \emph{nonexistence} of minimal elements, Rapha\"el and Szeftel~\cite{RaSz} introduced a new approach to obtain existence and uniqueness
of a minimal blow up solution under a necessary and sufficient condition on $k(x)$, in the absence of pseudo-conformal 
transformation. For other constructions of minimal mass solutions for NLS-type equations, by various methods, we refer to~\cite{BCD,BW,CG,KLR,LeMaRa}.
In particular, Krieger, Lenzmann and Rapha\"el \cite{KLR} addressed the case of the half-wave equation in one space dimension, which also involves the nonlocal operator $D^1$, and requires the use of commutator estimates. 
However, as pointed out in \cite{MaMeRa2,CoMa}, the minimal mass blow up for KdV-type equations is specific, in some sense less compact that for NLS or wave-type  equations, and requires the use of local norms, instead of global norms. A main difficulty in the present paper is to combine such local norms and nonlocal operators. A hint of the specificity of KdV-type blow up is given by the asymptotics found in \cite{CoMa}, showing the existence of a fixed tail for $S_{\rm KdV}(t)$. See also Remark~\ref{rk:BS} for more details.

\medskip

We expect that the strategy of this paper also applies to the following family of mass critical dispersion generalized Benjamin-Ono equations,  considered e.g. by  Angulo, Bona, Linares and Scialom \cite{AnBoLiSc}
 and by Kenig, Martel and Robbiano \cite{KeMaRo},
\begin{equation}\label{gBO}
 u_t +  ( |u|^{2\alpha}   u - D^{\alpha} u )_x = 0, \quad t\in  \mathbb{R},\ 
    x\in  \mathbb{R},
\end{equation}
for $\alpha\in (1,2)$. 
Recall that \cite{KeMaRo} extends perturbatively the analysis of \cite{MaMejmpa,Mejams} to the case where the model \eqref{gBO} is close to (gKdV), i.e. for $\alpha<2$ close to $2$. In particular, blow up in finite or infinite time for negative energy solutions is obtained in this context. The main obstruction  to extend such results to $\alpha \in [1,2)$ is the absence of suitable  Virial-type identity as mentionned before. In the mass subcritical situation, recall that the asymptotic stability of the  soliton of the   Benjamin-Ono equation \eqref{BO} was proved in \cite{KeMa}, extending previous results on (gKdV) (see \cite{MaMesub} and references therein), by using a specific algebra related to the explicit form of the soliton.

\medskip 

For $\alpha \in [\frac 12,1)$, the blow up problem for \eqref{gBO} is also relevant and important in physics. In particular, the dispersion in the case $\alpha=\frac12$, for which the nonlinearity is quadratic, is somehow reminiscent  of the linear dispersion of finite depth water waves with surface tension.  The corresponding Whitham equation with surface tension writes 
$$u_t +  ( u^2 - w(D)u )_x = 0, \quad t\in  \mathbb{R},\  x\in  \mathbb{R},
$$
where $w(D)$ is  the Fourier multiplier of symbol $w(\xi)=\left( \frac{\tanh (|\xi|)}{|\xi|} \right)^{\frac12}\left(1+\tau\xi^2\right)^{\frac12}$ and $\tau$ is a positive parameter related to the surface tension. Note that for high frequencies $w(\xi) \sim |\xi|^{\frac12}$,  which corresponds to the dispersion of \eqref{gBO} in the case $\alpha=\frac12$. 
We refer to Linares,  Pilod and~Saut \cite{LiPiSa} for a detailled discussion  and some progress on the local theory for the Cauchy problem and to Klein and Saut~\cite{KlSa} for numerical simulations. Obviously, weaker dispersion can only complicate the problem and it seems quite challenging to address the full range $\alpha \in [\frac 12,1)$.

\subsection{Notation} 
For any positive  $a$ and $b$, the notation $a \lesssim b$ means that $a \le c \,b$ holds for  a   universal constant $c>0$. 
Let $(\cdot,\cdot)$ denote the scalar product on $L^2(\mathbb R)$,
\begin{displaymath} 
(f,g)=\int_{\mathbb R} f(x) g(x) dx, 
\end{displaymath}
for $f, \, g$ two real valued functions in $L^2(\mathbb R)$.
For simplicity of notation, we often write $\int$ for $\int_{\mathbb R}$ and omit $dx$.
If $I$ is an interval of $\mathbb R$, then ${\bf 1}_I$   denotes the characteristic function of $I$.

\noindent Let $\chi \in C^{\infty}(\mathbb R)$ be a cut-off function such that 
\begin{equation}\label{def:chi}
0 \le \chi \le 1, \quad \chi'\ge 0 \ \text{on} \ \mathbb R, \quad \chi_{|_{(-\infty,-2)}} \equiv 0 \quad \text{and}  \quad \chi_{|_{(-1,+\infty)}} \equiv 1 \, .
\end{equation}
Let $\Lambda$ denote the generator of the $L^2$ scaling, 
\begin{equation} \label{lambda}
\Lambda f:=\frac12f+xf' \, .
\end{equation}
Let $\mathcal{L}$ be the linearized operator around $Q$, i.e. (see also Sect.~\ref{SW})
\begin{equation} \label{L}
\mathcal{L}:=D^1+1-3Q^2 \, .
\end{equation}
We  introduce the spaces $\mathcal{Y}_\ell$, for  $\ell \in \mathbb N$,
\begin{displaymath} 
\mathcal{Y}_\ell:=\big\{f \in C^{\infty}(\mathbb R) \,  : \, \forall \, k \in \mathbb N,  
\ \forall \, x \in \mathbb R,\  |f^{(k)}(x)| \lesssim{(1+|x|)^{-(\ell+k)}} \big\} \, ,
\end{displaymath}
and the space $\mathcal{Z}$
\begin{equation*} 
\mathcal{Z}:=\big\{   f \in C^{\infty}(\mathbb R) \cap L^{\infty}(\mathbb R) \, : \, f' \in \mathcal{Y}_2 \ \text{and} \ \forall \, x >0, \   |f(x)| \lesssim{(1+|x|)^{-1}}  \big\} \, .
\end{equation*}

\subsection{Outline of the proof} The general strategy of the proof is to adapt the robust arguments developed in \cite{RaSz} 
(see also the previous papers \cite{Me} and \cite{Ma}) to construct minimal mass solutions in contexts where few is known on the flow of the equation around the soliton, i.e. in the absence of a general Virial functional for the linearized flow around the ground state. 
In particular, we combine the control of an energy-type functional, suitably localized (see below), with a mixed Morawetz-Virial functional.
The coercivity of the energy functional uses in a crucial way the complete understanding of the kernel of $\mathcal L$, proved in \cite{FrLe}. 

However, the strategy of \cite{RaSz} has to be adapted to   one of the specificities of KdV and BO-type equations which requires the use of local estimates on the residual terms, as in \cite{MaMeRa1,MaMeRa2}, and to the nonlocal nature of the operator $D^1$, which requires specific localization arguments introduced in \cite{KeMa,KeMaRo}.
The proof of Theorem~\ref{th1} thus needs the combination of all existing techniques in similar contexts,
but this is still not enough. Indeed, because of the slow decay of the ground state $Q$ ($0<Q(x)\lesssim (1+|x|^2)^{-1}$ - see Proposition \ref{pr:decay}), one cannot satisfactorily improve the ansatz to a sufficient order as was done in \cite{CoMa}. This means that the error term (denoted by  $ \varepsilon$) cannot be too small, even in local norms around the soliton (in contrast, for (NLS) type equation, it can be taken arbitrarily small in global norms  - see e.g. \cite{LeMaRa} - and in local norms for (gKdV) - see \cite{CoMa}). 
The lack of good estimates on the error terms creates important difficulties to control the cubic terms that   are usually easily controlled. One of the main novelty of this paper is to push forward the algebra of \cite{RaSz} and \cite{MaMeRa2} to cancel out these cubic terms.

\medskip

In Sect.~\ref{sect2}, we recall known facts on the ground state $Q$, we construct the blow up profile (subsequently denoted by $Q_{b}$) and we introduce a suitable decomposition of any solution around the blow up profile. In Sect.~\ref{sect3}, we introduce a particular sequence of \emph{backwards in time solutions} of~\eqref{mBO} related to the special minimal mass regime of Theorem~\ref{th1} and we claim suitable uniform bootstrap estimates of the residual term $ \varepsilon$ and on the geometrical parameters.
In Sect.~\ref{S4}, we close the estimates on $ \varepsilon$ using mainly a  localized energy functional, but also several other functionals to cancel out diverging terms. In Sect.~\ref{S5}, we close the estimates on the parameters, adjusting carefully the \emph{final data} of the sequence of solutions. In Sect.~\ref{S6}, we use the weak convergence of the flow (from \cite{CuKe}) to obtain the solution $S(t)$ of Theorem~\ref{th1} by passing to the limit in the sequence of solutions of \eqref{mBO} uniformly controled in Sect.~\ref{sect3}-Sect.~\ref{S5}.

\subsection*{Acknowledgements}
The authors would like to thank Carlos Kenig for drawing their attention to the blow up problem for (mBO) and Jean-Claude Saut for   encouraging and helpful discussions.
This material is based upon work supported by the National Science Foundation under Grant No.~0932078 000
while Y.M.~was in residence at the Mathematical Sciences Research Institute in Berkeley, California, during the Fall 2015 semester. D.P. would like to thank the \'Ecole polytechnique for the kind hospitality during the elaboration of part of this work. This work was also partially supported by CNPq/Brazil, grant 302632/2013-1  and by the project ERC 291214 BLOWDISOL.

\section{Blow up profile}\label{sect2}

\subsection{Solitary waves} \label{SW}
The existence of nontrivial solutions to \eqref{E} is well-known from the works of Weinstein \cite{We1,We2} and Albert, Bona and Saut \cite{AlBoSa}. We recall here the main results.
  
\begin{proposition}[\cite{AlBoSa,We1,We2}] \label{existenceE} 
For $u \in H^{\frac12}(\mathbb R) \setminus \{0\}$, let
\begin{displaymath} 
W(u)=\frac{\Big(\int |D^{\frac12}u|^2 \Big)\Big(\int |u|^2   \Big)}{\int |u|^4 } \, .
\end{displaymath}
There exists a solution $Q \in H^{\frac12}(\mathbb R) \cap C^{\infty}(\mathbb R)$  of \eqref{E} that solves the minimization problem
	\begin{equation} \label{WeinsteinF}
	 \inf \big\{ W(u) : u \in H^{\frac12}(\mathbb R) \setminus \{0\} \big\}= W(Q) \, .
	\end{equation}
Moreover, by translation invariance, $Q$ is chosen to be even, positive on $\mathbb R$ and satisfying $Q'<0$ on $(0,+\infty)$.
\end{proposition}

Note that for the corresponding  equation with quadratic nonlinearity associated to the Benjamin-Ono equation
\begin{equation} \label{EBO}
D^1Q+Q-Q^2=0 , 
\end{equation}
there exists an explicit solution $Q_{\rm BO}(x)=\frac{2}{1+x^2}$. By using complex analysis techniques,  Amick and Toland \cite{AmTo1}\footnote{Amick and Toland proved the following stronger statement in~\cite{AmTo2}: any nonconstant bounded solution  of \eqref{EBO} is either $Q_{\rm BO}$ (up to translation)  or a periodic wave solution.} proved that up to translation,  $Q_{\rm BO}$ is the unique solution  of \eqref{EBO} 
in $H^{\frac 12}$. These techniques do not apply to \eqref{E}. 

\medskip

More recently, Frank and Lenzmann \cite{FrLe} addressed succesfully the question of uniqueness  for \eqref{E}. Their results actually hold for a large class of nonlocal problems involving the fractional Laplacian in one dimension and are related to the well-known notion of ground state.

\begin{definition} \label{groundstate}
A positive and even solution $Q$ of \eqref{E} is called a \textit{ground state} solution of \eqref{E} if $Q$ satisfies \eqref{WeinsteinF}.
\end{definition} 

The following uniqueness result was obtained in \cite{FrLe}.
\begin{theorem}[\cite{FrLe}]
The ground state solution $Q=Q(|x|)>0$ of \eqref{E} is unique.
Moreover, every minimizer of \eqref{WeinsteinF} is of the form $\beta Q(\gamma(\cdot-x_0))$, for some $\beta \in \mathbb C$, $\beta \neq 0$, $\gamma>0$ and $x_0 \in \mathbb R$.
\end{theorem}

Recall from  \cite{AmTo1,KeMaRo} the decay properties of the ground state solution.
\begin{proposition}[\cite{AmTo1,KeMaRo}]\label{pr:decay}
The  ground state   $Q$ of \eqref{E} satisfies $Q\in \mathcal Y_2$ and $E(Q)=0$.
\end{proposition}

Note that it is easy to check that   $E(Q)=0$. Indeed,  $Q$ satisfies the energy identity 
\begin{displaymath} 
\int Q^2 +\int |D^{\frac12}Q|^2 =\int Q^4  \, 
\end{displaymath}
and the Pohozaev identity\footnote{which follows by using that $\int_{\mathbb R}\mathcal{H}(Q')xQ'dx=0$, since $\mathcal{H}(x\phi)=x\mathcal{H}\phi$ if $\int \phi dx=0$.}
\begin{displaymath}
\int Q^2  =\frac12 \int  Q^4 \, ,
\end{displaymath}
which imply that  $E(Q)=0$. For future reference, we also note that for $|a|$ small, one has
\begin{equation}\label{Ea}\begin{aligned}
E((1-a)Q)
&= \frac 12 (1-a)^2 \int |D^{\frac 12} Q|^2 - \frac 14 (1-a)^4 \int Q^4 = \frac 12 ((1-a)^2 - (1-a) ^4) \int Q^2 \\
&= a \left(1-\frac 52a+2a^2-\frac 12a^3\right) \int Q^2.
\end{aligned}
\end{equation}

\medskip
  
We will need in the sequel technical facts related to the Hilbert transform.
It is well-known that $\mathcal{H}(\frac1{1+x^2})=\frac{x}{1+x^2}$.
More generally, we have the following result.
\begin{lemma} \label{HY}
If $f \in \mathcal{Y}_2$, then $\mathcal{H}f \in \mathcal{Y}_1$.	
\end{lemma}

\begin{proof} 
Let $f \in \mathcal{Y}_2$. 
By the definition of $\mathcal H$, we have for $k\geq 0$ and some constant $c_k$,
\begin{equation} \label{HY.1}
x^{k+1}\mathcal{H}(f^{(k)})=\mathcal{H}(x^{k+1}f^{(k)})+c_k \int f.
 \end{equation}
Moreover,  from the Sobolev embedding $H^1(\mathbb R) \hookrightarrow L^{\infty}(\mathbb R)$ and the
boundeness of $\mathcal{H}$ in $H^1$,
\[
\|\mathcal{H}f^{(k)}\|_{L^\infty}
 \lesssim \|\mathcal{H} f^{(k)}\|_{H^1}
 \lesssim \|f^{(k)}\|_{H^1} \lesssim  1\, ,
\]
\[
\|\mathcal{H}\big(x^{k+1}f^{(k)} \big)\|_{L^\infty} 
 \lesssim \|\mathcal{H}\big(x^{k+1}f^{(k)}\big)\|_{H^1} 
 \lesssim \|x^{k+1}f^{(k)}\|_{H^1}\lesssim  1\, .
\]
Thus, by \eqref{HY.1}, $|\mathcal Hf^{(k)}(x)|\lesssim (1+|x|)^{-(1+k)}$.
\end{proof}

We will also need the following variant of Lemma \ref{HY}. 
\begin{lemma} \label{pw.HY}
Let $a \in C^{\infty}(\mathbb R)$ be such that $a \in L^{\infty}(\mathbb R)$ and $a', \, a'' \in L^{\infty}(\mathbb R) \cap L^2(\mathbb R)$.	
Then, 
\begin{equation} \label{pw.HY.1}
\sup_{x \in \mathbb R} \, (1+x^2)\Big|\mathcal{H}\left(\frac{a(y)}{1+y^2}\right)'(x) \Big| \lesssim \sum_{j=0}^2\|a^{(j)}\|_{L^{\infty}}+\sum_{j=1}^2\|a^{(j)}\|_{L^2} \, .
\end{equation}
\end{lemma}

\begin{proof} 
First, we see from the Sobolev embedding $H^1(\mathbb R) \hookrightarrow L^{\infty}(\mathbb R)$ and the continuity of $\mathcal{H}$ in $H^1$ that 
\[
\Big\|\mathcal{H}\left(\frac{a(y)}{1+y^2}\right)'\Big\|_{L^\infty}
\lesssim \Big\|\left(\frac{a(y)}{1+y^2}\right)'\Big\|_{H^1} \lesssim  \sum_{j=0}^2\|a^{(j)}\|_{L^{\infty}}\, .
\]
Second, we deduce from \eqref{HY.1} with $k=1$ that
\[ x^2\mathcal{H}\left(\frac{a(y)}{1+y^2}\right)'(x) 
=\mathcal{H}\left(y^2\left(\frac{a(y)}{1+y^2}\right)'\right)(x)+\int \frac{a(y)}{1+y^2} \, .
\]
Arguing as above, we have   
\[
\Big\|\mathcal{H}\left(y^2\left(\frac{a(y)}{1+y^2}\right)'\right)\Big\|_{L^\infty}
\lesssim \Big\|y^2\left(\frac{a(y)}{1+y^2}\right)'\Big\|_{H^1} \lesssim  \sum_{j=0}^1\|a^{(j)}\|_{L^{\infty}}+\sum_{j=1}^2\|a^{(j)}\|_{L^2}\, .
\]

We conclude the proof of \eqref{pw.HY.1} gathering those estimates.
\end{proof}

We recall the properties of the   operator $\mathcal{L}$ defined in \eqref{L}.

\begin{lemma}[Linearized operator, \cite{We1,We2,FrLe}]  \label{Lproperties} 
The self-adjoint operator $\mathcal{L}$ in  $L^2$ with domain $H^1$ satisfies the following properties:
\begin{itemize} 
\item[(i)] \emph{Spectrum of $\mathcal{L}$.} The operator $\mathcal{L}$ has only one negative eigenvalue $-\kappa_0$ ( $\kappa_0>0$) associated to an even, positive eigenfunction $\chi_0$; $\ker \mathcal{L}=\{aQ' : a \in \mathbb R\}$; 
and $\sigma_{ess}( \mathcal{L})=[1,+\infty)$;

\item[(ii)] \emph{Scaling.} $\mathcal{L}\Lambda Q=-Q$ and $(Q,\Lambda Q)=0$, where $\Lambda$ is defined in \eqref{lambda};

\item[(iii)] for any function $h \in L^2(\mathbb R)$ orthogonal to $Q'$ (for the $L^2$-scalar product), there exists a unique function $f \in H^1(\mathbb R)$ orthogonal to $Q'$ such that $\mathcal{L}f=h$;

\item[(iv)] \emph{Regularity.} if $f \in H^1(\mathbb R)$ is such that $\mathcal{L}f \in \mathcal{Y}_1$, then $f \in \mathcal{Y}_1$;

\item[(v)] \emph{Coercivity of $\mathcal{L}$.} there exists $\kappa>0$ such that 
for all $f \in H^{\frac12}(\mathbb R)$,
\begin{equation} \label{coercivity}
(\mathcal{L}f,f) \ge \kappa \|f\|_{H^{\frac12}}^2-\frac1\kappa\Big((f,Q)^2+(f,\Lambda Q)^2 +(f,Q')^2 \Big) \,.
\end{equation}
\end{itemize}
	
\end{lemma} 

\begin{proof} (i) The fact that $\text{ker} \, \mathcal{L}=\{aQ' : a \in \mathbb R\}$ is a quite delicate property, proved by Frank and Lenzmann, see Theorem~2.3 in~\cite{FrLe}. The other properties were proved by Weinstein, see Proposition~4 in~\cite{We2}. See also \cite{AlBoSa}.
\smallskip	
	
(ii) The assertion follows directly by differentiating the equation satisfied by $Q_{\lambda}(x) = \lambda^{-\frac 12} Q(\lambda^{-1} x)$ with respect to $\lambda$ and taking $\lambda=1$. The property $(Q,\Lambda Q)=0$  follows from $( Q_{\lambda},Q_{\lambda}) = 
(Q,Q)$.\smallskip

(iii) Let $h \in L^2(\mathbb R)$. Observe that 
\begin{equation} \label{regularity.1}
\mathcal{L}f=h \ \text{for} \ f \in H^1(\mathbb R) \quad \Leftrightarrow \quad (id-T)f=(D^1+1)^{-1}h=:\widetilde{h} \ \text{for} \ f \in L^2(\mathbb R) \, ,
\end{equation}
where $Tf=(D^1+1)^{-1}(3Q^2f)$ is a compact operator on $L^2(\mathbb R)$. From (i), if $(h,Q')=0$, then $\widetilde{h} \in \text{ker} \, (id-T^{\star})^{\perp}$. Thus, the existence part of (iii) follows from the Fredholm alternative, while the uniqueness part follows directly from (i).

\smallskip

(iv) Assume now that $h \in \mathcal{Y}_1 \subseteq H^{\infty}(\mathbb R)$ and let $f \in H^1(\mathbb R)$ be solution to $\mathcal{L}f=h$. Then, it follows from \eqref{regularity.1} that $f \in H^{\infty}(\mathbb R)$. 

To prove the decay properties of $f$, we argue as in \cite{AmTo1}. As observed by Benjamin in \cite{Ben}, if $w=w(x,y)$ is the harmonic extension of $f$ in the upper half-plane $\mathbb R^2_+=\{(x,y) \in \mathbb R^2 : y >0\}$, then $\lim_{y \to 0}\partial_y w(x,y)=-D^1f(x)$. As a consequence, if $v=v(x,y)$ is a solution to 
\begin{equation} \label{regularity.2}
\left\{\begin{array}{l}  \Delta v=0 \quad \text{in} \quad \mathbb R^2_+ \\ 
\big(v-\partial_yv-3Q^2v\big)_{|_{y=0}}=h \, , \end{array} \right.
\end{equation}
then $f(x)=v(x,0)$ satisfies $\mathcal{L}f=h$.  

Following \cite{AmTo1}, the solution $v$ to \eqref{regularity.2} is given by 
\begin{displaymath}
v(x,y)=G(\cdot,y) \ast \big(3Q^2f+h \big)(x), \quad \forall \, (x,y) \in \overline{\mathbb R^2_+}\, ,
\end{displaymath}
where the kernel $G(x,y)$ is given by 
\begin{displaymath}
G(x,y)=\int_0^{+\infty}g(x,y+w)e^{-w}dw \quad \text{and} \quad g(x,y)=\frac1{\pi} \frac{y}{x^2+y^2} \, .
\end{displaymath}
Moreover, we easily see that $G$ is positive, harmonic on $\mathbb R^2_+$ and satisfies 
\begin{equation}\label{regularity.3}
\int_{-\infty}^{+\infty}G(x,y)dx=1, \quad \forall \, y \ge 0 \, ,
\end{equation}
\begin{equation} \label{regularity.4}
G(x,y) \lesssim \frac{1+y}{x^2+y^2}, \quad \forall \, (x,y) \in \overline{\mathbb R^2_+} \, ,
\end{equation}
and for any $k\geq 0$,
\begin{equation} \label{regularity.5}
\big|\partial_x^kG(x,0)\big| \lesssim \frac1{|x|^{2+k}}, \quad \forall \, |x| \ge 1 \, .
\end{equation}

In particular a solution $f \in H^{\infty}(\mathbb R)$ to $\mathcal{L}f=h$ satisfies 
\begin{equation} \label{regularity.6}
f(x)=G(\cdot,0) \ast \big(3Q^2f+h \big)(x), \quad \forall \, x \in \mathbb R \, .
\end{equation}
Since $h \in \mathcal{Y}_1$, $Q \in \mathcal{Y}_2$ and $f \in H^1(\mathbb R)$, we get that
\begin{displaymath}
|f(x)| \le G(\cdot,0) \ast \widetilde{h}(x), \quad \text{where} \quad \widetilde{h}=3\|f\|_{L^{\infty}}Q^2+|h|  \, .
\end{displaymath}
Let $|x| \ge 1$. It follows from \eqref{regularity.5} that 
\begin{displaymath} 
\begin{split}
|f(x)| &\le \sup_{t \in \mathbb R} \big\{ (1+|t|)|\widetilde{h}(t)| \big\} \int_{-\infty}^{+\infty}\frac{G(x-t,0)}{1+|t|}dt \\ 
& \lesssim \int_{|t-x| \le \frac12|x|}\frac{G(x-t,0)}{1+|t|}dt+\int_{|t-x|\ge \frac12|x|}\frac1{(1+|t|)(x-t)^2}dt \end{split}
\end{displaymath}
By using \eqref{regularity.3}, the first integral on the right-hand side of the above expression is  bounded by $2/|x|$, while the second integral is easily bounded by $c/|x|$. Therefore, we conclude that 
\begin{equation} \label{regularity.7}
\sup_{x \in \mathbb R}\big\{(1+|x|)|f(x)| \big\} < +\infty \, .
\end{equation}

Now, we prove by induction on $k$ that 
\begin{equation} \label{regularity.8}
\sup_{x \in \mathbb R}\big\{(1+|x|)^{(k+1)}|f^{(k)}(x)| \big\} < +\infty \, ,
\end{equation}
holds for all $k \in \mathbb N$. Let $l \in \mathbb N^{\star}$. 
Assume   that \eqref{regularity.8} holds for all $k \in \{0,l-1\}$. From \eqref{regularity.6},
\[
f^{(l)}(x)=\int G^{(l)}(x-t,0) \big(3Q^2f+h \big)(t)dt  := I+II
\]
where $I$, respectively $II$, corresponds to the region $|x-t| \le \frac12|x|$, respectively $|x-t|\ge \frac12|x|$.
Let $|x| \ge 1$. Arguing as above, we deduce from \eqref{regularity.5} that 
\begin{equation} \label{regularity.9}
|II| \lesssim \int_{|x-t|\ge \frac12|x|}\frac1{(1+|t|)|x-t|^{2+l}} \lesssim \frac1{|x|^{1+l}} \, ,
\end{equation}
where the implicit constant depends on $\|f\|_{L^{\infty}}$ and $\sup_{t \in \mathbb R}\{ (1+|t|)|h(t)|\}$.
To handle $I$, we  write, for $x\geq 1$ (the case $x<-1$ is handled similarly),
\[
I=\int_{|x-t| \le \frac12|x|}\partial_x^lG(x-t,0) \big(3Q^2f\big)(t)dt+\int_{|x-t| \le \frac12|x|}\partial_x^lG(x-t,0) h(t)dt   =: I_1+I_2 \, .\]
Several integrations by parts yield 
\begin{displaymath}
\begin{split}
I_2&=\sum_{j=0}^{l-1} (-1)^{j}\Big(\partial_x^{(l-j)}G(-\frac{x}2,0)h^{(j)}(\frac{3x}2)-(\partial_x^{(l-j)}G(\frac{x}2,0)h^{(j)}(\frac{x}2) \Big)\\ &\quad+ (-1)^{l}\int_{|t-x| \le \frac12|x|}G(x-t,0)h^{(l)}(t)dt \, .
\end{split}
\end{displaymath}
Since $G(.,0)\in \mathcal Y_2$ and $h \in \mathcal{Y}_1$, we obtain
$|I_2|\lesssim 1/{|x|^{1+l}}$.
By using the same strategy, we have (the first term exists only if $l\geq 2$),
\begin{displaymath}
\begin{split}
I_1&=\sum_{j=0}^{l-2}(-1)^j\Big(\partial_x^{(l-j)}G(-\frac{x}2,0)(3Q^2f)^{(j)}(\frac{3x}2)-(\partial_x^{(l-j)}G(\frac{x}2,0)(3Q^2f)^{(j)}(\frac{x}2) \Big)\\ &\quad+ (-1)^{l+1}\int_{|t-x| \le \frac12|x|}\partial_xG(x-t,0)(3Q^2f)^{(l-1)}(t)dt \, .
\end{split}
\end{displaymath}
Observe that  
$\big| (3Q^2f)^{(j)}(t) \big| \lesssim (1+|t|)^{-(5+j)}$, for all $j=0,\cdots l-1$, 
thanks to the Leibniz rule, the induction hypothesis on $f$ and the fact that $Q \in \mathcal{Y}_2$. Hence, it follows that 
$|I_1| \lesssim 1/{|x|^{1+l}}$ and so $|I|\lesssim 1/{|x|^{1+l}}$.
From this and \eqref{regularity.9}, we obtain estimate \eqref{regularity.8} with $k=l$. This finishes the proof of (iv).

\smallskip

(v) This is a standard property obtained as a consequence of (i). 
We refer to Proposition~4 in \cite{We2}. See also the proof of Lemma~2 (ii) in \cite{MaMe}.
\end{proof}

\subsection{Definition and estimates for the localized profile}

In this subsection, we  construct an approximate profile $Q_b$. 

\begin{lemma} \label{nlprofile}
There exists a unique function $P\in \mathcal{Z}$ such that 
\begin{equation} \label{nlprofile.1}
 (\mathcal{L}P)'=\Lambda Q, \quad  (P,Q')=0,   \quad \text{and} \quad   \lim_{y \to -\infty}P(y)=\frac12\int Q \, .
\end{equation}
Moreover,
\begin{equation} \label{nlprofile.2}
p_0:=(P,Q)=\frac18\Big( \int Q \Big)^2>0 \, .
\end{equation}
\end{lemma}

\begin{proof}We look   for a solution of \eqref{nlprofile.1} of the form $P=\widetilde{P}-\int_{y}^{+\infty} \Lambda Q$. Observe that $P$ solves the equation in \eqref{nlprofile.1} if 
\begin{displaymath} 
(\mathcal{L}\widetilde{P})'=\Lambda Q+\Big(\mathcal{L}\int_{y}^{+\infty}\Lambda Q  \Big)'=: R' \quad \text{where} \quad R=-\mathcal{H}\Lambda Q-3Q^2\int_{y}^{+\infty} \Lambda Q \, .
\end{displaymath}
It follows from Lemma \ref{HY} that $\mathcal{H}\Lambda Q\in \mathcal Y_1$ and thus $R \in \mathcal{Y}_1$. 
Moreover,  Lemma \ref{Lproperties} (i) and (ii) yield
\[
(R,Q')=-\int R'Q=-\int \Lambda Q \, Q-\int \Big(\mathcal{L}\int_{y}^{+\infty}\Lambda Q  \Big)'Q=0 \, .
\]
Thus, using Lemma \ref{Lproperties} (iii) and (iv), there exists a unique $\widetilde{P} \in \mathcal{Y}_1$ orthogonal to $Q'$ such that $\mathcal{L}\widetilde{P}=R$.
Set $P=\widetilde{P}-\int_{y}^{+\infty}\Lambda Q\in \mathcal Z$. Then, $P$ satisfies \eqref{nlprofile.1} and  $(P,Q')=0$. 
We also see that $\lim_{y\to -\infty} P = -\int \Lambda Q = \frac 12 \int Q$.
Moreover, by using that $\mathcal{L}\Lambda Q=-Q$ and integrating by parts, we compute 
\begin{displaymath}
p_0:=(P,Q)=-\int \mathcal{L}P \, \Lambda Q =\int \Big( \int_{y}^{+\infty}\Lambda Q\Big) \, \Lambda Q=\frac12 \Big( \int \Lambda Q \Big)^2=\frac18 \Big( \int Q \Big)^2 \, .
\end{displaymath}
\end{proof}

Since $P$ does not belong to $L^2(\mathbb R)$, we define a suitable cut-off version of it.
Recall that $\chi$ is defined in \eqref{def:chi}.
\begin{definition} 
The localized profile $Q_b$ is defined for all $b$ by 
\begin{equation} \label{Qb} 
Q_b(y)=Q(y)+ bP_b(y) \, .
\end{equation}	
where 
\begin{equation} \label{Pb}
P_b(y)=\chi_b(y)P(y), \quad  \chi_b(y)=\chi(|b|y) \, .
\end{equation}
and $\chi$ is defined in \eqref{def:chi}.
Define
\begin{equation}\label{Rb}
R_b = \mathcal L P_b - P_b = D^1 P_b - 3 Q^2 P_b .
\end{equation}
\end{definition}

\begin{lemma}[Estimates on the localized profile]
\label{lprofile}
For $|b|$ small, the following properties hold.
\begin{itemize}
\item[(i)] \emph{Pointwise estimate for $Q_b$.} For  all $y\in \mathbb R$,
\begin{equation} \label{lprofile.1}
|Q_b(y)| \lesssim \frac1{(1+|y|)^2} + |b| {\bf1}_{[-2,0]}(|b|y)+|b|{\bf1}_{[0,+\infty)}(y)\frac1{1+|y|}  \,,
\end{equation}
\begin{equation} \label{lprofile.100}
|Q_b'(y)| \lesssim \frac1{(1+|y|)^3}+|b|\frac1{(1+|y|)^2} {\bf1}_{[-2,+\infty)}(|b|y)+|b|^2 {\bf1}_{[-2,-1]}(|b|y)  \, ,
\end{equation}
%and, for all $k \in \mathbb N \cap [2,+\infty)$,
%\begin{equation} \label{lprofile.101}
%|Q_b^{(k)}(y)| \lesssim \frac1{(1+|y|)^{2+k}}+|b|\frac1{(1+|y|)^{1+k}}+
%|b|^{k}{\bf1}_{[-2,-1]}(|b|y)  \, .
%\end{equation}}
\begin{equation}\label{bd.PbRb}
\|P_b\|_{L^2}\lesssim |b|^{-\frac 12},\quad \|D^{\frac12} P_b\|_{L^2}\lesssim |\ln|b||^{\frac 12},\quad
\|R_b\|_{L^2} \lesssim 1\, ,
\end{equation}
\begin{equation} \label{yPbprime}
\|D^{\frac12}(yP_b')\|_{L^2} \lesssim 1 \, ,
\end{equation}
and 
\begin{equation} \label{dPb_db}
\Big\|D^{\frac12}\left(\frac{\partial P_b}{\partial b} \right)\Big\|_{L^2} \lesssim |b|^{-1} \, .
\end{equation}
\item[(ii)]\emph{Estimate for the equation of $Q_b$.} 
Let $\Psi_b$ be defined by 
\begin{equation} \label{Psib}
\Psi_b:=\big(\mathcal{H}Q_b'+Q_b-Q_b^3\big)'-b\Lambda Q_b+b^2\frac{\partial Q_b}{\partial b} \, .
\end{equation} 
Then
\begin{equation} \label{lprofile.2}
\|\Psi_b\|_{L^2}\lesssim |b|^{\frac 32},\quad 
\|\Psi_b-bP\chi_b'-\frac{b^2}2P_b\|_{L^2} \lesssim |b|^2 \, ,
\end{equation}
and
\begin{equation} \label{lprofile.3}
\|D^{\frac12}\Psi_b\|_{L^2} \lesssim  |b|^2|\ln |b||^{\frac12} \, .
\end{equation}
\item[(iii)]\emph{Projection of $\Psi_b$ in the direction $Q$.}
\begin{equation} \label{lprofile.4}
\big|(\Psi_b,Q)\big| \lesssim |b|^3  \, .
\end{equation}
\item[(iv)]\emph{Mass and energy for $Q_b$.}
\begin{equation} \label{lprofile.5}
\Big|\int Q_b^2-  \int Q^2   \Big| \lesssim  |b| \, ,
\end{equation}
\begin{equation} \label{lprofile.6}
\Big|E(Q_b)+  p_0 b \Big| \lesssim  |b|^{2} |\ln|b||\, .
\end{equation}
\end{itemize}
\end{lemma}
\begin{remark}\label{rk:profile}
We see that $P$ defined above decays as $1/y$ as $y\to +\infty$. This is still acceptable for our needs in the next sections.
However, if one tries to improve the ansatz $Q_{b}$ at a higher order in $b$ (as was done at any order in \cite{CoMa} for (gKdV), one faces non $L^2$ functions for $y>0$. This is an important difficulty. In this paper, we have chosen to restrict ourselves to the above ansatz $Q_{b}$, at the cost of a relatively large error terms (see \eqref{bootstrap.4}).
\end{remark}
\begin{remark}\label{rk:law}
On the definition of $\Psi_b$ in \eqref{Psib}, we anticipate the blow up relations $\frac {\lambda_s}{\lambda} \sim -b$ and  $b_s + b^2 \sim 0$, which will eventually lead to the $\frac 1 t$ blow up rate. See Lemma \ref{modulation}.
\end{remark}

\begin{proof}[Proof of Lemma \ref{lprofile}] (i) The proof of \eqref{lprofile.1},  \eqref{lprofile.100}   follows directly from $Q \in \mathcal{Y}_2$, $P \in \mathcal{Z}$ and the definition of $Q_b$ in \eqref{Qb}. We see   
using $P \in \mathcal{Z}$ that
\begin{equation} \label{lprofile.208}
\|P_b\|_{L^2}^2 \lesssim   \int_{y<0} \chi(|b|y)^2dy+\int_{y>0} \frac{dy}{1+y^2}  \lesssim |b|^{-1}\,.\end{equation}
Next, we split $\|D^{\frac 12}P_b\|_{L^2}$ as follows
\[
\|D^{\frac 12}P_b\|_{L^2}^2 = (D^1P_b,P_b)=(D^1P,P_b)-(D^1((1-\chi_b)P),P_b).
\]
Since $P\in \mathcal Z$ and $P'\in \mathcal Y_2$, we have by Lemma \ref{HY}, $D^1 P=\mathcal HP'\in \mathcal Y_1$ and thus
\[
|(D^1P,P_b)|\lesssim \int_{-2|b|^{-1}<y<0} \frac{dy}{1+|y|} + \int_{y>0} \frac {dy}{(1+|y|)^2} \lesssim |\ln|b||.
\]
Moreover,
\[
|(D^1((1-\chi_b)P),P_b)| \lesssim \|P_b\|_{L^2} \|(1-\chi_b)P\|_{\dot H^1}
\lesssim |b|^{-\frac 12} \left(\|(1-\chi_b)P'\|_{L^2} + \|\chi_b'P\|_{L^2}\right)\lesssim 1
\,.\]
This implies $\|D^{\frac12} P_b\|_{L^2}\lesssim |\ln|b||^{\frac 12}$.
Concerning $R_b$ defined in \eqref{Rb}, we have
\[
\|D^1P_b\|_{L^2}^2 = \|P_b'\|_{L^2}^2 \lesssim 1 \quad \hbox{and}\quad   \|Q^2 P_b\|_{L^2}\lesssim 1.
\]
These estimates prove \eqref{bd.PbRb}.

We proceed similarly to prove \eqref{yPbprime} and decompose $\|D^{\frac12}(yP_b')\|_{L^2}$ as follows
\[ 
\|D^{\frac12}(yP_b')\|_{L^2}^2=\int yP_b'\mathcal{H}(yP_b')'=\int yP_b'\mathcal{H}(yP')'
+\int yP_b'\mathcal{H}(y(P(\chi_b-1))')' \, .
\]
Since $(yP')'\in \mathcal{Y}_2$, we have that $\mathcal{H}(yP')'\in \mathcal Y_1$ by using Lemma \ref{HY} and thus
\[ 
\left| \int yP_b'\mathcal{H}(yP')' \right| \lesssim \int P_b' \frac{|y|}{1+|y|} \lesssim 1 \, .
\]
Moreover, since $P'\in \mathcal Y_2$,
\[ 
\left| \int yP_b'\mathcal{H}(yP(\chi_b-1)')' \right| \lesssim \|yP_b'\|_{L^2}\|(yP(\chi_b-1)')'\|_{L^2}
\lesssim |b|^{-\frac12}|b|^{\frac12} \lesssim 1 \, ,
\]
which finishes the proof \eqref{yPbprime}.

 To prove \eqref{dPb_db}, we observe that
$\frac{\partial P_b}{\partial b}=\text{sgn}(b)y\chi'(|b|y)P$. Thus, it follows by interpolation that 
\begin{displaymath} 
\Big\|D^{\frac12}\left(\frac{\partial P_b}{\partial b} \right)\Big\|_{L^2} 
\le \|y\chi'(|b|y)\|_{L^2}^{\frac12}\big\|\big(y\chi'(|b|y)\big)'\big\|_{L^2}^{\frac12} \lesssim |b|^{-1} \, ,
\end{displaymath}
which is the desired estimate. 

\smallskip	

(ii) Expanding the expression \eqref{Qb} of $Q_b$ in the definition of $\Psi_b$ and using \eqref{E}, we find that 
\begin{equation} \label{lprofile.7}
\Psi_b=b\Psi_1+b^2\Psi_2+b^3\Psi_3 \, ,
\end{equation}
where
\begin{displaymath} 
\Psi_1:=(\mathcal{L}P_b)'-\Lambda Q \, ,
\end{displaymath}

\begin{displaymath}
\Psi_2:=-3(QP_b^2)'-\Lambda(P_b)+\frac{\partial{Q}_b}{\partial b} \, ,
\quad \text{and} \quad
\Psi_3=-(P_b^3)' \, . 
\end{displaymath}

First, we prove \eqref{lprofile.2}. By  \eqref{nlprofile.1}, we rewrite $\Psi_1$ as 
\begin{displaymath}
\begin{split}
\Psi_1&=(\mathcal{L}P_b)'-(\mathcal{L}P)' \\ 
&=\mathcal{H}\big((P(\chi_b-1)\big)''+\big(P(\chi_b-1)\big)'-3\big(Q^2P(\chi_b-1)\big)'  \\ 
&=\mathcal{H}\big(P''(\chi_b-1) \big)+2\mathcal{H}\big(P'\chi_b' \big)+\mathcal{H}\big(P\chi_b''\big)+P'(\chi_b-1)+P\chi_b'-3\big(Q^2P(\chi_b-1)\big)'
\end{split}
\end{displaymath}
Since $\mathcal{H}$ is bounded in $L^2$, $P \in L^{\infty}(\mathbb R)$ and $P' \in \mathcal{Y}_2$, we obtain
\begin{displaymath} %\label{lprofile.200}
\|b\mathcal{H}\big(P''(\chi_b-1) \big)\|_{L^2} \lesssim |b|\Big( \int_{y \le -|b|^{-1}}\frac1{(1+|y|)^6} \, dy\Big)^{\frac12} \lesssim |b|^{\frac72} \, ,
\end{displaymath}
\begin{displaymath} %\label{lprofile.201}
\|b\mathcal{H}(\chi_b'P')\|_{L^2}=|b| \|\chi_b'P'\|_{L^2}	\lesssim |b|^2\Big( \int_{y \le -|b|^{-1}}\frac1{(1+|y|)^4}  \, dy\Big)^{\frac12} \lesssim |b|^{\frac72} \,  
\end{displaymath}
and 
\begin{displaymath} %\label{lprofile.202}
\|b\mathcal{H}(\chi_b''P)\|_{L^2}	\lesssim  |b|^3\Big( \int \chi''(|b|y)^2 \, dy\Big)^{\frac12} \lesssim  |b|^{\frac52} \, .
\end{displaymath}
Similarly, 
\begin{displaymath} %\label{lprofile.203}
\|bP'(\chi_b-1) \|_{L^2} \lesssim |b|\Big( \int_{y \le -|b|^{-1}}\frac1{(1+|y|)^4} \, dy\Big)^{\frac12} \lesssim |b|^{\frac52} \, ,
\end{displaymath}
and
\begin{displaymath} %\label{lprofile.204}
\|b\big(\chi_b'Q^2P\big)\|_{L^2} \lesssim |b|^2\Big( \int_{y \le -|b|^{-1}}\frac1{(1+|y|)^8}  \, dy\Big)^{\frac12} \lesssim |b|^{\frac{11}2}\, .
\end{displaymath}
Therefore, we conclude gathering those estimates that 
\begin{equation} \label{lprofile.205}
\|b\big(\Psi_1-\mathcal{H}(\chi_b''P)-P'(\chi_b-1)-P\chi_b')\|_{L^2} \lesssim |b|^{\frac72}
\quad \hbox{and}\quad
\|b\big(\Psi_1-P\chi_b')\|_{L^2} \lesssim |b|^{\frac52} \, .
\end{equation}
Also, note that\begin{equation} \label{lprofile.206}
\|bP\chi_b'\|_{L^2}  \lesssim  |b|^2\Big( \int |\chi'(|b|y)|^2 \, dy\Big)^{\frac12}  \lesssim |b|^{\frac32} \, .
\end{equation}
Now we focus on $b^2\Psi_2$ and $b^3\Psi_3$. One sees  easily that 
\begin{displaymath} %\label{lprofile.207}
\|b^2(Q\chi_b^2P^2)'\|_{L^2}+\|b^3(\chi_b^3P^3)'\|_{L^2} \lesssim |b|^2 \, .
\end{displaymath}
To deal with the remaining terms in $\Psi_2$, we observe that 
\begin{displaymath} 
\Lambda(P_b)=\frac12 P_b+yP'\chi_b +yP\chi_b' 
\end{displaymath}
and 
\begin{equation}\label{def.DQb}
\frac{\partial Q_b}{\partial b}=P_b+|b|y\chi'(|b|y)P=P_b+yP\chi_b' \, ,
\end{equation}
so that 
\begin{equation} \label{lprofile.207}
-\Lambda(P_b)+\frac{\partial Q_b}{\partial b}=\frac12P_b-yP'\chi_b \, .
\end{equation}
By using  $P \in \mathcal{Z}$,
\begin{displaymath} 
\|y\chi_bP'\|_{L^2} \lesssim \Big(\int \chi(|b|y)^2\frac{y^2}{1+y^4}dy\Big)^{\frac12} \lesssim 1 \, . 
\end{displaymath}
Hence
\begin{equation} \label{lprofile.209}
\|b^2(\Psi_2-\frac12P_b)\|_{L^2} + \|b^3\Psi_3\|_{L^2}	\lesssim |b|^2 \, . 
\end{equation}
Therefore, we conclude estimate \eqref{lprofile.2}  by gathering \eqref{lprofile.7}, \eqref{lprofile.205}, \eqref{lprofile.207} and \eqref{lprofile.209}. 

\medskip

 To prove estimate \eqref{lprofile.3}, we estimate each term on the right-hand side of \eqref{lprofile.7} in $\dot{H}^{\frac12}$.  For the sake of simplicity, we only explain how to deal with the terms corresponding to \eqref{lprofile.206} and \eqref{lprofile.207}, which are the most problematic ones. First, observe that 
\begin{displaymath} 
\big\| b(\chi_b'P)' \big\|_{L^2} \lesssim |b|\|\chi_b''P\|_{L^2}+|b|\|\chi_b'P'\|_{L^2} \lesssim  |b|^{\frac52} \, ,
\end{displaymath}
which gives after interpolation with \eqref{lprofile.206}
\begin{displaymath} 
|b|\big\| D^{\frac12}(\chi_b'P) \big\|_{L^2} \lesssim  |b|^2 \, .
\end{displaymath}
Moreover, we have from \eqref{bd.PbRb} that
\begin{displaymath} \label{lprofile.302}
|b|^2\big\| D^{\frac12}(P_b) \big\|_{L^2} \lesssim  |b|^2|\ln |b||^{\frac12}\, . 
\end{displaymath}
This last estimate gives the bound in \eqref{lprofile.3}.

\medskip
(iii) We take the scalar product of \eqref{lprofile.7} with $Q$. 
First,  for $\Psi_1$, we use \eqref{lprofile.205}. Since $\mathcal H Q \in \mathcal Y_1$ by Lemma \ref{HY},
\[
\big|\big(\mathcal{H}(\chi_b''P),Q\big)\big|=\big|\big( \chi_b''P,\mathcal HQ\big)\big|
\lesssim b^2 \int_{-2|b|^{-1}<y<-|b|^{-1}} \frac {dy}{1+|y|}\lesssim b^2.
\]
Moreover, since $P',Q\in \mathcal Y_2$,
\[\big|\big(P'(\chi_b-1),Q\big)\big|
+\big|\big( P\chi_b', Q\big) \big| \lesssim |b|^2\, ,\]
Thus, it follows from \eqref{lprofile.205} that 
\begin{equation} \label{lprofile.400}
\big|\big( b\Psi_1, Q\big) \big| \lesssim |b|^{3} \, .
\end{equation}
Next, we see using the computations in \eqref{lprofile.207} that 
\begin{equation} \label{lprofile.401}
\begin{split}
(\Psi_2,Q)&=\big(-3(QP^2)'-\Lambda P+P,Q \big)
-3\big((QP^2(\chi_b-1))',Q \big)\\ & \quad
+\frac12\big(P(\chi_b-1),Q \big)
-\big(yP'(\chi_b-1),Q \big) \, .
\end{split}
\end{equation}
By taking the scalar product of the equation in \eqref{nlprofile.1} with $P$, we get that 
\begin{displaymath}
(P',P)-3\big((Q^2P)',P \big)=-(\Lambda P,Q) \, .
\end{displaymath}
On the one hand, it follows from \eqref{nlprofile.2} that 
\begin{displaymath}
(P,P')=-\frac12\lim_{y \to -\infty}P(y)^2=-(P,Q) \, .
\end{displaymath}
On the other hand,  integration by parts yields 
\begin{displaymath}
-3\big((Q^2P)',P \big)=3((P^2Q)',Q) \, .
\end{displaymath}
Hence, we deduce combining those identities that 
\begin{equation} \label{lprofile.402}
\big(-3(P^2Q)'+P-\Lambda P,Q \big)=0 \, ,
\end{equation}
so that the first term on the right-hand side of \eqref{lprofile.401} cancels out. We estimate the other ones as follows: 
\begin{displaymath}
\big|\big((QP^2(\chi_b-1))',Q \big) \big| \lesssim \int_{y \le |b|^{-1}} \frac1{(1+|y|)^5}dy \lesssim |b|^4 \, ,
\end{displaymath}
\begin{displaymath}
\big|\big(P(\chi_b-1),Q \big) \big| \lesssim \int_{y \le |b|^{-1}} \frac1{(1+|y|)^2}dy \lesssim |b| \, ,
\end{displaymath}
and 
\begin{displaymath}
\big|\big(yP'(\chi_b-1),Q \big) \big| \lesssim \int_{y \le |b|^{-1}} \frac1{(1+|y|)^3}dy \lesssim |b|^2 \, .
\end{displaymath}
This implies that 
$\big|\big( b^2\Psi_2,Q \big)\big| \lesssim |b|^3$. Finally, we see easily that 
$\big|\big( b^3\Psi_3,Q \big)\big| \lesssim |b|^3$.
This finishes the proof of estimate \eqref{lprofile.4}.

\medskip
(iv) By using the definition of $Q_b$ in \eqref{Qb}, we compute
\begin{displaymath} 
\int Q_b^2=\int Q^2+2b(Q,\chi_bP)+b^2\int \chi_b^2P^2 \, .
\end{displaymath}
Moreover, observe that $\big|b(Q,\chi_bP) \big|  \lesssim |b|$
and $\Big|b^2\int \chi_b^2 P^2 \Big| \lesssim |b|$. This finishes the proof of estimate \eqref{lprofile.5}.

\medskip
 
Expanding $Q_b=Q+bP_b$ in the definition of $E$ in \eqref{mass}, we see that 
\begin{displaymath} 
E(Q_b)=E(Q)+b\int  P_b \big(D^1Q-Q^3\big)+\frac{b^2}2\int|D^{\frac12}(P_b)|^2+\mathcal{O}(b^2) \, .
\end{displaymath}
From $E(Q)=0$,  \eqref{E} and $p_0=(P,Q)$, it follows that
\begin{equation} \label{lprofile.601}
\Big|E(Q_b)+p_0 b \Big| \lesssim |b|\Big| \int PQ(\chi_b-1) \Big|+|b|^2\int|D^{\frac12} P_b|^2+\mathcal{O}(b^2) \,.
\end{equation}
By  $| \int PQ(\chi_b-1)|\lesssim \int_{y<-|b|^{-1}} Q(y)dy  \lesssim |b|$ and \eqref{bd.PbRb}, 
we  finish the proof of  \eqref{lprofile.6}.
\end{proof}

\subsection{Decomposition of the solution using refined profiles}
\begin{lemma}[First modulation around $Q$]\label{modulation.le1}
There exists $\delta_0>0$ such that if $v\in H^{\frac12}$ satisfies
\begin{equation} \label{tube.1}
\inf_{\lambda_1>0,\, x_1 \in \mathbb R} \|\lambda_1^{\frac12}v(\lambda_1\cdot+x_1)-Q\|_{H^{\frac12}} =\delta <\delta_0 \, ,
\end{equation}
then there exist unique $(\lambda_v,x_v)\in (0,+\infty) \times \mathbb R$ such that  the function $\eta_v$ defined by
\begin{equation} \label{modulation.1.1}
\eta_v(y)=\lambda_v^{\frac12}v(\lambda_v y+x_v)-Q(y)
\end{equation}
satisfies  
\begin{equation} \label{modulation.1.2} 
(\eta_v, Q')=(\eta_v,\Lambda Q)=0 
,\quad \|\eta_v\|_{H^{\frac12}} \lesssim\delta \,.
\end{equation}
\end{lemma}
\begin{proof}
The existence (and uniqueness) of $(\lambda_v,x_v)$ such that $\eta_v$ defined in \eqref{modulation.1.1} satisfies \eqref{modulation.1.2}  follows from standard arguments,   based on the implicit function theorem. 
We refer for example to Proposition 1 in \cite{MaMegafa}.
The key point of the proof is the non-degeneracy of the Jacobian matrix:
\[
\left|
  \begin{array}{ccc}
    ({\Lambda} Q,{\Lambda} Q) &  ({\Lambda} Q,Q')\\ 
 ( Q',{\Lambda} Q)  & ( Q',Q')\\ 
  \end{array}\right|=({\Lambda} Q,{\Lambda} Q)   ( Q',Q')\not= 0.
\]
\end{proof}

\begin{lemma}[Refined modulated flow] \label{modulation}
 Let $u$ be a solution to \eqref{mBO} on a time interval $\mathcal{I}$ such that
\begin{equation} \label{tube.2}
\sup_{t \in \mathcal{I}} \inf_{\lambda_1>0,\,x_1 \in \mathbb R} \|\lambda_1^{\frac12}u(t,\lambda_1\cdot+x_1)-Q\|_{H^{\frac12}}=\delta <\delta_0 \, .
\end{equation}
For $t\in \mathcal I$, let $\lambda(t):=\lambda_{u(t)}$ and $x(t):=x_{u(t)}$ be given by Lemma \ref{modulation.le1} and set
\begin{equation} \label{modulation.1}
\varepsilon(t,y)=\lambda^{\frac12}(t)u(t,\lambda(t)y+x(t))-Q_{b(t)}(y), \quad \mbox{where}\quad b(t) = - \frac{E(u_0)}{p_0} \lambda(t)
\end{equation}
and $Q_b$ is defined in \eqref{Qb}.
Then, there exists $\delta_0>0$ such that the following holds.
\begin{itemize} 
\item[(i)] \emph{Almost orthogonalities and smallness.} On $\mathcal{I}$, it holds
\begin{equation} \label{modulation.3}
\|\varepsilon\|_{H^{\frac12}}^2+|b| \lesssim \delta \,.
\end{equation}
\begin{equation} \label{modulation.2} 
|(\varepsilon, Q')|+|(\varepsilon,\Lambda Q)|\lesssim |b|.
\end{equation}
\item[(ii)] \emph{Conservation laws.} On $\mathcal{I}$, it holds
\begin{equation} \label{modulation.5} 
\|\varepsilon\|_{H^{\frac 12}}^2   \lesssim  |b|+   \int u_0^2 - \int Q^2   \, ,
\end{equation}
\begin{equation} \label{modulation.6} 
\Big|- (\varepsilon, Q) +  b( \varepsilon,R_b) +  \frac 12 \int \Big[|D^{\frac12}\varepsilon|^2 -\frac 12 \Big( (Q_b+\varepsilon)^4-Q_b^4-4Q_b^3 \varepsilon\Big) \Big]\Big|  \lesssim |b|^{2} |\ln|b||,
\end{equation}
where $R_b$ is defined in \eqref{Rb}, and
\begin{equation}\label{modulation.6bis}
|(\varepsilon,Q)|\lesssim |b|^2|\ln |b|| + |b| \|\varepsilon\|_{L^2}  + \|\varepsilon\|_{\dot H^{\frac 12}}^2+\int \varepsilon^2 Q^2.
\end{equation}
\item[(iii)] \emph{Equation of $\varepsilon$.} 
The function $(\lambda,x,b) : \mathcal{I} \to (0,+\infty) \times \mathbb R^2$ is of class $\mathcal{C}^1$. 
For  $t_0 \in \mathcal{I}$ and $s_0 \in \mathbb R$, define the rescaled time variable $s$ by
\[s(t)=s_0+\int_{t_0}^t \frac{dt'}{\lambda^2(t')} \quad\hbox{and}\quad  \mathcal{J}=s(\mathcal{I}).\]
Then, on $\mathcal J$, the function $\varepsilon(s,y)$ is  solution of
\begin{equation} \label{modulation.4}
\begin{aligned}
&\varepsilon_s-\Big( D^1 \varepsilon+\varepsilon-\big((Q_b+\varepsilon)^3-Q_b^3\big) \Big)_y
\\  & \quad =\frac{\lambda_s}{\lambda}\Lambda \varepsilon +(\frac{\lambda_s}{\lambda}+b)\Lambda Q_b+(\frac{x_s}{\lambda}-1)(Q_b+\varepsilon)_y-(b_s+b^2)\frac{\partial Q_b}{\partial b}+\Psi_b \, ,
\end{aligned}
\end{equation}
where $\Psi_b$ is defined in \eqref{Psib}.
\item[(iv)] \emph{Modulation equations.} On $\mathcal J$, it holds
\begin{equation} \label{modulation.7}
\Big| \frac{x_s}{\lambda}-1\Big|+
\Big| \frac{\lambda_s}{\lambda}+b\Big| \lesssim  |b|^2+\Big(\int \frac{\varepsilon^2}{1+y^2} \Big)^{\frac12}  .
\end{equation}
\end{itemize}
\end{lemma}

\begin{remark}
The choice  $b = - \frac{E(u_0)}{p_0} \lambda$ in \eqref{modulation.1} is not really standard. Usually, for simplicity, $b(t)$ is tuned so that $(\varepsilon,Q)=0$ (see e.g. \cite{MaMeRa1}). The choice \eqref{modulation.1} leads to the relation \eqref{modulation.6}, which is a sufficient substitute to the exact orthogonality $(\varepsilon,Q)=0$.
In Sect.~\ref{S4}, such a sharp choice related to energy will be technically important in our proof. 
\end{remark}

\begin{proof} (i)
For all $t\in \mathcal I$, let $\lambda(t):=\lambda_{u(t)}$ and $x(t):=x_{u(t)}$ be given by Lemma \ref{modulation.le1} such that
$
\eta(t,y)=\lambda(t)^{\frac12}u(t,\lambda(t) y+x(t))-Q(y)$ satisfies
$
(\eta(t), Q')=(\eta(t),\Lambda Q)=0$ and $\|\eta(t)\|_{H^{\frac 12}}\lesssim \delta$.
Then,  since $E(Q)=0$, $\lambda E(u)=\lambda E(u_0) = E(Q+\eta)=\mathcal O(\delta)$. It follows that $b(t)$ defined
by $b = - ({E(u_0)}/{p_0}) \lambda$ satisfies $|b(t)|\lesssim \delta$. Moreover, $\varepsilon(t)$ defined by \eqref{modulation.1} satisfies $\varepsilon = -b P_b + \eta$, where $P_b$ is defined in \eqref{Pb}.
Since $\|P_b\|_{L^2}\lesssim |b|^{-\frac 12}$ and $\|P_b\|_{\dot H^{\frac 12}}\lesssim 1$, we obtain
$\|\varepsilon\|_{ H^{\frac 12}}\lesssim |b|^{\frac 12}+\delta\lesssim \delta^{\frac 12}$.
Since $|(P_b,\Lambda Q)|+|(P_b,Q')|\lesssim 1$, estimates \eqref{modulation.2} are consequences of $
(\eta(t), Q')=(\eta(t),\Lambda Q)=0$.

\medskip

(ii) 
By using the conservation of the energy $E$ defined in \eqref{mass}, we get from \eqref{lprofile.6} and the equation of $Q$  that
\begin{align*}
\lambda E(u_0)&=E(Q_b+\varepsilon)=E(Q_b)+\int (D^1\varepsilon)  Q_b+\frac 12 \int |D^{\frac12}\varepsilon|^2-\frac14 \int \Big((Q_b+\varepsilon)^4-Q_b^4 \Big)\\ 
&=-p_0 b +\int \varepsilon (D^1 Q_b -Q_b^3) + \frac 12\int |D^{\frac12}\varepsilon|^2-\frac14\int \Big((Q_b+\varepsilon)^4-Q_b^4-4Q_b^3\varepsilon\Big)\\ &\quad +\mathcal O(|b|^2|\ln|b||) \, .
\end{align*}
We compute ($R_b$ being defined in \eqref{Rb})
\[
D^1 Q_b - Q_b^3  = D^1 Q - Q^3 + b (D^1 P_b -3 Q^2 P_b) -3 b^2 Q P_b^2 - b^3 P_b^3
= -Q + b R_b + \mathcal O_{L^2}(|b|^2).
\] 
Thus, by the choice $\lambda E(u_0) = -p_0 b$ and $\|\varepsilon\|_{L^2}\lesssim 1$, one obtains
\begin{equation}\label{energie.1}
 \int \Big[|D^{\frac12}\varepsilon|^2-\frac12 \Big(Q_b+\varepsilon)^4-Q_b^4-4Q_b^3\varepsilon\Big)\Big]
 -2(\varepsilon,Q)+2b(\varepsilon,R_b)=\mathcal O(|b|^2|\ln|b||),
\end{equation}
which is \eqref{modulation.6}.
In particular, by \eqref{bd.PbRb} and \eqref{gnp},
$\int \varepsilon^4 \lesssim \|\varepsilon\|_{H^{\frac 12}}^4$, 
$\int |Q_b| |\varepsilon|^3 \lesssim \|\varepsilon\|_{H^{\frac 12}}^3$,
$\int |Q_b^2-Q^2| \varepsilon^2 \lesssim |b| \int \varepsilon^2$, and thus,
\begin{equation}\label{energie.2}
|(\varepsilon,Q)|\lesssim |b|^2 |\ln|b|| + |b| \|\varepsilon\|_{L^2}   + \|\varepsilon\|_{\dot H^{\frac 12}}^2
+\int \varepsilon^2 Q^2 \,.
\end{equation}

\smallskip

By using the $L^2$ conservation, \eqref{lprofile.5} and \eqref{bd.PbRb}, we get (using the notation in \eqref{ulambda}),
\begin{align*} 
\|u_0\|_{L^2}^2
&=\|u_{\lambda^{-1}}(\cdot+x(t),t)\|_{L^2}^2=\|Q_b\|_{L^2}^2+2(\varepsilon,Q_b)+\|\varepsilon\|_{L^2}^2\\
&=\|Q\|_{L^2}^2+2(\varepsilon,Q)+\|\varepsilon\|_{L^2}^2+\mathcal O(|b|+|b|^{\frac 12}\|\varepsilon\|_{L^2}).
\end{align*}
Summing this identity with \eqref{energie.1}, we obtain
\[
\int \Big[|D^{\frac12}\varepsilon|^2+|\varepsilon|^2-\frac12 \Big((Q_b+\varepsilon)^4-Q_b^4-4Q_b^3\varepsilon\Big)\Big]
=\mathcal O(|b|+|b|^{\frac 12}\|\varepsilon\|_{L^2})+\|u_0\|_{L^2}^2-\|Q\|_{L^2}^2\,.
\]
Since by \eqref{gnp}, $\int \varepsilon^4 \lesssim \|\varepsilon\|_{H^{\frac 12}}^4$, 
$\int |Q_b| |\varepsilon|^3 \lesssim \|\varepsilon\|_{H^{\frac 12}}^3$ and
$\int |Q_b^2-Q^2| \varepsilon^2 \lesssim |b| \int \varepsilon^2$,
we obtain
\[
(\mathcal L \varepsilon,\varepsilon) = \mathcal O(|b|+|b|^{\frac 12}\|\varepsilon\|_{L^2})+\|u_0\|_{L^2}^2-\|Q\|_{L^2}^2 + \mathcal O(\|\varepsilon\|_{H^{\frac 12}}^3).
\]
By the coercivity of $\mathcal L$ (see \eqref{coercivity}) and \eqref{modulation.2}, \eqref{energie.2},  for $\delta$ small enough, we obtain
\[
\|\varepsilon\|_{H^{\frac 12}}^2 \lesssim |b|+\|u_0\|_{L^2}^2-\|Q\|_{L^2}^2.
\]

(iii) The $\mathcal C^1$ regularity of $\lambda(t)$ and $x(t)$ is a standard fact, see e.g.~Proposition~1 in \cite{MaMe}.
Equation \eqref{modulation.4} follows directly by writing that $u$ under the decomposition \eqref{modulation.1} solves \eqref{mBO}
and by using the definition of $\Psi_b$ in \eqref{Psib}. An intermediate step in the computations is the derivation of the following equation for $\eta(s,y)$, which is more handy to derive later the estimates for $x_s$ and $\lambda_s$
\begin{equation} \label{modulation.0}
 \eta_s-\Big( D^1 \eta+\eta-\big((Q+\eta)^3-Q^3\big) \Big)_y
   =  \frac{\lambda_s}{\lambda} \Lambda (Q+\eta)+(\frac{x_s}{\lambda}-1)(Q+\eta)_y. \end{equation}
   
\medskip

(iv) Differentiating the  orthogonality relation $(\eta,Q')=0$  with respect to $s$, using \eqref{modulation.0}, 
$(\Lambda Q,Q')=0$ and the decay properties of $Q$, we
  see that
\begin{equation}\label{ortho.1}
\left| \left(\frac {x_s}\lambda- 1\right) - \frac 1{\|Q'\|_{L^2}^2} \int \eta \mathcal L(Q'') \right|   \lesssim \left(\int \frac {\eta^2}{1+y^2}\right)^{\frac 12} \left(\left| \frac{\lambda_s}{\lambda}\right| + \left| \frac {x_s}\lambda- 1\right|\right)+
 \int \frac {\eta^2+|\eta|^3}{1+y^2}  .
\end{equation}
Similarly, using $(\eta,\Lambda Q)=0$, we have
\begin{equation}\label{ortho.2}
\left| \frac {\lambda_s}\lambda - \frac 1{\|\Lambda Q\|_{L^2}^2} \int \eta \mathcal L\big((\Lambda Q)'\big) \right| \lesssim \left(\int \frac {\eta^2}{1+y^2}\right)^{\frac 12} \left(\left| \frac{\lambda_s}{\lambda}\right| + \left| \frac {x_s}\lambda- 1\right|\right)+\int \frac {\eta^2+|\eta|^3}{1+y^2}.
\end{equation}
Combining \eqref{ortho.1} and \eqref{ortho.2}, we obtain
\begin{equation}\label{ortho.3}
 \left| \left(\frac {x_s}\lambda- 1\right) - \frac 1{\|Q'\|_{L^2}^2} \int \eta \mathcal L(Q'') \right|+
\left| \frac {\lambda_s}\lambda - \frac 1{\|\Lambda Q\|_{L^2}^2} \int \eta \mathcal L\big((\Lambda Q)'\big) \right| \lesssim \int \frac {\eta^2+|\eta|^3}{1+y^2} .
\end{equation}
Now, we insert $\eta=\varepsilon+bP_b$ in \eqref{ortho.3}.
Note that
\[
\int P \mathcal L(Q'') = - \int (\mathcal L P)' Q' =-\int Q' \Lambda Q=0,\quad
\int P \mathcal L\big((\Lambda Q)'\big) = - \int (\mathcal L P)' \Lambda Q = - \|\Lambda Q\|_{L^2}^2\,;
\]
moreover, from the definition of $P_b$ and the fact that $P$ is bounded,
\[
\int \frac {|P-P_b|}{1+y^2}  \lesssim  \int \frac {|1-\chi_b|}{1+y^2} \lesssim |b|\,, 
\quad  \int \frac {|bP_b|^2+|bP_b|^3}{1+y^2} \lesssim |b|^{2},
\]
 and from \eqref{gn},
 \[
 \int \frac { | \varepsilon|^3} {1+y^2}  \lesssim \left( \int\frac { | \varepsilon|^2}{1+y^2}   \right)^{\frac 12}
 \left( \int   | \varepsilon|^4 \right)^{\frac 12}\lesssim \left( \int\frac { | \varepsilon|^2}{1+y^2}   \right)^{\frac 12}.
 \]
Thus,  we obtain
\begin{equation}\label{ortho.4}
 \left|  \frac {x_s}\lambda- 1 \right|+
\left| \frac {\lambda_s}\lambda+ b \right| \lesssim |b|^{2}+
\left( \int \frac {\varepsilon^2}{1+y^2} \right)^{\frac 12} ,
\end{equation}
which yields \eqref{modulation.7}.
\end{proof}

\subsection{Cauchy problem and weak continuity of the flow}\label{S.2.4}
In this subsection, we recall known facts on the Cauchy problem in $H^{\frac 12}$ and the weak continuity of the flow for \eqref{mBO}. Then, we show that the decomposition of Lemma~\ref{modulation} is stable by weak $H^{\frac 12}$ limit, a technical fact that will be used in the proof of Theorem~\ref{th1}.

\begin{theorem}[\cite{KeTa,CuKe}]\label{th:CP}
For any $u_0\in H^{\frac 12}(\mathbb R)$, the following holds true. 

 {\rm (i)}  there exist $T=T(\|u_0\|_{H^{\frac 12}})$ and a unique solution $u$ of the equation \eqref{mBO} satisfying
$u(0,\cdot)=u_0$ and $u\in C([-T,T]\, :\, H^{\frac 12}(\mathbb R)) \cap X_T$,
for a certain resolution space $X_T$ (see \cite{KeTa}).
Moreover, for any $R>0$,   the flow map data solution is Lipschitz continuous from $ \big\{u_0 \in H^{\frac12}(\mathbb R) : \| u_0\|_{H^{\frac12}} < R \big\} $ into   $C([-T(R),T(R)]:H^{\frac12}(\mathbb R))$. 

 {\rm (ii)}  Let $\{u_{0,n}\}$ be a sequence of $H^{\frac12}$ initial data such that
\[
u_{0,n}\rightharpoonup u_0\quad \hbox{in } H^{\frac 12}\ \hbox{as }\ n\to +\infty.\]
Assume that, for some $C, T_1, T_2>0$, for all $n>0$, the corresponding solution $u_n(t)$ of~\eqref{mBO}
exists on $[-T_1,T_2]$ and satisfies $\max_{t\in [-T_1,T_2]}\|u_n(t)\|_{H^{\frac 12}}\leq C$.
Let $u(t)$ be the solution of~\eqref{mBO} corresponding to $u_0$.
Then, $u(t)$ exists on $[-T_1,T_2]$ and
\begin{equation}\label{we:1}
\forall t\in [-T_1,T_2], \quad u_n(t) \rightharpoonup u(t)\quad \hbox{in } H^{\frac 12}\ \hbox{as }\ n\to +\infty.
\end{equation}
\end{theorem}
The second part of the Theorem, \emph{i.e.}~the weak convergence of $u_n(t)$ to $u(t)$, is stated in the last remark of \cite{CuKe}.  
For similar statements, we   refer to Lemma~30 in~\cite{MaMejmpa} in the case of the critical (gKdV) equation,
to Theorem~5 in~\cite{KeMa} in the case of the Benjamin-Ono equation in $H^{\frac 12}$,
to \cite{CuKe} for Benjamin-Ono in $L^2$  
and to Lemma~3.4 in~\cite{Co} for the mass supercritical (gKdV) equation.

As usual, given any $u_0\in H^{\frac 12}(\mathbb R)$, we   consider the   solution 
$u\in C([0,T^\star)\, :\, H^{\frac 12}(\mathbb R))$ emanating from $u_0$ at $t=0$ and defined on its maximal interval of existence $[0,T^\star)$. If $T^\star<+\infty$, we see from Theorem~\ref{th:CP} that
$\lim_{t \uparrow T^\star} \| D^{\frac12}u(t)\|_{L^2}=+\infty$.
A similar statement holds for negative times.

\begin{remark}
Let $u(t)$ be a solution of (mBO) such that $T^\star<+\infty$. Then,
\[
\|D^{\frac 12} u(t)\|_{L^2}\gtrsim (T^\star-t)^{-\frac 14}.
\]
Indeed, let $0<t_{0}<T^\star$, and let $v(s,y)$ be the following solution of (mBO) 
\[
v(s,y)= \lambda_{0}^{-\frac 12} u(\lambda_0^{-2} s + t_0,\lambda_{0}^{-1} y),\quad
\lambda_0=\|D^{\frac 12} u(t_0)\|_{L^2}^2.
\]
Then, $\|D^{\frac 12}v(0)\|_{L^2} = 1$ and $\|v(0)\|_{L^2} = \|u(t_0)\|_{L^2}= \|u(0)\|_{L^2}$.
By Theorem~\ref{th:CP}~(i), it follows that $v(s)$ exists as a solution of \eqref{mBO} in $H^{\frac 12}$ on a time interval 
$[0,S]$, $S>0$ independent of $t_0$. Thus, $T^\star> \lambda_0^{-2} S  + t_0$, which is equivalent to
\[
(T^\star-t_0)   \|D^{\frac 12} u(t_0)\|_{L^2}^4 \gtrsim 1.
\]
\end{remark}

Now, we claim the following  consequence of Theorem~\ref{th:CP}~(ii) on the decomposition of  Lemma \ref{modulation}.
 
\begin{lemma}[Weak $H^{\frac12}$ stability of the decomposition] \label{le:weak}
Let $\{u_{0,n}\}$ be a sequence of $H^{\frac12}$ initial data such that
\[
u_{0,n}\rightharpoonup u_0\quad \hbox{in } H^{\frac 12}\ \hbox{as }\ n\to +\infty.
\]
Let $u(t)$ be the solution of~\eqref{mBO} corresponding to $u_0$.
Assume   that for all $n>0$, $u_n(t)$ exists and satisfies~\eqref{tube.2} on $[-T_1,T_2]$ for some $T_1,T_2>0$ and that the parameters of the decomposition
$(\lambda_n,x_n,b_n,\varepsilon_n)$ of~$u_n$ given by Lemma~\ref{modulation} satisfy, for some $c,C>0$, for all $n$ large,
\begin{equation} \label{hypoweak}
\forall t\in [-T_1,T_2],\quad 0<c\leq \lambda_n(t)<C ,\quad \lambda_n(0)=1,\quad x_n(0)=0.
\end{equation}
Then, $u(t)$ exists and satisfies~\eqref{tube.2} on $[-T_1,T_2]$ and its decomposition $(\lambda,x,b,\varepsilon)$ satisfies, as $n\to +\infty$,
\[
\forall t\in [-T_1,T_2],\quad \varepsilon_n(t)\rightharpoonup \varepsilon(t)\quad \hbox{in } H^{\frac 12},\quad
\lambda_n(t)\to \lambda(t),\quad x_n(t)\to x(t), \quad b_n(t)\to b(t).
\]
\end{lemma}

\begin{proof}[Sketch of proof]
 We use the strategy of the proof of Lemma~17  in~\cite{MaMejmpa}.
We also refer to~\cite{MeRainvent}, page 599, for a more detailed argument.
The first step of the proof is to note that estimates~\eqref{modulation.7}
provide uniform bounds on the time derivatives of the geometric parameters $(\lambda_n(t),x_n(t),b_n(t))$ on $[-T_1,T_2]$.
Therefore, by Ascoli's theorem, up to the extraction of a subsequence,
\[
(\lambda_n(t),x_n(t),b_n(t))\to (\widetilde\lambda(t),\widetilde x(t),\widetilde b(t)) \quad \hbox{on } [-T_1,T_2],
\]
for some functions $(\widetilde\lambda(t),\widetilde x(t),\widetilde b(t))$.
Writing the orthogonality conditions~\eqref{modulation.2} in terms of $u_n(t)$ and $(\lambda_n(t),x_n(t),b_n(t))$,
using~\eqref{we:1} and passing to the limit as $n\to +\infty$,
we see that $u(t)$ and the limiting parameters $(\widetilde\lambda(t),\widetilde x(t),\widetilde b(t))$
satisfy the same orthogonality relations.
In particular, they correspond to the unique parameters $(\lambda(t),x(t),b(t))$ given by Lemma~\ref{modulation}.
This uniqueness statement proves by a standard argument that, for the whole sequence,
$(\lambda_n(t),x_n(t),b_n(t))$ converges to $(\lambda(t),x(t),b(t))$ on $[-T_1,T_2]$ as $n\to +\infty$.
It follows from \eqref{modulation.1} that $\varepsilon_n(t)\rightharpoonup \varepsilon(t)$ in $H^{\frac 12}$ as $n\to +\infty$.
\end{proof}

\subsection{Estimates and localization arguments for  fractional Laplacians}\label{s.2.5}
First, we recall various useful inequalities and commutator estimates 
related to $D^\alpha$ ($0<\alpha \le 1$),
and the Hilbert transform $\mathcal H$.
\begin{lemma}[\cite{Ca,KPV,DaMcPo}]\label{COMMUTATOR}
For any $f,g,a\in \mathcal S(\mathbb R)$,
\begin{equation}\label{gnp}
\forall \, 2\leq p< +\infty,\quad \|f\|_{L^p}
 \lesssim \|f\|_{L^2}^{\frac 2p} \|D^{\frac 12} f\|_{L^2}^{\frac {p-2}p},
\end{equation}
\begin{equation}\label{COMM1}
\forall \, 0<\alpha \le 1, \quad \|[D^{\alpha},g]f \|_{L^2}
\lesssim \|f\|_{L^4} \|D^{\alpha} g\|_{L^4},
\end{equation}
where $[D^{\alpha},g]f=D^{\alpha}(fg)-gD^{\alpha}f$, 
and
\begin{equation} \label{Calderon}
   \forall \, l, m \in \mathbb N \, ,
\quad
 \big\|\partial_x^l[\mathcal{H},a]\, \partial_x^mf\|_{L^2} \lesssim \|\partial_x^{l+m}a\|_{L^{\infty}}\|f\|_{L^{2}} \, ,\end{equation}
 where $[\mathcal{H},a]\, g=\mathcal{H}(ag)-a\mathcal{H}g$.
\end{lemma} 
Recall that \eqref{gnp} is the Gagliardo-Nirenberg inequality, which follows from complex interpolation and Sobolev embedding.

Estimate \eqref{COMM1} in the case $\alpha=1$ is due to Calder\'on \cite{Ca}, see also Coifman and Meyer \cite{CM}, 
formula (1.1).
Estimate \eqref{COMM1} in the case $0<\alpha<1$ is a consequence of Theorem A.8 in \cite{KPV} for functions depending only on $x$, with the following choice of parameters: $0<\alpha<1$, $\alpha_1=0$, $\alpha_2=\alpha$,
$p=2$, $p_1=p_2=4$.

Finally, estimate \eqref{Calderon} in the case $l=0$ and $m=1$ is the classical Calder\'on commutator estimate proved in \cite{Ca}. The general case was proved by Bajvsank and Coifman \cite{BaCo} (see also Lemma 3.1 of~\cite{DaMcPo} for a different proof).

\medskip

The following estimates are direct consequences of \eqref{gnp}-\eqref{Calderon}.
\begin{lemma} For any $a,f \in \mathcal S(\mathbb R)$,
\begin{equation}\label{m.1}
\left| \int (D^1 f) f' a \right|\lesssim \| f\|_{L^2}^2 \|a''\|_{L^\infty},
\end{equation}
and
\begin{equation}\label{m.2}
\left| \int (D^1 f) f a'  - \int |D^{\frac 12} f|^2 a'\right|\lesssim \|D^{\frac 12} f\|_{L^2}^{\frac 32}\| f\|_{L^2}^{\frac 12} \|a''\|_{L^2}^{\frac 34}\|a'\|_{L^2}^{\frac 14}.
\end{equation}
\end{lemma}
\begin{proof}
First,
\begin{align*}
\int (D^1 f) f' a & = 
\int (\mathcal H f') f' a 
= - \int f' \mathcal H (f' a) 
= - \int (\mathcal H f') f' a  - \int f' [\mathcal H,a] f'.
\end{align*}
Thus,
\begin{align*}
\int (D^1 f) f' a & = 
 \frac 12 \int f  \left([\mathcal H,a] f'\right)' ,
\end{align*}
which, combined to Cauchy-Schwarz inequality and \eqref{Calderon} (with $l=m=1$) implies \eqref{m.1}

\medskip

Second,
\begin{align*}
\int (D^1 f) f a' & = 
\int (D^{\frac 12} f)  D^{\frac 12} (f a')
= \int |D^{\frac 12} f|^2    a'
+\int (D^{\frac 12} f) [D^{\frac 12} , a']  f .
\end{align*}
By \eqref{COMM1} with $\alpha=\frac12$ and \eqref{gnp},
\begin{align*}
\left| \int (D^1 f) f a'  - \int |D^{\frac 12} f|^2 a'\right|
& \lesssim \|D^{\frac 12} f\|_{L^2} \| f\|_{L^4} \| D^{\frac 12} a'\|_{L^4}
\\
&\lesssim \|D^{\frac 12} f\|_{L^2}^{\frac 32}\| f\|_{L^2}^{\frac 12} \|D^1(a')\|_{L^2}^{\frac 12}\|D^{\frac 12}(a')\|_{L^2}^{\frac 12}
\\&\lesssim \|D^{\frac 12} f\|_{L^2}^{\frac 32}\| f\|_{L^2}^{\frac 12} \|a''\|_{L^2}^{\frac 34}\|a'\|_{L^2}^{\frac 14},
\end{align*}
which proves \eqref{m.2}.
\end{proof}

In the proof of Theorem~\ref{th1}, we will also need sharp localization arguments reminiscent of the  identity and the smoothing effect first observed by Kato for the generalized KdV equations (see \cite{Ka} and also e.g. \cite{MaMejmpa}).

Let  \[\phi(x)=\frac 1{\pi} \int_{-\infty}^x \frac{d y}{1+y^2}.\]
We recall the following estimates. 
\begin{lemma}[Lemmas 6 and 7 of \cite{KeMaRo}, Lemmas 2 and 3 of \cite{KeMa}]\label{le.commutateur}
There exists $C_1>0$ such that, for any $f\in \mathcal S(\mathbb R)$,
\begin{equation} \label{eq:41b}
  \left|\int (D^1 f)f' \phi\right|\leq C_1 \int f^2 \phi' ,
\end{equation}
and
\begin{equation}\label{eq:41c}
  \left| \int (D^1 f)f \phi'
 -\int  \left| D^{\frac 1 2}\left( f\sqrt{\phi'} \right)\right|^2\right|
 \leq C_1 \int f^2 \phi'.
\end{equation}
\end{lemma}
\begin{remark}
When compared to \eqref{m.1}-\eqref{m.2}, the main point of the estimates  \eqref{eq:41b}-\eqref{eq:41c} is to obtain error terms depending only on localized $L^2$ quantities. In return, such estimates require a special choice of function $\phi$ - see \cite{KeMaRo}.
The weak decay of $\phi'$ is a difficulty in Sect.~4, but   due to the nonlocal nature of $D^1$, it is not clear whether \eqref{eq:41b}-\eqref{eq:41c} hold for functions decaying at $\infty$ faster than $\phi'$ (see also Lemma \ref{HY}).
\end{remark}

\section{Uniform estimates}\label{sect3}
 In this section, we define a specific sequence of global solutions with a rigid behavior, related to the desired minimal mass regime in Theorem \ref{th1}.
\subsection{Uniform bounds for a sequence of solutions}\label{s:BS}
Let $T_n=\frac1n$ for $n>1$ large. Let
\begin{equation} \label{unTn}
u_n^{in}(x)=\frac {1-a^{in}_n} {(\lambda^{in}_n)^{\frac12}} Q\left(\frac{x-x_n^{in}}{\lambda_n^{in}}\right) \quad \text{where} \quad   x_n^{in}=-\ln(n),
\end{equation}
the parameter $\lambda_n^{in}\sim \frac 1n$ is to be fixed later and $a^{in}_n$ is uniquely chosen (depending on $\lambda_n^{in}$) so that (see \eqref{Ea})
\begin{equation}\label{ain}
  E(u_n^{in})=(\lambda_n^{in})^{-1}E((1-a^{in}_n) Q)=p_0,\quad
 a^{in}_n\sim \frac {p_0}{\int Q^2}\lambda_n^{in} \sim \frac {p_0}{\int Q^2}\frac 1n\,.
\end{equation}
(Recall that $p_0$ is defined in \eqref{nlprofile.2}.)
Note that $\int (u_n^{in})^2=(1-a^{in}_n)^2 \int Q^2$ so that
\begin{equation}\label{inittrois}
  \int Q^2 - \int (u_n^{in})^2  \sim \frac {p_0}{\int Q^2}  \frac 2n.
\end{equation}
 We consider the global solution $u_n$ of \eqref{mBO}  corresponding to the data $u_n(T_n)=u_n^{in}$.
For $n$ large, let $0<\tau_n\leq +\infty$ be such that $\mathcal{I}_n=[T_n,T_n+\tau_n)$ is the maximal time interval where $u_n$ satisfies \eqref{tube.2} ($\tau_n>0$ exists by \eqref{unTn} and by continuity of $t\mapsto u_n(t)$).  By using Lemma~\ref{modulation}, we  decompose $u_n(t)$ 
for $t \in \mathcal{I}_n$, as 
\begin{equation} \label{decomposition.1} 
u_n(t,x)=\frac1{\lambda^{\frac12}_n(t)} \left(Q_{b_n(t)} +\varepsilon_n\right)(t,y) \, ,
\quad y = \frac{x-x_n(t)}{\lambda_n(t)},
\end{equation}
where $\lambda_n, \, x_n, \, b_n$ are $\mathcal{C}^1$ functions. 
We define the rescaled time variable $s$ by 
\begin{equation} \label{sn}
s=s(t)=S_n+\int_{T_n}^t\frac{dt'}{\lambda_n^2(t')} \quad \text{with} \quad s(T_n)=S_n=-n \, .
\end{equation}
We   consider all time dependent functions (such as $\lambda_n, \, b_n, \, x_n$ and $\varepsilon_n$) indifferently as functions of $t$ on $\mathcal I_n$ or of $s$ on the  interval $\mathcal{J}_n=s(\mathcal{I}_n)$. 
From now on and until Sect.~\ref{S6} where we go back to the original variables $(t,x)$, we work for fixed $b$ and with the rescaled variables $(s,y)$.
Note that by \eqref{unTn} and \eqref{modulation.1}, we have
 \begin{equation} \label{decomposition.2}\begin{aligned}
& \lambda_n(S_n)=\lambda_n^{in}, \quad b_n(S_n)=b_n^{in}=-\lambda^{in}_n, \quad x_n(S_n)=x_n^{in},\\
& \varepsilon_n(S_n)= \varepsilon_n^{in}= -a^{in}_nQ - b_n^{in} P_{b_n^{in}}.
\end{aligned}\end{equation}

\medskip

Fix 
\begin{equation}\label{theta.def}
\frac 35 < \theta < \frac 23 \quad  \hbox{and}\quad  B=100 \, C_1\, ,
\end{equation}
where $C_1>0$ is the universal constant in Lemma  \ref{le.commutateur}.
Let
\begin{equation}\label{def.varphi}
\phi(x)=\frac 1{\pi} \int_{-\infty}^x \frac{d y}{1+y^2}=\frac {\arctan x}{\pi} + \frac 12\,,\quad
\varphi(s,y) = \frac {\phi\left(\frac {y}B +|s|^{\theta}\right)} {\phi\left(|s|^{\theta}\right)}\,,
\end{equation}
and \begin{equation}\label{defN}
\mathcal N(\varepsilon_n)=\left( \int  |D^{\frac 12} \varepsilon_n|^2 +\varepsilon_n^2 \varphi \,\right)^{\frac 12} .
\end{equation}
Moreover, let
\begin{equation}\label{defJ1}
\rho(y)=\int_{-\infty}^y \Lambda Q(y') dy', \quad
J_{n}(s)= \int \varepsilon_{n}(s,y) \rho(y) \chi(-y|s|^{-\frac 23}) dy,
\end{equation}
and
\begin{equation}\label{defl01}
\mu_{n}(s) =|1-J_{n}(s)|^{\frac 1{p_{0}}} \lambda_{n}(s).
\end{equation}
Note that $|J_{n}|\lesssim \int_{y<0} |\varepsilon_n|(1+|y|)^{-1} + \int_{0<y<2|s|^{\frac 23}}| \varepsilon_{n}|\lesssim |s|^{\frac 13} \| \varepsilon_{n}\|_{L^2}.$
Observe that by \eqref{decomposition.2},
\begin{equation}\label{epsilton.initial}
\|\varepsilon_n(S_n)\|_{L^2} \lesssim  {|S_n|^{-\frac 12}},  \quad
\mathcal N(\varepsilon_n(S_n)) \lesssim |a_n^{in}|+|b_n^{in}| |S_n|^{\frac \theta2}\lesssim |S_n|^{-1+\frac \theta 2}.
\end{equation}

Let $C^{\star}>1$, $S_0<-1$  and $n_0 >1$  to be fixed later.

\begin{definition} \label{bootstrap}  For $n \ge n_0$ and $s \in \mathcal{J}_n$, $s <S_0$. We say that $(\lambda_n, \, b_n, \, x_n,\varepsilon_n)$ satisfy the \textit{bootstrap estimates} at the time $s$ if
\begin{equation} \label{bootstrap.1}
\Big|\mu_n(s)-\frac1{|s|} \Big|  \leq   {|s|^{-\frac 76 + \frac \theta 4}}, 
\end{equation}
\begin{equation} \label{bootstrap.2}
\Big|\lambda_n(s)-\frac1{|s|} \Big|=\Big|b_n(s)-\frac1s \Big| \leq   {|s|^{-\frac {13}{12} + \frac \theta 8}} \, ,
\end{equation}
\begin{equation} \label{bootstrap.3} 
\Big|x_n(s)+\ln |s|  \Big| \leq |s|^{-\frac 1{12}+ \frac \theta 8},
\end{equation}
and 
\begin{equation} \label{bootstrap.4}
\mathcal N(\varepsilon_n(s))  \le C^{\star}  {|s|^{-1+\frac \theta2}} \, .
\end{equation}
\end{definition}

\begin{definition} \label{bootstraptime}
For $n \ge n_0$, we define (if this set is not empty),
\begin{equation} %\label{bootstraptime.1}
S_n^{\star}:=\sup \big\{s \in \mathcal{J}_n \cap (S_n,S_0) \ : \   \text{\eqref{bootstrap.1}-\eqref{bootstrap.4}  are satisfied on $[S_n,s]$} \,\big\} \,,
\end{equation}
and $\mathcal I_n^\star = [S_n,S_n^\star]$.
\end{definition}

The main result of this section states that there exists at least one choice of $\lambda_n^{in}\sim\frac 1n$ such that $S_n^\star=S_0$ i.e. such that the bootstrap estimates in Definition \ref{bootstrap} are valid up to a time $S_0$ independent of~$n$.
\begin{proposition} \label{bootstrapprop}
There exist  $C^\star>1$, $S_0<-1$ and $n_0>1$ such that, for all $n>n_0$, there exists $\lambda_n^{in}$   
satisfying 
\begin{equation}\label{Snstar}
\Big|\lambda_n^{in} - \frac 1{n}\Big|< n^{-\frac {13}{12} + \frac \theta 8}
\quad \hbox{and}\quad  S_n^{\star}= S_0.
\end{equation}
\end{proposition}

We prove Proposition \ref{bootstrapprop} in Sections~\ref{S4} and~\ref{S5}.
For the sake of simplicity,  we will omit the subscript $n$  and write $x$, $\lambda$, $\mu$, $b$ and $\varepsilon$ instead of $x_n$, $\lambda_n$, $\mu_n$, $b_n$ and $\varepsilon_n$.

\subsection{Consequences of the bootstrap bounds}
 For future reference, we state here some   consequences of the   bootstrap estimates.
\begin{lemma}
For all $s\in \mathcal I^\star$,
\begin{equation}\label{BS.1}
\lambda+|b|+\|\varepsilon\|_{L^2}^2\lesssim |s|^{-1},
\end{equation}
\begin{equation}\label{BS.2}
\left|\frac{x_s}\lambda -1 \right| + \left| \frac {\lambda_s}{\lambda} +b\right|   \lesssim |s|^{- 2}+ \mathcal N (\varepsilon)\lesssim C^\star |s|^{-1+\frac \theta2},
\end{equation}
\begin{equation}\label{BS.3}
|b_s+b^2|=|\lambda|  \left| \frac {\lambda_s}{\lambda} +b\right|  \lesssim |s|^{-3}+ |s|^{-1}\mathcal N (\varepsilon)\lesssim C^\star |s|^{-2+\frac \theta2} ,
\end{equation}
 so that 
\begin{equation} \label{BS.5}
|b_s| \le |b_s+b^2|+|b|^2 \lesssim C^{\star}|s|^{-2+\frac{\theta}2} ,
\end{equation}
and 
\begin{equation}\label{BS.4}
|(\varepsilon,Q')|+|(\varepsilon,\Lambda Q)|\lesssim |s|^{-1} ,\quad |(\varepsilon,Q)|\lesssim |s|^{-\frac 32}+
\mathcal N (\varepsilon)^2\lesssim (C^\star)^2 |s|^{-2+\theta}.
\end{equation}
\end{lemma}
\begin{proof}
From \eqref{ain} and the definition of $b$ in \eqref{modulation.1}, we have $b=-\lambda$.
Thus, these estimates are direct consequences of the definition of $S^\star_n$ and Lemma~\ref{modulation}.
\end{proof}
\begin{remark}\label{rk:BS}
We note that the estimate on the full $L^2$ norm of $ \varepsilon$ given by \eqref{BS.1} cannot be improved.
In contrast, the $\dot H^{\frac 12}$ norm and local $L^2$ norm of $ \varepsilon$ have a better decay as $t\downarrow 0$ (see \eqref{bootstrap.4}). This phenomenon seems specific to KdV-type equations and requires the use of local norms and estimates, which, as mentionned in the Introduction, are delicate to combine with the non local operator $D^1$. See Lemma~\ref{le.commutateur}. Note that in the present paper, we localize only the $L^2$ term in the definition of the energy and local norm (see \eqref{defN} and \eqref{defF} below). See also \cite{CoMa} on such questions.
\end{remark}

\section{Control of   $\mathcal N(\varepsilon)$}\label{S4}

In this section, we close the estimates for $\mathcal N(\varepsilon)$, i.e., we strictly improve \eqref{bootstrap.4} under the bootstrap assumptions.
We consider the following main functionals
\begin{equation}\label{defF}
F = \int |D^{\frac 12}\varepsilon|^2 +\varepsilon^2 \varphi
- \frac 12 \left(\left( Q_{b}+\varepsilon\right)^4-Q_{b}^4- 4 Q_{b}^3 \varepsilon\right)  ;
\end{equation}
\begin{equation}\label{defG}
G  = \int \psi  \varepsilon^2
\quad \hbox{where}\quad
\psi(y)=  |s|^{\frac 12} \pi\left(  \phi\left(y|s|^{-\frac 12}\right)-\frac 12\right) \chi^2(y|s|^{-\frac 35}),
\end{equation}
\begin{equation}\label{defK}
K   = \int \varepsilon \mathcal L_\varphi P_b \quad \hbox{where} \quad  
\mathcal L_\varphi f = D^1 f + f \varphi -3Q^2 f\,.
\end{equation}
The functional $F$ is a localized energy functional in $\varepsilon$. 
At the quadratic order in $\varepsilon$, it is similar to $(\mathcal L\varepsilon,\varepsilon)$, except for the localization factor $\varphi$ on the $L^2$ term, as in $(\mathcal L_\varphi \varepsilon, \varepsilon)$.
The functional $G$ is related to Virial identity and is useful to cancel some critical terms in $\frac {d F} {ds}$, 
as first observed for the (NLS) equation in \cite{RaSz}, see Lemma \ref{le:GG}. The functional $K$ appears as an error term in the variation of $E$, but we will see in Lemma \ref{le:KK} that it enjoys a special algebra related to scaling variation.

\subsection{Energy-type estimate}

\begin{lemma}\label{le:FF}  There exists $C_0>0$ independent of $C^{\star}$ such that  for $|S_0|$ large enough, depending on $C^\star$, and for all $s\in \mathcal I^\star$,
\begin{equation}\label{eq:F}\begin{aligned}
&\frac {d}{ds}\left( \frac F\lambda\right) +\frac 1{4\lambda} \int \varepsilon^2 \varphi_y
 \leq    {C_0}\, C^\star|s|^{-2+\theta } - A
 -\left(\frac{\lambda_s}{\lambda} + b\right)  \left(\frac{F}\lambda {  -}  K + \frac Z\lambda\right),
 \end{aligned}\end{equation}
 where
 \begin{align}\label{defA}
A&=   \int \varepsilon^2 \varphi  -6 \int \varepsilon^2 Q  \Lambda Q   
 + 2 \left(\frac{\lambda_s}{\lambda} + b\right) \int \varepsilon y \Lambda  Q
 + 2 \left(\frac{x_s}{\lambda} -1\right) \int \varepsilon y   Q' \, ,
\\
\label{defZ}
 Z &=  \int \varepsilon^2 y \varphi_y \chi_1 -2 \int \varepsilon \Lambda Q (1-\varphi) \chi_2
\end{align}
and 
\begin{equation}   
\chi_1(s,y)=1-\chi(y |s|^{-2/3}), \quad
\chi_2(s,y)=1-\chi\left(y \frac 4 B |s|^{-\theta}\right).\label{chi1chi2}
\end{equation}
\end{lemma}

\begin{proof}
The following notation will   be used
\begin{equation}\label{defV}\begin{aligned}
&V=V(\varepsilon)=D^1 \varepsilon + \varepsilon   - \left((Q_b+\varepsilon)^3 - Q_b^3\right),\quad 
\\&V_\varphi=V_\varphi(\varepsilon) = D^1 \varepsilon + \varepsilon \varphi - \left((Q_b+\varepsilon)^3 - Q_b^3\right).
\end{aligned}\end{equation}
We compute using \eqref{modulation.4},
\begin{align*}
\lambda \frac d{ds} \left(\frac F\lambda\right) & =  2 \int \varepsilon_s V_\varphi - \frac {\lambda_s}\lambda F
  -2 \int \partial_s Q_b \left( (Q_b+\varepsilon)^3 - Q_b^3 - 3 Q_b^2 \varepsilon\right) +  \int \varepsilon^2  \partial_s \varphi \\
& = f_1+f_2+f_3+f_4+f_5 \, ,
\end{align*}
where
\begin{align*}
	& f_1 = 2 \int V_y  V_\varphi \, , \\
	& f_2 = 2 \frac{\lambda_s}{\lambda} \int (\Lambda \varepsilon) V_\varphi  - \frac {\lambda_s}\lambda  F + 2(\frac {\lambda_s}{\lambda} + b) \int (\Lambda Q_b) V_\varphi - 2 (b_s+b^2)  \int \frac{\partial Q_b}{\partial b} V_\varphi \, ,\\
	& f_3 = 2(\frac {x_s}\lambda-1) \int (Q_b+\varepsilon)_y V_\varphi \, , \\
	& f_4 = 2 \int \Psi_b V_\varphi \, , \\
	& f_5 = -2 \int (Q_b)_s \left( (Q_b+\varepsilon)^3 - Q_b^3 - 3 Q_b^2 \varepsilon\right)
	+ \int \varepsilon^2  \partial_s \varphi \, .
\end{align*}
In the rest of the proof,
we estimate the terms $f_1$, \ldots, $f_5$, taking $S_0$ large enough, possibly depending on  $C^\star$.
For future reference, note that
\begin{equation}\label{def.Dvarphi}
\varphi_y(s,y)=\partial_y \varphi(s,y)= \frac 1{B}  \frac{\phi'\left(\frac {y}B +|s|^{\theta}\right)}{\phi\left(|s|^{\theta}\right)}
= \frac 1{B\pi}  \frac{1}{\phi\left(|s|^{\theta}\right)}\frac 1{1+\left(\frac yB + |s|^{\theta}\right)^2}\,.
\end{equation}

\medskip

\noindent \emph{Estimate for $f_1$.} Integrating by parts, we have
\[
f_1 = 2 \int V_y V_\varphi =2 \int V_y V + 2 \int V_y (V_\varphi-V) = 2 \int V_y \varepsilon (\varphi-1) .
\]
Note that by integration by parts,
\begin{align*}
& 2 \int \left(D^1 \varepsilon\right)_y \varepsilon (\varphi-1) = 2 \int \left(D^1 \varepsilon\right)_y \varepsilon  \varphi=
-2 \int \left(D^1 \varepsilon\right)\left( \varepsilon_y  \varphi  +\varepsilon \varphi_y\right).
\end{align*}
By the change of variable $y' = \frac y{B}+|s|^\theta$, we have  from Lemma \ref{le.commutateur},  
\[  \left|\int \left(D^1 \varepsilon\right)  \varepsilon_y  \varphi  \right| 
\leq \frac{ C_1} B \int   \varepsilon^2\varphi_y \]
and 
\[  
\left|\int \left(D^1 \varepsilon\right)   \varepsilon \varphi_y 
- \int  \left| D^{\frac 1 2}\left( \varepsilon\sqrt{\varphi_y}\right)\right|^2\right|
\le \frac {C_1} B \int   \varepsilon^2 \varphi_y
\]
Recall that $B = 100 C_1$.
Thus,
\begin{align*}
 2 \int \left(D^1 \varepsilon\right)_y \varepsilon (\varphi-1)   \leq  - 2\int  \left| D^{\frac 1 2}\left( \varepsilon\sqrt{\varphi_y}\right)\right|^2 +\frac 1{25} \int \varepsilon^2 \varphi_y.
\end{align*}
Next, integrating by parts,
\[ 
2 \int \varepsilon_y  \varepsilon (\varphi-1) = - \int \varepsilon^2 \varphi_y,
\]
and
\begin{align*}
 \int \big((Q_b&
+\varepsilon)^3-Q_b^3\big)_y \varepsilon (\varphi-1)\\
&=-\int \left((Q_b+\varepsilon)^3-Q_b^3\right)  \varepsilon_y (\varphi-1)-\int \left((Q_b+\varepsilon)^3-Q_b^3\right)  \varepsilon  \varphi_y\\
& = -\frac 14 \int \left(   {(Q_b+\varepsilon)^4}  - 4 Q_b^3 \varepsilon -  {Q_b^4} \right)_y (\varphi-1) \\
&\quad+ \int (Q_b)_y \left( (Q_b+\varepsilon)^3-3 \varepsilon Q_b^2-Q_b^3\right) (\varphi-1)
 -\int \left((Q_b+\varepsilon)^3-Q_b^3\right)  \varepsilon  \varphi_y\\
& = \frac 14 \int \left( {(Q_b+\varepsilon)^4}  - 4 Q_b^3 \varepsilon -  {Q_b^4}  - 4 \left((Q_b+\varepsilon)^3-Q_b^3\right)  \varepsilon\right) \varphi_y
\\ &\quad +\int (Q_b)_y \left( (Q_b+\varepsilon)^3-3 \varepsilon Q_b^2-Q_b^3\right) (\varphi-1).
\end{align*}
Note that
\[
\left|   {(Q_b+\varepsilon)^4}  - 4 Q_b^3 \varepsilon -  {Q_b^4}  - 4 \left((Q_b+\varepsilon)^3-Q_b^3\right)  \varepsilon\right|\lesssim
\varepsilon^4 + \varepsilon^2 Q_b^2.
\]
First, by Cauchy-Schwarz inequality and then \eqref{gnp},
we have
\begin{align*}
\int \varepsilon^4 \varphi_y& = \int \left|\varepsilon \sqrt{\varphi_y}\right|^2 \varepsilon^2  
\lesssim \left( \int \left|\varepsilon \sqrt{\varphi_y}\right|^4\right)^{\frac 12} \left(\int \varepsilon^4 \right)^{\frac 12}
\\& \lesssim  \left( \int \left|\varepsilon \sqrt{\varphi_y}\right|^2\right)^{\frac 12}
\left( \int \left|D^{\frac 12}\left(\varepsilon \sqrt{\varphi_y}\right)\right|^2\right)^{\frac 12} \left(\int \varepsilon^2 \right)^{\frac 12}\left(\int \left|D^{\frac 12}\varepsilon\right|^2 \right)^{\frac 12}\\
& \lesssim \frac 1 {|s|} \left( \int \left|D^{\frac 12}\left(\varepsilon \sqrt{\varphi_y}\right)\right|^2 + \int \varepsilon^2  {\varphi_y}\right).
\end{align*}
Second, since
\[
Q^2 \varphi_y \lesssim \left\{\begin{aligned}
& |s|^{-2\theta} Q^2 \lesssim |s|^{-2 \theta} \varphi  &\quad \hbox{for $y>-|s|^{\frac 14}$},\\
& |s|^{-1} \varphi_y  &\quad \hbox{for $y<-|s|^{\frac 14}$},
 \end{aligned} \right.
\]
we have
\[\int \varepsilon^2 Q_b^2 \varphi_y
\lesssim    \int \varepsilon^2 (Q^2+b^2) \varphi_y 
\lesssim   |s|^{-2\theta} \mathcal N (\varepsilon)^2+ |s|^{-1} \int |\varepsilon|^2 \varphi_y.
\]
Thus,  for $|S_0|$ large enough,
\begin{align*}
&\left|\int \left(   {(Q_b+\varepsilon)^4} - 4Q_b^3 \varepsilon -  {Q_b^4}  - 4\left((Q_b+\varepsilon)^3-Q_b^3\right)  \varepsilon\right) \varphi_y\right|\\
& \leq \frac 1{100}  \left(\int  \left| D^{\frac 1 2}\left( \varepsilon\sqrt{\varphi_y}\right)\right|^2 + \int \varepsilon^2 \varphi_y\right)
+ C|s|^{-2\theta}\mathcal N(\varepsilon)^2.
\end{align*}
Similarly, 
\[
 \left|(Q_b+\varepsilon)^3-3 Q_b^2\varepsilon -Q_b^3\right|\lesssim  |\varepsilon|^3 + \varepsilon^2 .
\]
Since $(Q_b)_y = Q' + b P' \chi_b + b P \chi_b'$, 
\[
|(Q_b)_y|\lesssim |Q'|+|b|Q + b^2\mathbf{1}_{[-2 |b|^{-1},-  |b|^{-1}]} \,,
\]
we first observe that
\begin{align*}
  \int_{y< -\frac B 2 |s|^{\theta}} |(Q_b)_y|  \left(|\varepsilon|^3+\varepsilon^2\right)
\lesssim   |s|^{-3\theta}  \int \left(|\varepsilon|^3+\varepsilon^2\right) \lesssim  |s|^{-1-3\theta}.
\end{align*}
Second, by the definition of $\varphi$,
\begin{equation}\label{phi_B}
\hbox{for $y>-\frac B 2 |s|^{\theta}$},\quad
|1-\varphi(y)|=|\varphi(0)-\varphi(y)|\lesssim |y| \max_{[-\frac B 2 |s|^{\theta},+\infty)} \varphi_y \lesssim |y| |s|^{-2\theta},
\end{equation}
and so, since $|y|Q\lesssim Q^{\frac 12} \lesssim \varphi$ and $\int |\varepsilon|^3 \varphi \lesssim \int \varepsilon^2 \varphi  + \int \varepsilon^4
\lesssim \mathcal N(\varepsilon)^2$,
\[
\int_{y> -\frac B 2 |s|^{\theta}} |(Q_b)_y|  \left(|\varepsilon|^3+\varepsilon^2\right) |1-\varphi|
\lesssim |s|^{-2\theta} \int  (|\varepsilon|^3 + \varepsilon^2 ) \varphi
\lesssim |s|^{-2\theta} \mathcal N(\varepsilon)^2.
\]

In conclusion, since $\frac35<\theta<\frac23$, we obtain
for $|S_0|$ large enough,
\begin{equation}\label{eq:f1}
f_1 \leq  - \int  \left( D^{\frac 1 2}\left( \varepsilon\sqrt{\varphi_y}\right)\right)^2 -\frac 12 \int \varepsilon^2 \varphi_y + 
C |s|^{-3+\theta}.
\end{equation}
\medskip

\noindent \emph{Estimate for $f_2$.}
Since $ \mathcal H(y \varepsilon_y) = y \mathcal H( \varepsilon_y)$, we have
\[\int y \varepsilon_y (D^1 \varepsilon) =  \int y \varepsilon_y \mathcal H(\varepsilon_y) = - \int \mathcal H(y \varepsilon_y)(\varepsilon_y) =
- \int y \mathcal \mathcal H(\varepsilon_y) (\varepsilon_y) =0,\]
and so
\[
2 \int (\Lambda \varepsilon) (D^1 \varepsilon)=   \int   |D^{\frac 12} \varepsilon|^2 .
\]
By integration by parts,
\[
2 \int (\Lambda \varepsilon) \varepsilon\varphi= -\int \varepsilon^2 y   \varphi_y.
\]
Moreover,
\[ - 2 \int \Lambda \varepsilon   \left((Q_b+\varepsilon)^3-Q_b^3\right)
=- \int \varepsilon \left((Q_b+\varepsilon)^3-Q_b^3\right) - 2 \int y\varepsilon_y \left((Q_b+\varepsilon)^3-Q_b^3\right)\]
where,  integrating by parts,
\begin{align*}
  &- 2 \int y\varepsilon_y \left((Q_b+\varepsilon)^3-Q_b^3\right)\\
 & =-\frac 12 \int y \left(\left( Q_{b}+\varepsilon\right)^4-Q_{b}^4- 4 Q_{b}^3 \varepsilon\right)_y  
    +2 \int y (Q_b)_y \left(\left( Q_{b}+\varepsilon\right)^3-Q_{b}^3- 3 Q_{b}^2 \varepsilon \right)\\
  & =\frac 12 \int   \left(\left( Q_{b}+\varepsilon\right)^4-Q_{b}^4- 4 Q_{b}^3 \varepsilon\right)
  +2 \int y (Q_b)_y \left(\left( Q_{b}+\varepsilon\right)^3-Q_{b}^3- 3 Q_{b}^2 \varepsilon \right).
\end{align*}
Note that
\begin{align*}
&- \int \varepsilon \left((Q_b+\varepsilon)^3-Q_b^3\right)+ \int   \left(\left( Q_{b}+\varepsilon\right)^4-Q_{b}^4- 4 Q_{b}^3 \varepsilon\right)
  +2 \int y (Q_b)_y \left(\left( Q_{b}+\varepsilon\right)^3-Q_{b}^3- 3 Q_{b}^2 \varepsilon \right)\\
&= 6 \int \varepsilon^2 Q_b \Lambda Q_b + 2 \int \varepsilon^3 \Lambda Q_b .
\end{align*}
Therefore, for the first part of $f_2$, we obtain
\begin{align*}
 2 \frac{\lambda_s}{\lambda} \int (\Lambda \varepsilon) V_\varphi  - \frac {\lambda_s} \lambda F
 &= \frac{\lambda_s}{\lambda} \left( -\int \varepsilon^2 y \varphi_y - \int \varepsilon^2 \varphi 
+6 \int \varepsilon^2 Q_b \Lambda Q_b + 2 \int \varepsilon^3 \Lambda Q_b \right)\\
&= - \frac{\lambda_s}{\lambda} \int \varepsilon^2 y \varphi_y \\
& +  (\frac {\lambda_s}{\lambda} + b)
\left( - \int \varepsilon^2 \varphi  +6 \int \varepsilon^2 Q_b \Lambda Q_b + 2 \int \varepsilon^3 \Lambda Q_b \right),
\\&   -b \left( - \int \varepsilon^2 \varphi  +6 \int \varepsilon^2 Q_b \Lambda Q_b + 2 \int \varepsilon^3 \Lambda Q_b \right),
\end{align*}

\medskip

Then, we observe that 
\begin{align*}
2 \int (\Lambda Q_b) V_\varphi
&= 2 \int (\Lambda Q_b) (D^1 \varepsilon + \varepsilon \varphi - 3 Q_b^2 \varepsilon - 3 Q_b \varepsilon^2 - \varepsilon^3)\\
&= 2 \int \varepsilon (D^1 (\Lambda Q_b) + \Lambda Q_b \varphi - 3 Q_b^2 \Lambda Q_b) 
-   6 \int \varepsilon^2 Q_b \Lambda Q_b - 2 \int \varepsilon^3 \Lambda Q_b.
\end{align*}
Thus, summing these two expressions, we obtain
\begin{align*}
f_2
&= - \frac{\lambda_s}{\lambda} \int \varepsilon^2 y \varphi_y \\
& +2  (\frac {\lambda_s}{\lambda} + b)
  \left[\int  (D^1 (\Lambda Q_b) + \Lambda Q_b \varphi - 3 Q_b^2 \Lambda Q_b) \varepsilon -\frac 12 \int \varepsilon^2 \varphi\right],\\
& - 2 (b_s+b^2)  \int \frac{\partial Q_b}{\partial b} V_\varphi
\\&   -b \left( - \int \varepsilon^2 \varphi  +6 \int \varepsilon^2 Q_b \Lambda Q_b + 2 \int \varepsilon^3 \Lambda Q_b \right)\\
&=  f_{2,1}+f_{2,2}+f_{2,3}+f_{2,4}.
\end{align*}
First, we split $f_{2,1}$ into two parts
\begin{equation}\label{e.f21}
f_{2,1} =- \left(\frac{\lambda_s}{\lambda}+b\right) \int \varepsilon^2 y \varphi_y + b \int \varepsilon^2 y \varphi_y.
\end{equation}
We split the first term in right-hand side of \eqref{e.f21} using $Z_1=\int \varepsilon^2 y \varphi_y \chi_1$,
where $\chi_1$ is defined in \eqref{defZ}
\[
\int \varepsilon^2 y \varphi_y = Z_1 + \int \varepsilon^2 y \varphi_y \chi( y |s|^{-2/3}).
\]
Note that
\begin{align*}
\left|\int \varepsilon^2 y \varphi_y \chi( y |s|^{-2/3})\right|
&\lesssim \int_{|y|<2|s|^{2/3}} |y| \varphi_y \varepsilon^2+ \int_{y>2|s|^{2/3}} |y| \varphi_y \varepsilon^2\\ 
&\lesssim |s|^{\frac 23} \int \varepsilon^2\varphi_y + |s|^{-2/3} \mathcal N(\varepsilon)^2,
\end{align*}
and thus by \eqref{BS.2},
\begin{align*}
\left|\frac{\lambda_s}{\lambda}+b\right|\left|\int \varepsilon^2 y \varphi_y \chi( y |s|^{-2/3})\right|
&\lesssim C^\star |s|^{-\frac 13+\frac \theta 2} \int \varepsilon^2\varphi_y 
+ (C^\star)^3 |s|^{-\frac {11} 3+\frac 32 \theta}\\
&\lesssim o(1) \int \varepsilon^2\varphi_y  + |s|^{-3+\theta}.
\end{align*}
For the second term on the right-hand side of \eqref{e.f21}, we see that
\begin{align*}
|b|  \int \varepsilon^2 |y| \varphi_y &\lesssim 
 |s|^{-\frac 12 (1-\theta)}  \int_{|y|< |s|^{\frac 12(1+\theta)}}\varepsilon^2 \varphi_y
+|s|^{-1} \int_{|y|> |s|^{\frac 12(1+\theta)}} \varepsilon^2 |y \varphi_y|\\
& \lesssim |s|^{-\frac 12 (1-\theta)}  \int \varepsilon^2 \varphi_y
+|s|^{-1}\int_{|y|> |s|^{\frac 12(1+\theta)}} \frac{\varepsilon^2}{|y|}
\lesssim |s|^{-\frac 12 (1-\theta)}  \int \varepsilon^2 \varphi_y
+|s|^{-\frac 52 - \frac\theta2}.\end{align*}
Thus,
\[
  f_{2,1} = - \left(\frac{\lambda_s}{\lambda}+b\right) Z_1+  o(1)  \int \varepsilon^2 \varphi_y
+\mathcal O(|s|^{-3+\theta})\quad \hbox{where}\quad
Z_1 = \int \varepsilon^2 y \varphi_y \chi_1.
\]

\medskip

Second, we compute   $D^1 (\Lambda Q_b) + \Lambda Q_b \varphi - 3 Q_b^2 \Lambda Q_b$ to simplify the expression of $f_{2,2}$.
First, we claim that for any function $f$,
\begin{equation}\label{calculs.L}
\mathcal L(f')=(\mathcal Lf)'+6QQ'f,\quad \mathcal L(\Lambda f)=\frac 32 \mathcal Lf + y (\mathcal Lf)' +6yQQ'f-f + 3Q^2 f.
\end{equation}
The first identity follows directly from the definition of $\mathcal L$. For the second one, we proceed as follows
\begin{align*}
 \mathcal L(\Lambda f)& = \frac 12 \mathcal Lf + \mathcal H (yf')'+yf'-3Q^2 yf'
 =\frac 32 \mathcal Lf + \mathcal H (yf'') -f + 3Q^2 f +yf'-3Q^2 yf'\\
& =\frac 32 \mathcal Lf + y (\mathcal H (f'))' -f + 3Q^2 f +yf'-3Q^2 yf'\\
& =\frac 32 \mathcal Lf + y (\mathcal Lf)' -f + 3Q^2 f+6yQQ'f.
\end{align*}

Since $\mathcal L\Lambda Q = -Q$ and $Q_b = Q+bP_b$, we have
\begin{align*}
&D^1 (\Lambda Q_b) + \Lambda Q_b \varphi - 3 Q^2 \Lambda Q_b=
\mathcal L(\Lambda Q) + b \mathcal L(\Lambda P_b) + (\Lambda Q+b\Lambda P_b)(\varphi-1)\\
& = - Q + \frac 32 b  \mathcal L P_b + by(\mathcal LP_b)_y - bP_b +3bQ^2 P_b + 6 by QQ'P_b + (\Lambda Q+b\Lambda P_b)(\varphi-1).
\end{align*}
We use the notation from the proof of Lemma \ref{lprofile}, $\Psi_1 = (\mathcal LP_b - \mathcal LP)_y = (\mathcal LP_b)_y - \Lambda Q.$ Thus,
\begin{align*}
&D^1 (\Lambda Q_b) + \Lambda Q_b \varphi - 3 Q^2 \Lambda Q_b\\
& = - Q +  by\Lambda Q +by\Psi_1 +\frac 32 b  \mathcal L P_b - bP_b +3bQ^2 P_b + 6b y QQ'P_b + (\Lambda Q+b\Lambda P_b)(\varphi-1).
\end{align*}
Then,
\begin{align*}
&D^1 (\Lambda Q_b) + \Lambda Q_b \varphi - 3 Q^2 \Lambda Q_b\\
& = - Q +  by\Lambda Q +by\Psi_1 +\frac 32 b  \mathcal L_\varphi P_b - b\varphi P_b +3bQ^2 P_b + 6b y QQ'P_b + (\Lambda Q+by(P_b)_y)(\varphi-1).
\end{align*}
Next, we see that
\begin{align*}
- 3 (Q_b^2-Q^2) \Lambda Q_b = -6 b Q \Lambda Q P_b -b^2 \left(\frac 92 QP_b^2 + 6yQP_b(P_b)_y + 3yQ'P_b^2\right)
-3 b^3 P_b^2 \Lambda P_b.
\end{align*}
Combining these computations, we obtain
\begin{align*}
&D^1 (\Lambda Q_b) + \Lambda Q_b \varphi - 3 Q_b^2 \Lambda Q_b\\
& = - Q +  b \left(y\Lambda Q +\frac 32  \mathcal L_\varphi P_b +y P\chi_b' \varphi - P_b \varphi\right)+by(\Psi_1-P\chi_b')   + (\Lambda Q+byP'\chi_b)(\varphi-1)\\
&-b^2 \left(\frac 92 QP_b^2 + 6yQP_b(P_b)_y + 3yQ'P_b^2\right)
-3 b^3 P_b^2 \Lambda P_b  .
\end{align*}
We claim the following estimates
\begin{align}
& \| \mathcal L_\varphi \frac {\partial Q_b}{\partial b} - (\mathcal L_\varphi P_b + y P \chi_b'\varphi)\|_{L^2} \lesssim |s|^{-\frac 12},
\label{cl.1}\\
& \mathcal L_\varphi P_b - P_b \varphi = R_b,\label{cl.2}\\
& \|b y (\Psi_1 -  P \chi_b') \|_{L^2}\lesssim   |b|^{\frac32} ,\label{cl.3}\\
&   \|QP_b^2\|_{L^2} + \|yQP_b(P_b)_y \|_{L^2}+ \|yQ'P_b^2\|_{L^2}+ |b|^\frac 12 \|P_b^2 \Lambda P_b\|_{L^2} \lesssim1.\label{cl.4}
\end{align}
Indeed, by using \eqref{def.DQb},
\[ \mathcal L_\varphi \frac {\partial Q_b}{\partial b} - (\mathcal L_\varphi P + y P \chi_b'\varphi)  = D^1 (yP\chi_b')-3 Q^2 P \chi_b'
\]
and $\|D^1 (yP\chi_b')\|_{L^2} \lesssim \|P\chi_b'\|_{L^2}+\|yP'\chi_b'\|_{L^2}+\|yP\chi_b''\|_{L^2}\lesssim |b|^{\frac 12},$
$\|Q^2 P \chi_b'\|_{L^2}\lesssim |b|^{\frac 52}$, which proves \eqref{cl.1}.
Next, \eqref{cl.2} is a simple consequence of the definitions of $\mathcal L_\varphi$ and  $R_b$ in \eqref{defK} and \eqref{Rb}.
Then, estimate \eqref{cl.3} is proved as in \eqref{lprofile.205}. Finally, \eqref{cl.4} is straightforward.

\medskip

Inserting \eqref{cl.1}-\eqref{cl.4}, we obtain
\begin{align*}
&D^1 (\Lambda Q_b) + \Lambda Q_b \varphi - 3 Q_b^2 \Lambda Q_b\\
& = - Q +b R_b +  b  y\Lambda Q + b \mathcal L_\varphi \frac {\partial Q_b} {\partial b}  -\frac 12 \mathcal L_{\varphi} P_b + (\Lambda Q+byP'\chi_b)(\varphi-1) 
+\mathcal O_{L^2} (|s|^{-\frac 32}),
\end{align*}
and thus, using \eqref{modulation.6} and $\|\varepsilon\|_{L^2} \lesssim |s|^{-\frac 12}$,
\begin{align*}
&\int  (D^1 (\Lambda Q_b) + \Lambda Q_b \varphi - 3 Q_b^2 \Lambda Q_b) \varepsilon -\frac 12 \int \varepsilon^2 \varphi
\\
&\quad =  -\frac 12 F -\frac b2 K + b \int \varepsilon \mathcal L_\varphi \frac {\partial Q_b} {\partial b}
+b \int \varepsilon y \Lambda Q+ \int (\Lambda Q+byP'\chi_b)(\varphi-1)\varepsilon + \mathcal O(|s|^{-2}\ln |s|).
\end{align*}

Concerning $f_{2,3}$, by the definition of $V_\varphi$, one has
\begin{align*}
\int \frac{\partial Q_b}{\partial b} V_\varphi  = 
\int \varepsilon  \mathcal L_\varphi \frac{\partial Q_b}{\partial b} 
- \int  \frac{\partial Q_b}{\partial b} \left((6bQP_b+3b^2P_b^2) \varepsilon+3 Q_b \varepsilon^2+ \varepsilon^3\right)
\end{align*}
One sees easily by using \eqref{def.DQb} that
\[
\left| \int  \frac{\partial Q_b}{\partial b} \left((6bQP_b+3b^2P_b^2) \varepsilon+3 Q_b \varepsilon^2+ \varepsilon^3\right)
\right| \lesssim |s|^{-1} \mathcal N(\varepsilon) + \mathcal N(\varepsilon)^2
\lesssim (C^\star)^2 |s|^{-2+\theta},
\]
and thus by \eqref{BS.2}, since $b=-\lambda$,
\[
 f_{2,3} =- 2 b \left(\frac {\lambda_s}{\lambda}+b\right)\int  \varepsilon\mathcal L_\varphi \frac{\partial Q_b}{\partial b} 
 + \mathcal O( (C^\star)^3 |s|^{-4+\frac 32\theta}).
\]
 
For $f_{2,2}$ and $f_{2,3}$, we thus obtain
\begin{align*}
f_{2,2}+f_{2,3} 
&=  -\left(\frac{\lambda_s}{\lambda} + b\right)  (F + b K  ) 
+2 b \left(\frac{\lambda_s}{\lambda} + b\right) \int \varepsilon y \Lambda Q\\
&+2 \left(\frac{\lambda_s}{\lambda} + b\right) \int (\Lambda Q+byP'\chi_b)(\varphi-1)\varepsilon  
+ \mathcal O( |s|^{-3+\theta}).
\end{align*} 
We split the third term in the right-hand side as follows
\begin{align*}
&2 \left(\frac{\lambda_s}{\lambda} + b\right) \int (\Lambda Q+byP'\chi_b)(\varphi-1)\varepsilon 
\\&= -  \left(\frac{\lambda_s}{\lambda} + b\right)Z_2 + 
2 \left(\frac{\lambda_s}{\lambda} + b\right) \int \left(\Lambda Q\chi\left( y \frac4B |s|^{-\theta}\right)+ byP'\chi_b\right)(\varphi-1)\varepsilon
\end{align*}
where $Z_2   = - 2 \int  \Lambda Q(\varphi-1)\varepsilon \chi_2$.

\medskip

%We will use the following properties of $\varphi$:
% \begin{equation}\label{inter}
%  \hbox{ for $y >- \frac B2 |s|^\theta$,}\quad |\varphi(y)-1|=|\varphi(y)-\varphi(0)|
 % \lesssim |y| \sup_{y >- \frac B2 |s|^\theta} |\varphi_y| \lesssim |y| |s|^{-2\theta}. \end{equation}

It follows from \eqref{phi_B} and $|\Lambda Q|\lesssim (1+|y|)^{-2}$ that
\[
\left| \int  \Lambda Q\chi\left( y \frac4B |s|^{-\theta}\right)(\varphi-1)\varepsilon \right|
\lesssim |s|^{-2\theta}\int_{y>- \frac B2|s|^{\theta}} \frac {|\varepsilon|}{1+|y|}
\lesssim |s|^{-2\theta} \|\varepsilon\|_{L^2} \lesssim |s|^{-2\theta-\frac 12}.
\]
Thus, by \eqref{BS.2}, since $\theta>3/5$, for $|S_0|$ large enough,
\[
\left|\left(\frac{\lambda_s}{\lambda} + b\right) \int  \Lambda Q\chi\left( y \frac4B |s|^{-\theta}\right)(\varphi-1)\varepsilon \right|
\lesssim C^\star |s|^{-\frac 32 -\frac 32\theta} \lesssim |s|^{-3+\theta}.
\]
Next, since $P'\in \mathcal Y_2$, using again \eqref{phi_B},
\begin{align*}
\left| \int byP'\chi_b (\varphi-1)\varepsilon\right|&\lesssim 
  |s|^{-1-2\theta} \int_{|y|< \frac B2|s|^{\theta}} |\varepsilon|
+|s|^{-1} \int_{|y|> \frac B2|s|^{\theta}} \frac{|\varepsilon|}{|y|} \\
&\lesssim  |s|^{-1-\frac 32 \theta} \|\varepsilon\|_{L^2} + |s|^{-1-\frac  \theta2} \|\varepsilon\|_{L^2}
\lesssim |s|^{-\frac 32-\frac \theta2}.
\end{align*}
Thus, by \eqref{BS.2}, since $\theta>\frac 12$, we obtain
\[
\left|\left(\frac{\lambda_s}{\lambda} + b\right) \int byP'\chi_b (\varphi-1)\varepsilon\right|
\lesssim 
  C^\star |s|^{-\frac 52}\lesssim |s|^{-3+\theta}.
\]
 
\medskip

In conclusion for   $f_{2,2}$ and $f_{2,3}$, we   obtain
\begin{align*}
f_{2,2}+f_{2,3} 
 =  -\left(\frac{\lambda_s}{\lambda} + b\right)  (F + b K +Z_2 ) 
+2 b \left(\frac{\lambda_s}{\lambda} + b\right) \int \varepsilon y \Lambda Q 
+ \mathcal O( |s|^{-3+\theta}).
\end{align*} 
where $Z_2   = - 2 \int  \Lambda Q(\varphi-1)\varepsilon \chi_2$.

\medskip

For $f_{2,4}$, we claim
\begin{equation}\label{cl.30}
f_{2,4} = b \left(   \int \varepsilon^2 \varphi  -6 \int \varepsilon^2 Q \Lambda Q \right) +\mathcal O(|s|^{-3+\theta}).
\end{equation} 
 Indeed, we check that
\[
|b|  \int \varepsilon^2 |Q_b \Lambda Q_b-Q\Lambda Q| \lesssim |b|^2 \int |\varepsilon|^2 \lesssim |s|^{-3},
\]
and
\[
|b| \left|\int \varepsilon^3 \Lambda Q_b \right|\lesssim |b|\int |\varepsilon|^3\lesssim |b| \left(\int |\varepsilon|^2 \right)\left(\int |D^{\frac 12} \varepsilon|^2\right)^{\frac 12} \lesssim C^\star |s|^{-3 + \frac \theta 2}.
\] 
 
\medskip

In conclusion for $f_2$, we obtain
\begin{align} \label{eq:f2}
f_2 
&=  -\left(\frac{\lambda_s}{\lambda} + b\right)  (F + b K +Z )\nonumber \\
& +b \left( \int \varepsilon^2 \varphi  -6 \int \varepsilon^2 Q  \Lambda Q   
+ 2 \left(\frac{\lambda_s}{\lambda} + b\right) \int \varepsilon y \Lambda  Q\right)\\
&+  o(1)  \int \varepsilon^2 \varphi_y+\mathcal O(|s|^{-3+\theta}). \nonumber
\end{align}

\medskip

\noindent \emph{Estimate for $f_3$.}
By integration by parts, we have
\begin{align*}
2  \int (Q_b+\varepsilon)_y V_\varphi & = 2  \int (Q_b+\varepsilon)_y (D^1 \varepsilon + \varepsilon \varphi - 3 Q_b^2 \varepsilon - 3 Q_b \varepsilon^2 - \varepsilon^3)\\
& = 2 \int \varepsilon \left(D^1 (Q_b)_y + (Q_b)_y \varphi - 3 Q_b^2 (Q_b)_y \right)
 -   \int \varepsilon^2 \varphi_y .
\end{align*}
As before, using \eqref{calculs.L}, since $\mathcal LQ'=0$,
\begin{align*}
D^1 (Q_b)_y + (Q_b)_y \varphi - 3 Q^2 (Q_b)_y &= \mathcal L Q' + b \mathcal L (P_b)_y  + (Q+bP_b)_y (\varphi-1)\\
& = b (\mathcal LP_b)_y + 6bQQ'P_b + (Q+bP_b)_y (\varphi-1)\\
&= b \Lambda Q + b \Psi_1 + 6bQQ'P_b + (Q+bP_b)_y (\varphi-1),
\end{align*}
and
\begin{align*}
- 3 (Q_b^2-Q^2) (Q_b)_y = - 6 bQ Q' P_b -   b^2 (3Q' P_b^2+6QP_b(P_b)_y)  - 3b^3 P_b^2 (P_b)_y.
\end{align*}
Thus,
\begin{align*}
D^1 (Q_b)_y + (Q_b)_y \varphi - 3 Q_b^2 (Q_b)_y &= b \Lambda Q + b \Psi_1 + (Q+bP_b)_y (\varphi-1)\\ &\quad-   b^2 (3Q' P_b^2+6QP_b(P_b)_y)   - 3b^3 P_b^2 (P_b)_y.
\end{align*}
We claim
\begin{align}
&\|b \Psi_1\|_{L^2} \lesssim |s|^{-\frac 32},\label{cl.10}\\
&\|Q'P_b^2\|_{L^2} + \|P_b^2 (P_b)_y\|_{L^2}+ \|QP_b(P_b)_y\|_{L^2} \lesssim 1, \label{cl.11}\\
& \int \varepsilon \Lambda Q = \int \varepsilon y Q' + \mathcal O\big((C^\star)^2|s|^{-2+\theta}\big),\label{cl.12}\\
& \|Q' (\varphi-1)\|_{L^2} + |b| \|(P_b)_y (\varphi-1)\|_{L^2} \lesssim |s|^{-2 \theta}. \label{cl.13}
\end{align}
Indeed, \eqref{cl.10} is a consequence of \eqref{lprofile.205} and \eqref{lprofile.206},
\eqref{cl.11} is a direct consequence of the definition of $P_b$ and \eqref{cl.12} follows from 
$\Lambda Q = \frac 12 Q+ yQ'$ and \eqref{BS.4}.  
To prove \eqref{cl.13}, first we apply \eqref{phi_B}, so that
\[
\|Q' (\varphi-1)\|_{L^2}^2\lesssim \int \frac {|\varphi-1|^2}{1+y^6} dy\lesssim
|s|^{- 4 \theta}\int_{|y| < \frac B2 |s|^\theta} \frac {dy}{1+y^4} + 
\int_{|y| > \frac B2 |s|^\theta} \frac {dy}{1+y^6}\lesssim |s|^{-4 \theta}.
\]
Second, since $P\in \mathcal Z$,
\begin{align*}
\|{ |b|}(P_b)_y (\varphi-1)\|_{L^2}^2&\lesssim { |b|^2}\int \frac {|\varphi-1|^2}{1+y^4} dy+  { |b|^2}\int |\chi_b'|^2
\\ &\lesssim
|s|^{{  -2}- 4 \theta}\int_{|y| < \frac B2 |s|^\theta} \frac {dy}{1+y^2} +{ |s|^{-2}}
\int_{|y| > \frac B2 |s|^\theta} \frac  {dy}{1+y^4}+ |s|^{{ -3}} \lesssim  |s|^{ -3 }.
\end{align*}

\medskip

Thus, using also \eqref{BS.2}, we obtain, for $S_0$ large (possibly depending on $C^\star$),
using $\theta>\frac 35$,
\begin{align} \label{eq:f3}
f_3 = 2 b\left(\frac{x_s}{\lambda}-1\right)  \int \varepsilon y Q'
+o(1) \int \varepsilon^2 \varphi_y +\mathcal O(|s|^{-3+  \theta}).
\end{align}

\medskip

\noindent \emph{Estimate for $f_4$.} Recall that
\begin{align*}
f_4 & = 2 \int \Psi_b V_\varphi
=2 \int \Psi_b \left(D^1 \varepsilon + \varepsilon \varphi - \left((Q_b+\varepsilon)^3 - Q_b^3\right)\right).
\end{align*}
First, using \eqref{lprofile.3},
\begin{align*}
\left| \int \Psi_b D^1 \varepsilon \right| & \lesssim 
\left| \int (D^{\frac 12}\Psi_b) (D^{\frac 12} \varepsilon) \right|\lesssim
\|D^{\frac 12}\Psi_b\|_{L^2} \|D^{\frac 12}\varepsilon\|_{L^2}\lesssim |s|^{-2}|\ln |s||^{\frac12} \mathcal N(\varepsilon). 
\end{align*}
Second, by \eqref{lprofile.2}, we have
\begin{equation} \label{PsibVarphi}
 \left(\int \Psi_b^2 \varphi \right)^{\frac 12}\lesssim 
 |b|^2 + |b| \left( \int P^2 (\chi_b')^2 \varphi \right)^{\frac 12} + |b|^2 \left( \int P_b^2\varphi \right)^{\frac 12} \lesssim |b|^{2-\frac \theta 2}.
\end{equation}
Thus,
\begin{align*}
\left| \int \Psi_b \varepsilon \varphi \right| & \lesssim 
\left(\int \Psi_b^2 \varphi\right)^{\frac 12}\left(\int \varepsilon^2 \varphi\right)^{\frac 12}
\lesssim |s|^{-2+\frac \theta 2}\mathcal N(\varepsilon)\lesssim C^\star |s|^{-3+\theta}.  
\end{align*} 
Using \eqref{gnp} (with $p=6$), \eqref{lprofile.2} and \eqref{BS.1}, we also have
\begin{displaymath}
\left| \int \Psi_b \varepsilon^3  \right| \le \|\Psi_b\|_{L^2}\|\varepsilon\|_{L^6}^3 \lesssim 
|s|^{-\frac32}\|\varepsilon\|_{L^2} \mathcal{N}(\epsilon)^2 \lesssim (C^{\star})^2|s|^{-4+\theta}  .
\end{displaymath}
Estimating the other terms $\int \Psi_bQ_b ^2 \varepsilon$ and $\int \Psi_b Q_b \varepsilon^2$ similarly, we obtain
\begin{equation} \label{eq:f4}
|f_4|\lesssim |s|^{-2+\frac \theta 2}\mathcal N(\varepsilon)\lesssim C^\star |s|^{-3+\theta}.
\end{equation}

\medskip

\noindent \emph{Estimate for $f_5$.}
\begin{align*}
f_5 & = -2 \int (Q_b)_s \left( (Q_b+\varepsilon)^3 - Q_b^3 - 3 Q_b^2 \varepsilon\right)
 + \int \varepsilon^2  \partial_s \varphi = f_{5,1}+f_{5,2}.
\end{align*}
First, by \eqref{def.DQb} and \eqref{BS.3},
\[
|(Q_b)_s| = |b_s| \left| \frac{\partial Q_b}{\partial b}\right| \lesssim (|s|^{-2} + |s|^{-1}\mathcal N(\varepsilon))
\lesssim C^\star |s|^{-2 +\frac \theta 2}.
\]
Thus,
\begin{align*}
|f_{5,1}|&\lesssim C^\star |s|^{-2 +\frac \theta 2} \int |\varepsilon|^2 (Q+|b|) + |\varepsilon|^3
\lesssim C^\star |s|^{-2 +\frac \theta 2} \left( \mathcal N(\varepsilon)^2   + |s|^{-\frac 32}\right)\\
&\lesssim (C^\star)^3 |s|^{-4 +\frac 32\theta }\lesssim |s|^{-3+\theta}.
\end{align*}

Second, we see from the definition of $\varphi$ in \eqref{def.varphi}
\begin{align} \label{varphi_s}
\partial_s \varphi & = \theta s^{-1} |s|^\theta  \frac {\phi'\left(\frac y B + |s|^\theta\right)}{\phi(|s|^\theta)}
- \theta s^{-1} |s|^\theta \frac {\phi\left(\frac y B + |s|^\theta\right)\phi'(|s|^\theta)}{\phi^2(|s|^\theta)}.
\end{align}
We also have that
\[
|s|^{-1+\theta} \int \varepsilon^2  \frac {\phi'\left(\frac y B + |s|^\theta\right)}{\phi(|s|^\theta)}
\lesssim |s|^{-1+\theta} \int \varepsilon^2 \varphi_y,
\]
and 
\[
|s|^{-1+\theta} \int \varepsilon^2  \frac {\phi\left(\frac y B + |s|^\theta\right)\phi'(|s|^\theta)}{\phi^2(|s|^\theta)}
\lesssim |s|^{-1-\theta} \int \varepsilon^2 \varphi \lesssim |s|^{-1-\theta} (\mathcal N(\varepsilon))^2.\]

Thus,
\begin{equation} \label{eq:f5}
f_5  = o(1) \int \varepsilon^2 \varphi_y + \mathcal O(|s|^{-3+\theta})
 .
\end{equation}

Therefore, we conclude the proof of \eqref{eq:F} gathering \eqref{eq:f1}, \eqref{eq:f2}, \eqref{eq:f3}, \eqref{eq:f4} and \eqref{eq:f5}.
\end{proof}

\subsection{Virial-type estimate} 
Now, to extend to the cubic Benjamin-Ono equation the technique developed in \cite{RaSz}, we prove the following   suitably localized Virial-type identity.
\begin{lemma}\label{le:GG}
For $|S_0|$ large enough possibly depending on $C^{\star}$, for all $s\in \mathcal I^\star$, 
\begin{equation}\label{bg:G}
|G|\lesssim |s|^{-\frac 12} 
\end{equation}
and
\begin{equation}\label{eq:G}\begin{aligned}
\frac {d G}{ds} &=-2 \int \big|D^{\frac 12} (\varepsilon \rho)\big|^{2} -\int (\varepsilon  \rho)^2 
+ 6\int Q^2 (\varepsilon  \rho)^2  -6 \int Q\Lambda Q \varepsilon^2
\\ &\quad +2 \left(\frac{\lambda_s}{\lambda} + b\right) \int \varepsilon y \Lambda  Q+2 \left(\frac{x_s}{\lambda} -1\right) \int \varepsilon y   Q'
 \\ & \quad+ \mathcal O(|s|^{-2+\theta})
 +  \mathcal O(|s|^{-\frac 1{20}} \|D^{\frac 12}(\varepsilon \rho)\|_{L^2}^2)\, ,
 \end{aligned}
\end{equation}
where
\begin{equation}\label{def.rho}
\rho(s,y) =  \frac {\chi(y|s|^{-\frac 35})}{\left(1+y^2 |s|^{-1}\right)^{\frac 12}}.
\end{equation}
\end{lemma}
\begin{proof}
The bound \eqref{bg:G} follows from $\|\varepsilon\|_{L^2}\lesssim |s|^{-1}$ (see \eqref{BS.1}) and the 
bound  $|\psi|\lesssim |s|^{\frac 12}$ coming directly from the definition of $\psi$ in \eqref{defG}.

\medskip

Now, we prove \eqref{eq:G}.
We compute using \eqref{modulation.4},
\[
\frac { dG}{ds} = 2 \int  \varepsilon_s \varepsilon \psi+ \int \varepsilon^2 \psi_s= g_1 + g_2 +g_3+g_4+g_5+g_6+g_7,
\]
where
\begin{align*}
	&g_1 = 2 \int V_y \varepsilon \psi\,,&
	g_2 &= 2\frac{\lambda_s}{\lambda}\int (\Lambda \varepsilon) \varepsilon \psi\,, \\
	&g_3 = 2\left(\frac{\lambda_s}{\lambda}+b\right) \int \Lambda Q_b \varepsilon \psi\,, &
	g_4 &= 2 \left(\frac {x_s}\lambda -1\right) \int (Q_b+\varepsilon)_y \varepsilon\psi\,,\\
	&g_5 = -2 (b_s+b^2) \int \frac{\partial Q_b}{\partial b} \varepsilon \psi\, , &
	g_6& =  2\int \Psi_b \varepsilon\psi ,\\& g_7= \int \varepsilon^2 \psi_s .
\end{align*}
 
We claim that the following technical facts on $\psi$ and $\rho$.
\begin{lemma}\label{le:psi} The following hold.

\noindent{\rm (i)} Pointwise estimates.
\begin{equation}\label{bd.psi}
|\psi(y)-y|\lesssim |s|^{-\frac 12} |y|^2,\quad |\psi|\lesssim |s|^{\frac 12},
\end{equation}
 \begin{equation}\label{e.g10}
|y \psi_y|\lesssim |s|^{\frac 12} \mathbf{1}_{y>-2 |s|^{\frac 35}},
\end{equation}
\begin{equation}\label{e.g14}
\psi_y =  
\frac {\chi^2(y|s|^{-\frac 35})}{1+y^2 |s|^{-1}}  +  \mathbf{1}_{-2|s|^{\frac 35} <y<-|s|^{\frac 35}}\mathcal O( |s|^{-\frac 1{10}}),
 \end{equation}
\begin{equation}\label{e.g15}
|\psi_y-1|+|\rho-1|\lesssim \frac {y^2 |s|^{-1}}{1+y^2 |s|^{-1}} ,\quad |\psi_y|\lesssim 1   ,\quad
|\psi_{yy}|\lesssim |s|^{-\frac 12},
\end{equation}
\begin{equation}\label{e.g17}
|\psi_{yyy}|\lesssim   
 \frac {|s|^{-1}}{(1+y^2 |s|^{-1})^2}+
|s|^{-\frac {11}{10}} \mathbf{1}_{-2|s|^{ \frac 35}<y<-|s|^{\frac 35}} \lesssim |s|^{-1} .
\end{equation}
{\rm (ii)} Norm estimates.
\begin{equation}\label{e.g16}
\|\rho\|_{L^2}+\|\psi_y\|_{L^2} \lesssim |s|^{\frac 14},\quad \|\rho_y\|_{L^2}+\|\psi_{yy}\|_{L^2} \lesssim |s|^{-\frac 14}.
\end{equation}
\end{lemma}
\begin{proof}
By  Taylor expansion, since $\phi(0)=\frac 12$, $\phi'(0)=\frac 1{\pi}$ and $\sup_{\mathbb{R}} |\phi''|<+\infty$,
one finds for $|y|<|s|^{\frac 12}$,
\[
\Big|\phi( |s|^{-\frac 12} y) - \frac 12 - \frac 1{\pi}  |s|^{-\frac 12}  y\Big| \lesssim |s|^{-1}|y|^2.
\]
The definition of $\psi$ in \eqref{defG} then implies \eqref{bd.psi} for $|y|<|s|^{\frac 12}$.
For $|y|>|s|^{\frac 12}$, \eqref{bd.psi} is a consequence of $|\psi|\lesssim |s|^{\frac 12}$.
Next, note that by the definition of $\psi$ in \eqref{defG},
\begin{equation}\label{e.001}\begin{aligned}
\psi_y 
&= \pi \phi'(y|s|^{-\frac 12}) \chi^2(y|s|^{-\frac 35}) + |s|^{-\frac 1{10}} \pi \Big(\phi(y|s|^{-\frac 12})-\frac 12\Big) (\chi^2)'(y |s|^{-\frac 35})\\
&= \frac {\chi^2(y|s|^{-\frac 35})}{1+y^2 |s|^{-1}}  + |s|^{-\frac 1{10}} \pi \Big(\phi(y|s|^{-\frac 12})-\frac 12\Big) (\chi^2)'(y |s|^{-\frac 35}),
\end{aligned}\end{equation}
which implies directly \eqref{e.g10} and \eqref{e.g14}. Moreover,
\[
|\psi_y-1|\lesssim \frac {y^2 |s|^{-1}}{1+y^2 |s|^{-1}} \mathbf{1}_{y>-|s|^{\frac 35}} + 
|s|^{-\frac 1{10}} \mathbf{1}_{y<-|s|^{\frac 35}}  .
\]
Differentiating \eqref{e.001}, we have
\begin{align*}
\psi_{yy} 
&= |s|^{-\frac 12} \pi   \phi''\left(y|s|^{-\frac 12}\right) \chi^2(y|s|^{-\frac 35})
+ 2 |s|^{-\frac 35} \pi   \phi'\left(y|s|^{-\frac 12}\right) (\chi^2)'(y|s|^{-\frac 35})\\ &
+ |s|^{-\frac 7{10}} \pi \left(\phi\left(y|s|^{-\frac 12}\right) -\frac 12\right)(\chi^2)''(y|s|^{-\frac 35}).\end{align*}
and thus $|\psi_{yy}|\lesssim |s|^{-\frac 12}$.
This proves  \eqref{e.g15}.
Now, we estimate $\psi_{yyy}$. By direct computations, we have
\begin{align*}
\psi_{yyy} & = 
|s|^{-1} \pi   \phi'''\left(y|s|^{-\frac 12}\right) \chi^2(y|s|^{-\frac 35})
+ 3 |s|^{-\frac 12 - \frac 35} \pi   \phi''\left(y|s|^{-\frac 12}\right) (\chi^2)'(y|s|^{-\frac 35})\\ &
+ 3 |s|^{-\frac 65} \pi \phi'\left(y|s|^{-\frac 12}\right) (\chi^2)''(y|s|^{-\frac 35})
+ |s|^{-\frac 95 +\frac 12} \pi\left(  \phi\left(y|s|^{-\frac 12}\right)-\frac 12\right)  (\chi^2)'''(y|s|^{-\frac 35}).
\end{align*}
Thus,
\[
|\psi_{yyy}|\lesssim 
 \frac {|s|^{-1}}{(1+y^2 |s|^{-1})^2}+
|s|^{-\frac {11}{10}} \mathbf{1}_{-2|s|^{-\frac 35}<y<-|s|^{\frac 35}} \lesssim |s|^{-1}.
\]

Finally, we prove \eqref{e.g16}. 
First, 
\[
\int \rho^2 \lesssim \int \frac{dy}{1+y^2 |s|^{-1}} \lesssim |s|^{\frac 12}.
\]
Note that by direct computation
\[
\rho_y = |s|^{-\frac 35} \frac{\chi'(y|s|^{-\frac 35})}{(1+y^2 |s|^{-1})^{\frac 12}}
- |s|^{-\frac 12} \frac {y |s|^{-\frac 12} \chi(y|s|^{-\frac 35})}{(1+y^2 |s|^{-1})^{\frac 32}}
\]
and thus,  
\[
\int (\rho_y)^2 \lesssim |s|^{-\frac 65} \int \frac{dy}{1+y^2 |s|^{-1}} + |s|^{-1} \int  \frac{y^2 |s|^{-1} dy}{(1+y^2 |s|^{-1})^{3}}  \lesssim  |s|^{-\frac 12}.
\]
By \eqref{e.001}, we have $\int (\psi_y-\rho)^2 \lesssim |s|^{-\frac 15} \int_{-2|s|^{\frac 35} <y<-|s|^{\frac 35}} dy\lesssim |s|^{\frac 25}$ and
\[
|(\psi_y-\rho)_y|\lesssim |s|^{-\frac 35 }\mathbf{1}_{-2|s|^{\frac 35}<y<-|s|^{\frac 35}}\quad \hbox{and so}\quad
\int |(\psi_y-\rho)_y|^2 \lesssim |s|^{-\frac 35 },
\]
which finishes the proof of \eqref{e.g16}.
\end{proof} 

\noindent \emph{Estimate for $g_1$.}  
We claim   
\begin{equation}\label{e.g1}
\begin{split}
g_1&=  -2 \int \big|D^{\frac 12} (\varepsilon \rho)\big|^{2} -\int (\varepsilon \rho)^2+ 
{6\int Q^2 (\varepsilon  \rho)^2  -6 \int Q\Lambda Q \varepsilon^2}\\ & \quad+\mathcal O(|s|^{-2+\theta})
+  \mathcal O( |s|^{-\frac 1{20}} \|D^{\frac 12}(\varepsilon \rho)\|_{L^2}^2). 
 \end{split}
\end{equation}
By using the definition of $V$ in \eqref{defV} and integrations by parts, we decompose $g_1$ as
\[ g_1=g_{1,1}+g_{1,2}+g_{1,3} \, ,
\]
where
\begin{align*}
	g_{1,1} & =  2 \int (D^1 \varepsilon)_y \varepsilon  \psi ,\\
	g_{1,2} &=- \int \varepsilon^2 \psi_y ,\\
	g_{1,3}&=
	-\frac 12 \int \big((Q_b+\varepsilon)^4 - Q_b^4 - 4Q_b^3 \varepsilon\big) \psi_y
	+{ 2} \int \big((Q_b+\varepsilon)^3 - Q_b^3  \big) \varepsilon \psi_y
	\\ & \quad -2 \int \big((Q_b+\varepsilon)^3-Q_b^3 - 3 Q_b^2 \varepsilon\big) (Q_b)_y \psi .
\end{align*}
 
First, concerning $g_{1,1}$, we claim the following two estimates
\begin{equation}\label{e:g02}
\left| \int (D^1 \varepsilon) \varepsilon_y \psi \right|\lesssim |s|^{-2+\theta},
\end{equation}
\begin{equation}\label{e:g03}
\left| \int  (D^1 \varepsilon)  \varepsilon \psi_y -\int \big|D^{\frac 12} (\varepsilon\rho)\big|^2 \right|\lesssim 
|s|^{-2+\theta} +   |s|^{-\frac 1{20}} \|D^{\frac 12}(\varepsilon \rho)\|_{L^2}^2.
\end{equation} 
{With these estimates in hand, we obtain after integration by parts that 
\begin{equation} \label{e.g11}
\begin{split}
g_{1,1}&=-2\int (D^1\varepsilon) \varepsilon_y \psi-2\int (D^1\varepsilon) \varepsilon \psi_y \\ &
=-2\int \big|D^{\frac 12} (\varepsilon\rho)\big|^2+\mathcal{O}(|s|^{-2+\theta}) 
+  \mathcal O( |s|^{-\frac 1{20}} \|D^{\frac 12}(\varepsilon \rho)\|_{L^2}^2)\, .
\end{split}
\end{equation}}

\noindent \textit{Proof of \eqref{e:g02}}.
Using \eqref{m.1} with $a=\psi$, since $\|\psi_{yy}\|_{L^\infty} \lesssim |s|^{-\frac 12}$
(see \eqref{e.g15}), we obtain
from \eqref{BS.1},
\[
\left| \int (D^1 \varepsilon) \varepsilon_y \psi \right|\lesssim 
\| \varepsilon\|_{L^2}^2 \|\psi_{yy}\|_{L^\infty}   
\lesssim  |s|^{-\frac 32} \lesssim |s|^{-2+\theta}.
\]
\medskip

\noindent \textit{Proof of \eqref{e:g03}}. Using \eqref{m.2} with $a=\psi$,
since $\|a''\|_{L^2} \lesssim |s|^{-\frac 14}$ and $\|a'\|_{L^2} \lesssim |s|^{\frac 14}$,
we have
\begin{align*}
\left| \int  (D^1 \varepsilon)  \varepsilon \psi_y - \int  |D^{\frac 12} \varepsilon|^2\psi_y \right|
&\lesssim  
   \|\varepsilon\|_{\dot H^{\frac 12}}^{\frac 32} \|\varepsilon\|_{L^2}^{\frac 12} 
\| \psi_{yy}\|_{L^2}^{\frac 34} \|\psi_y\|_{L^2}^{\frac 14}\\
&\lesssim (C^\star)^{\frac 32} |s|^{-\frac {15}8 + \frac 34 \theta}\lesssim |s|^{-2+\theta},
\end{align*}
since $\theta>\frac 12$.
By \eqref{e.001} and the definition of $\rho= \frac {\chi(y|s|^{-\frac 35})}{(1+y^2 |s|^{-1})^{\frac12}}$,
 we see that $|\Psi_y - \rho^2 |\lesssim |s|^{-\frac 1{10}} \mathbf{1}_{-2|s|^{\frac 35}<y<-|s|^{\frac 35}} $, and thus
 \begin{align*}
\left|  \int  |D^{\frac 12} \varepsilon|^2\psi_y - \int  |D^{\frac 12} \varepsilon|^2\rho^2\right|
&\lesssim |s|^{-\frac 1{10}} \|D^{\frac 12} \varepsilon\|_{L^2}^2\lesssim |s|^{-\frac 1{10}} (\mathcal N(\varepsilon))^2
\lesssim |s|^{-2+\theta}.
\end{align*}
Now, we claim 
\begin{equation}
\label{e:g003}
\left| \int  | D^{\frac 12} \varepsilon|^2\rho^2  -\int |D^{\frac 12} (\varepsilon\rho)|^2 \right|\lesssim 
|s|^{-2+\theta}+   |s|^{-\frac 1{20}} \|D^{\frac 12}(\varepsilon \rho)\|_{L^2}^2,
\end{equation}
which is sufficient to finish the proof of \eqref{e:g03}.
Indeed,  using \eqref{COMM1} and \eqref{gnp}, 
\begin{align*}
\left|\int  |D^{\frac 12}  \varepsilon|^2\rho^2  -\int |D^{\frac 12} (\varepsilon\rho)|^2\right|
& = \left|\int \left( (D^{\frac 12} \varepsilon)\rho + D^{\frac 12} (\varepsilon\rho)\right)
\left( (D^{\frac 12} \varepsilon)\rho  - D^{\frac 12} (\varepsilon\rho)\right)\right| \\
& \lesssim \left(\|(D^{\frac 12} \varepsilon) \rho\|_{L^2} + \|D^{\frac 12}(\varepsilon \rho)\|_{L^2} \right)
\|[D^{\frac 12},\rho] \varepsilon \|_{L^2}\\
&  \lesssim \left(\|D^{\frac 12} \varepsilon \|_{L^2} + \|D^{\frac 12}(\varepsilon \rho)\|_{L^2} \right) 
\|\varepsilon\|_{\dot H^{\frac 12}}^{\frac 12} \|\varepsilon\|_{L^2}^{\frac 12} \|\rho_y\|_{L^2}^{\frac 34} \|\rho\|_{L^2}^{\frac 14} 
\\& \lesssim |s|^{-2+\theta} +  (C^{\star})^{\frac12}|s|^{-\frac78+\frac \theta 4} \|D^{\frac 12}(\varepsilon \rho)\|_{L^2} 
\\& \lesssim |s|^{-2+\theta} +   |s|^{-\frac 1{20}} \|D^{\frac 12}(\varepsilon \rho)\|_{L^2}^2 .
\end{align*} 
 
\medskip
 
Second, we see from \eqref{e.g14}, 
\begin{align}
- g_{1,2} 
&= \int \varepsilon^2 \psi_y = \int \varepsilon^2 \frac {\chi^2(y|s|^{-\frac 35})}{1+y^2 |s|^{-1}}  
+\mathcal O( |s|^{-\frac 1{10}})\int_{-2|s|^{\frac 35} <y<-|s|^{\frac 35}} \varepsilon^2 \nonumber\\
&= \int \varepsilon^2 \frac {\chi^2(y|s|^{-\frac 35})}{1+y^2 |s|^{-1}}  
+\mathcal O((C^\star)^2 |s|^{-\frac {21}{10}+\theta})\nonumber\\&= \int {\varepsilon^2} \rho^2  
+\mathcal O( |s|^{-  2+\theta}), \label{e.g12}
\end{align}
since (by $\theta>\frac 35$), $\int_{-2|s|^{\frac 35} <y<-|s|^{\frac 35}} \varepsilon^2\lesssim \mathcal N(\varepsilon)^2\lesssim
(C^\star)^2 |s|^{-2+\theta}$.

\medskip

Next, we claim
\begin{equation}
g_{1,3}  = {6} \int Q^2 (\varepsilon\rho)^2  -6 \int Q \Lambda Q \varepsilon^2 + \mathcal O(|s|^{-2+\theta})\label{e.g31}. 
\end{equation}
We start with rough bounds, using \eqref{bd.psi} (in particular $|\psi Q'|\lesssim 1$), and \eqref{e.g15}  , and then \eqref{gnp},
\begin{align*}
&\left| g_{1,3} - {3}\int  \psi_y  Q^2 \varepsilon^2 + 6 \int \psi QQ' \varepsilon^2\right|\\
& \lesssim \int |\psi_y| \left( |b|  |\varepsilon|^2  +|\varepsilon|^3 + |\varepsilon|^4\right) 
+\int |\psi| \left( |b||\varepsilon|^2 + (|Q'|+|b|) |\varepsilon|^3 \right)\\
& \lesssim |s|^{-\frac 12} \int \varepsilon^2 + \int |\varepsilon|^3 + \int |\varepsilon|^4
\lesssim |s|^{-\frac 32}\lesssim |s|^{-2+\theta}.
\end{align*}
Now, using \eqref{e.g15}, 
\[
\left|\int  \psi_y  Q^2 \varepsilon^2 - \int   Q^2 \varepsilon^2 \right| \lesssim 
|s|^{-1} \int \varepsilon^2 \lesssim |s|^{-2} \lesssim |s|^{-2 +\theta},
\]
and, using \eqref{bd.psi}, 
\[
\left|\int \psi QQ' \varepsilon^2-\int y QQ' \varepsilon^2 \right| \lesssim 
|s|^{-\frac 12} \int \varepsilon^2 \lesssim |s|^{-\frac 32} \lesssim |s|^{-2 +\theta}.
\]
Thus,
\begin{align*}
g_{1,3}& =  {3}\int   Q^2 \varepsilon^2 -6 \int y QQ' \varepsilon^2+\mathcal O(|s|^{-2 +\theta})
\\ &={6} \int Q^2 \varepsilon^2-6 \int Q\Lambda Q \varepsilon^2 +\mathcal O(|s|^{-2 +\theta}).
\end{align*}
Hence, to finish the proof of \eqref{e.g31}, we only have to prove
\begin{equation}\label{bof2}
 \int Q^2 \varepsilon^2 |1-\rho^2|  \lesssim |s|^{-2 +\theta},
\end{equation}
but similarly   as before, this follows  from $|\rho^2-1|\lesssim y^2 |s|^{-1}$.

Therefore, we conclude the proof of \eqref{e.g1} gathering \eqref{e.g11}, \eqref{e.g12} and \eqref{e.g31}.

\medskip
\noindent \emph{Estimate for $g_2$.}
We claim   
\begin{equation}\label{e.g2}
g_2=\mathcal O(|s|^{-2+\theta}). 
\end{equation}
Indeed, integrating by parts,
\[
g_2 = - \frac {\lambda_s}{\lambda} \int \varepsilon^2 y \psi_y
\]
Note that since $\theta>\frac 35$, $\varphi \gtrsim 1$ for $y>-2|s|^{\frac 35}$, and thus $\int_{y>-2 |s|^{\frac 35}} |\varepsilon|^2 \lesssim \mathcal N(\varepsilon)^2 \lesssim (C^\star)^2 |s|^{-2+\theta}$. Combining this, \eqref{e.g10}  and \eqref{BS.2}, we find
$
|g_2| \lesssim (C^\star)^3 |s|^{-\frac 52+\frac 32 \theta} 
$, which implies \eqref{e.g2} since $\theta<1$.

\medskip
\noindent \emph{Estimate for $g_3$.}
We claim   
\begin{equation}\label{e.g3}
g_3= 2 \left(\frac{\lambda_s}{\lambda} + b\right) \int \varepsilon y \Lambda  Q+\mathcal O( |s|^{-2+\theta}). 
\end{equation}
First,   by \eqref{bd.psi}
\begin{align*}
\left|\int \Lambda Q  \varepsilon (\psi-y)\right|& \lesssim
\int_{|y|>|s|^{\frac 12}} |\varepsilon| (|\psi|+|y|)\frac {dy}{y^2} + |s|^{-\frac 12} \int_{|y|<|s|^{\frac 12}} |\varepsilon| \\
& \lesssim \|\varepsilon\|_{L^2}  \left(\int_{|y|>|s|^{\frac 12}} dy/y^2\right)^{\frac 12}+ |s|^{-\frac 14} \|\varepsilon\|_{L^2}
    \lesssim  |s|^{-\frac 34}.
\end{align*}
Second, since $\psi\equiv 0$ for $y<-2|s|^{\frac 35}$, $|\psi|\lesssim |s|^{\frac 12}$, and $|\Lambda P_b |\lesssim (1+y_+)^{-1}$
(here $y_+=\max(0,y)$),
\begin{align*}
\left|b \int \Lambda P_b \varepsilon  \psi \right| \lesssim
|s|^{-\frac 12} \int_{y>-2|s|^{\frac 35 }} |\varepsilon| (1+y_+)^{-1}
\lesssim |s|^{- \frac 15} \|\varepsilon\|_{L^2} \lesssim |s|^{-\frac 7{10}}.
\end{align*}
Thus, \eqref{e.g3} follows from \eqref{BS.2} and $\theta>\frac 35$.

\medskip
\noindent \emph{Estimate for $g_4$.}
The proof of the following estimate is similar and easier (due to the stronger decay of $Q'$ with respect to $\Lambda Q$)  
\begin{equation}\label{e.g4}
g_4=2 \left(\frac{x_s}{\lambda} -1\right) \int \varepsilon y   Q' +\mathcal O(  |s|^{-2+\theta}). 
\end{equation}
{Note that to deal with the term $2\left(\frac{x_s}{\lambda} -1\right) \int \varepsilon_y\varepsilon \psi$, we integrate by parts and use \eqref{e.g15}  so that
\[ \left|2\left(\frac{x_s}{\lambda} -1\right) \int \varepsilon_y\varepsilon \psi \right|=
\left|\left(\frac{x_s}{\lambda} -1\right) \int \varepsilon^2 \psi_y \right|
\lesssim \mathcal{N}(\epsilon) \int \varepsilon^2 \lesssim C^{\star} |s|^{-2+\frac{\theta}2} \, .
\]}
\medskip

\noindent \emph{Estimate for $g_5$.}
We claim   
\begin{equation}\label{e.g5}
g_5=\mathcal O(|s|^{-2+\theta}). 
\end{equation}
From \eqref{def.DQb}, one has $\Big|\frac{\partial Q_b}{\partial b}\Big|\lesssim (1+y_+)^{-1} $  and thus,
using \eqref{BS.3} and \eqref{bd.psi},
\begin{align*}
|g_5|&\lesssim C^\star |s|^{-2+\frac \theta 2} \|\psi\|_{L^\infty} \int_{y>-2 |s|^{\frac 35}} |\varepsilon|(1+y_+)^{-1} \\
&\lesssim  C^\star |s|^{-\frac {17}{10} + \frac \theta 2} { |s|^{\frac12}}\|\varepsilon\|_{L^2} \lesssim C^\star |s|^{-{ \frac {17}{10}} + \frac \theta 2}
\lesssim |s|^{-2+\theta},
\end{align*}
since $\theta>\frac35$.
 
\medskip
\noindent \emph{Estimate for $g_6$.}
We claim   
\begin{equation}\label{e.g6}
g_6=\mathcal O(|s|^{-2+\theta}). 
\end{equation}
Indeed, by $\|\psi\|_{L^\infty} \lesssim |s|^{\frac 12}$, Cauchy-Schwarz inequality and \eqref{lprofile.2}, 
we have (using $\theta>\frac 12$)
\[
|g_6|\lesssim |s|^{-1} \|\varepsilon\|_{L^2} \lesssim |s|^{-\frac 32} \lesssim |s|^{-2+\theta}.
\]

\medskip
\noindent \emph{Estimate for $g_7$.}
We claim   
\begin{equation}\label{e.g7}
g_7=\mathcal O(|s|^{-2+\theta}). 
\end{equation}
By the definition of   $\psi$ in \eqref{defG},
\[
  \partial_s \psi = - \frac 12 |s|^{-1} \psi + 
\frac{\pi}2 |s|^{-1}  y \phi'(y|s|^{-\frac 12}) \chi^2(y|s|^{-\frac 35}) 
+ \frac 35 |s|^{-\frac{11}{10}} \pi \Big(\phi(y|s|^{-\frac 12})-\frac 12\Big) y (\chi^2)'(y |s|^{-\frac 35}).
\]
Thus,
$|\partial_s \psi|\lesssim |s|^{-\frac 12} \mathbf{1}_{y>-2 |s|^{\frac 35}}$ and
$$|g_7|\lesssim |s|^{-\frac 12} \mathcal N(\varepsilon)^2 \lesssim (C^\star)^2 |s|^{-\frac 52 + \theta}
\lesssim |s|^{-2+\theta} .$$

\medskip

Therefore, combining \eqref{e.g1}, \eqref{e.g2}, \eqref{e.g3}, \eqref{e.g4}, \eqref{e.g5}, \eqref{e.g6} and 
 \eqref{e.g7},  we finish the proof of \eqref{eq:G}.
\end{proof}

\subsection{Estimates on the functional $K$} 
Recall that $K=\int \varepsilon \mathcal{L}_{\varphi} P_b$ is defined in \eqref{bootstrap.4}.  
\begin{lemma}\label{le:KK}
For $|S_0|$ large enough possibly depending on $C^{\star}$, for all $s\in \mathcal I^\star$, 
\begin{equation} \label{bg:K}
|K| \lesssim \mathcal{N}(\varepsilon)|s|^{\frac{\theta}2} \lesssim C^{\star}|s|^{-1+\theta} \, .
\end{equation}
and
\begin{equation}\label{eq:K}
\Big|\frac {d K}{ds}+\big(\frac{\lambda_s}{\lambda}+b\big)p_0 \Big| 
\lesssim \left( \int \varepsilon^2 \varphi_y \right)^{\frac12} + (C^\star)^2 |s|^{  -2+\frac{3\theta}2}  \, .
\end{equation}
\end{lemma}
\begin{proof}
We begin with the proof of \eqref{bg:K}. By using the definition of $\mathcal{L}_{\varphi}$ in \eqref{defK}, we decompose $K$ as 
\begin{displaymath} 
K=\int \varepsilon D^1P_b +\int \varepsilon P_b \varphi-3\int \varepsilon Q^2P_b \, .
\end{displaymath}
First, we deduce from \eqref{Pb}, \eqref{def.varphi} and \eqref{defN}  that
\begin{displaymath}
\begin{split}
\Big| \int \varepsilon P_b \varphi \Big| &\lesssim \int_{y>0}	\frac{|\varepsilon|}{1+y}+\int_{ -2B|s|^{\theta}<y<0}|\varepsilon| \varphi+\int_{-2|s|<y<-2B|s|^{\theta}}|\varepsilon| \varphi \\
& \lesssim \mathcal{N}(\varepsilon)+\mathcal{N}(\varepsilon)|s|^{
\frac{\theta}2}+\mathcal{N}(\varepsilon)\left( \int_{-2|s|<y<-2B|s|^{\theta}}\frac1{1+|y|}\right)^{\frac12} \\ & \lesssim \mathcal{N}(\varepsilon)|s|^{
\frac{\theta}2} \, .
\end{split}
\end{displaymath}
Second, we deduce from \eqref{bd.PbRb} and \eqref{defN} that
\begin{displaymath}
\Big| \int \varepsilon D^1P_b  \Big| = \Big| \int D^{\frac12}\varepsilon D^{\frac12}P_b  \Big| \lesssim 
\|D^{\frac12}\varepsilon\|_{L^2} \|D^{\frac12}P_b\|_{L^2} \lesssim \mathcal{N}(\varepsilon) \ln|s| \, .
\end{displaymath}
Moreover, we get easily that 
\begin{displaymath}
\Big|\int \varepsilon Q^2P_b  \Big| \lesssim \mathcal{N}(\varepsilon) \, .
\end{displaymath}
Those estimates together with the bootstrap hypothesis \eqref{bootstrap.4} conclude the proof of \eqref{bg:K}. 

\smallskip
Next, we turn to the proof of \eqref{eq:K}. By using the equation of $\varepsilon_s$ in \eqref{modulation.4}, we compute 
\begin{equation} \label{eq:K.1}
\begin{split}
\frac{dK}{ds}&=\int \varepsilon_s \mathcal{L}_{\varphi}P_b+\int \varepsilon \varphi_s P_b+b_s \int \varepsilon \mathcal{L}_{\varphi}\left(\frac{\partial P_b}{\partial b}\right)\\
&=k_1+k_2+k_3+k_4+k_5+k_6+k_7 \, ,
\end{split}
\end{equation}
where
\begin{align*}
	& k_1 =  \int V_y  \mathcal{L}_{\varphi}P_b \, , 
	& k_2& =  \frac{\lambda_s}{\lambda} \int (\Lambda \varepsilon) \mathcal{L}_{\varphi}P_b \, , \\
	&k_3=(\frac {\lambda_s}{\lambda} + b) \int (\Lambda Q_b) \mathcal{L}_{\varphi}P_b \, ,
	& k_4& = (\frac {x_s}\lambda-1) \int (Q_b+\varepsilon)_y \mathcal{L}_{\varphi}P_b\, , \\
	&k_5= -(b_s+b^2)  \int \frac{\partial Q_b}{\partial b} \mathcal{L}_{\varphi}P_b\, ,
	& k_6& =  \int \Psi_b \mathcal{L}_{\varphi}P_b\, , \\
	&k_7 = \int \varepsilon \varphi_s P_b+b_s \int \varepsilon \mathcal{L}_{\varphi}\left(\frac{\partial P_b}{\partial b}\right)\, .
\end{align*}
In the rest of the proof, we estimate $k_1, \cdots, k_7$ separately, taking $S_0$ large enough, possibly depending on $C^{\star}$.

\medskip

\noindent \emph{Estimate for $k_1$.} In order to estimate $k_1$, we rewrite $\mathcal{L}_{\varphi}P_b$ as 
\begin{equation} \label{LvarphiPb}
\begin{split}
\mathcal{L}_{\varphi}P_b&=D^1P_b+P_b\varphi-3Q^2P_b \\ 
&=\mathcal{L}P+D^1(P(\chi_b-1))+P(\chi_b\varphi-1)-3Q^2P(\chi_b-1) \, .
\end{split}
\end{equation}
Moreover,  integrating by parts and using the definition of $V$ in \eqref{defV},
\begin{displaymath}
\begin{split}
k_1&=-\int \mathcal{L}\varepsilon \big(\mathcal{L}_{\varphi}P_b \big)_y+3\int (Q_b^2-Q^2) \varepsilon \big(\mathcal{L}_{\varphi}P_b \big)_y+\int \big(3Q_b\varepsilon^2+\varepsilon^3\big)\big(\mathcal{L}_{\varphi}P_b \big)_y \\ 
&=k_{1,1}+k_{1,2}+k_{1,3} \, .
\end{split}
\end{displaymath}

First, we deal with $k_{1,1}$. Recalling that $\mathcal{L}(\Lambda Q)=-Q$,  the equation in \eqref{nlprofile.1} and \eqref{LvarphiPb}, we have 
\begin{displaymath}
\begin{split}
k_{1,1}&=\int \varepsilon Q-\int \mathcal{L}\varepsilon \big(D^1(P(\chi_b-1))\big)_y-\int  \mathcal{L}\varepsilon \big(P(\chi_b\varphi-1)\big)_y+3\int \mathcal{L}\varepsilon \big(Q^2P(\chi_b-1)\big)_y\\ 
&=\int \varepsilon Q+k_{1,1,1}+k_{1,1,2}+k_{1,1,3} \, .
\end{split}
\end{displaymath}
From the definition of $\mathcal{L}$, \eqref{BS.1}, the definition of $\chi_b$ in \eqref{Pb}
and $P\in \mathcal Z$, we get that 
\begin{equation} \label{eq:k111}
|k_{1,1,1}| \lesssim \|\varepsilon\|_{L^2}\big(\|(P(\chi_b-1))'''\|_{L^2}+\|(P(\chi_b-1))''\|_{L^2}\big) \lesssim |s|^{-2} \, ,
\end{equation}
and
\begin{equation} \label{eq:k113}
|k_{1,1,3}| \lesssim \|\varepsilon\|_{L^2}\big(\|(Q^2P(\chi_b-1))'\|_{L^2}+\|(Q^2P(\chi_b-1))''\|_{L^2}\big) \lesssim |s|^{-\frac92} \, .
\end{equation}
We rewrite $k_{1,1,2}$ as 
\begin{displaymath} 
k_{1,1,2}=\int D^1\varepsilon
\big(P(\chi_b\varphi-1)\big)_y+\int(1-3Q^2) \varepsilon \big(P(\chi_b\varphi-1)\big)_y \, .
\end{displaymath}
To treat the first term above, we use the decomposition 
$$\big(P(\chi_b\varphi-1)\big)_y=P'(\chi_b\varphi-1)+P\chi_b'\varphi+P_b\varphi_y \, .$$
On the one hand, observe that 
\begin{displaymath}
\int D^1\varepsilon \big(P'(\chi_b\varphi-1)+P\chi_b'\varphi\big)
=-\int \mathcal{H}\varepsilon\Big(P''(\chi_b\varphi-1)+2P'\chi_b'\varphi
+P\chi_b''\varphi+P\chi_b'\varphi_y\Big) ,
\end{displaymath}
\[ \|P''(1-\chi_b\varphi)\|_{L^2}\lesssim \left(\int_{y<-\frac{B}2|s|^{\theta}}\frac1{|y|^6} \right)^{\frac12}
+|s|^{-2\theta}\|yP''\|_{L^2}\lesssim |s|^{-2\theta} \, ,
\]
thanks to \eqref{phi_B}, using $P'\in \mathcal Y_2$ and the definition of $\varphi$ in \eqref{def.varphi},
\[ \|P'\chi_b'\varphi\|_{L^2}\lesssim |b|\left(\int_{-2|s|<y<-|s|}\frac1{|y|^6} \right)^{\frac12} \lesssim |s|^{-\frac72} \, ,
\]
\[ \|P\chi_b''\varphi\|_{L^2}\lesssim |b|^2\left(\int_{-2|s|<y<-|s|}\frac1{|y|^2} \right)^{\frac12} \lesssim |s|^{-\frac52} \, ,
\]
and 
\[ \|P\chi_b'\varphi_y\|_{L^2}\lesssim |b|\left(\int_{-2|s|<y<-|s|}\frac1{|y|^4} \right)^{\frac12} \lesssim |s|^{-\frac52} \, .
\]
Then, we deduce from \eqref{BS.1} that 
\begin{displaymath} 
\left|\int D^1\varepsilon \big(P'(\chi_b\varphi-1)+P\chi_b'\varphi\big)\right|\lesssim |s|^{-\frac12-2\theta} \, .
\end{displaymath}
On the other hand, Lemma \ref{pw.HY} and the translation invariance of $\mathcal{H}$ 
together with \eqref{def.Dvarphi} imply that
\[
|\mathcal{H}\partial_y\big(P_b \varphi_y\big)| \lesssim \varphi_y,
\]
and thus
\begin{displaymath} 
\left|\int D^1\varepsilon P_b \varphi_y \right|=
\left|\int \varepsilon \mathcal{H}\partial_y\big(P_b \varphi_y\big) \right|
\lesssim \int |\varepsilon|\varphi_y \lesssim \left(\int \varepsilon^2 \varphi_y \right)^{\frac12} \, .
\end{displaymath}
Hence, we conclude  that 
\begin{equation} \label{eq:k112.1}
\left|\int D^1\varepsilon
\big(P(\chi_b\varphi-1)\big)_y\right| \lesssim  |s|^{-\frac12-2\theta}+\left(\int \varepsilon^2 \varphi_y \right)^{\frac12} \, .
\end{equation}
We write the second term in the decomposition of $k_{1,1,2}$ as 
\begin{displaymath}
\begin{split}
\int(1&-3Q^2) \varepsilon \big(P(\chi_b\varphi-1)\big)_y
\\&=\int(1-3Q^2) \varepsilon P'(\chi_b\varphi-1)+\int(1-3Q^2) \varepsilon P\chi_b'\varphi+\int(1-3Q^2) \varepsilon P_b\varphi_y \, .
\end{split}
\end{displaymath}
By using  \eqref{BS.1}, $P'\in \mathcal Y_2$, and \eqref{phi_B}, we have 
\begin{displaymath}
\begin{split}
\left|\int(1-3Q^2) \varepsilon P'(\chi_b\varphi-1) \right|
&\lesssim \|\varepsilon\|_{L^2}\left(\int_{y<-\frac{B}2|s|^{\theta}}\frac1{|y|^4}\right)^{\frac12}+|s|^{-2\theta}\int_{y>-\frac{B}2|s|^{\theta}} |\varepsilon|\frac{|y|}{(1+|y|)^2} \\ 
& \lesssim |s|^{-\frac12-\frac{3\theta}2}+|s|^{-2\theta}\mathcal{N}(\varepsilon) \, ,
\end{split}
\end{displaymath}
\[ \left|\int(1-3Q^2) \varepsilon P\chi_b'\varphi \right|
\lesssim |b|\left(\int\varepsilon^2\varphi \right)^{\frac12}\left(\int_{-2|s|<y<-|s|}\frac1{|y|} \right)^{\frac12}
 \lesssim |s|^{-1}\mathcal{N}(\varepsilon) \, ,
\]
and 
\[\left|\int(1-3Q^2) \varepsilon P_b\varphi_y \right|  
 \lesssim \left( \int \varepsilon^2 \varphi_y \right)^{\frac12} \|\varphi_y\|_{L^1}^{\frac 12}
\lesssim \left( \int \varepsilon^2 \varphi_y \right)^{\frac12} \, .
\]
Hence, we deduce gathering those estimates and using \eqref{bootstrap.4} that 
\begin{equation} \label{eq:k112.2}
\left|\int(1-3Q^2) \varepsilon \big(P(\chi_b\varphi-1)\big)_y\right| \lesssim  |s|^{-\frac12-\frac{3\theta}2}+\left(\int \varepsilon^2 \varphi_y \right)^{\frac12} \, ,
\end{equation}
since $2-\frac{\theta}2>\frac12+\frac{3\theta}2$. Therefore, we conclude from \eqref{eq:k112.1} and \eqref{eq:k112.2} that
\begin{equation} \label{eq:k112}
|k_{1,1,2}| \lesssim  |s|^{-\frac12-\frac{3\theta}2}+\left(\int \varepsilon^2 \varphi_y \right)^{\frac12} \, .
\end{equation}

Next, we derive an estimate for $k_{1,2}$. Since 
\[k_{1,2}=3\int (2bQP_b+b^2P_b^2) \varepsilon \big(D^1P_b+P_b\varphi-3Q^2P_b \big)_y \, ,
\]
we deduce easily from \eqref{BS.1} that
\begin{equation} \label{eq:k12}
|k_{1,2}| \lesssim |b|\|\varepsilon\|_{L^2} \lesssim |s|^{-\frac32} \, .
\end{equation}

Finally, since 
\[ k_{1,3}=\int \big(3(Q+bP_b)\varepsilon^2+\varepsilon^3\big)\big(D^1P_b+P_b\varphi-3Q^2P_b \big)_y \, ,\]
we deduce by using the Cauchy-Schwarz inequality, the Gagliardo-Nirenberg inequalities \eqref{gnp} with $p=4$ and $p=6$, \eqref{bootstrap.4} and \eqref{BS.1} that
\begin{equation} \label{eq:k13}
|k_{1,3}| \lesssim \left( \int \varepsilon^4\right)^{\frac12}+\left( \int \varepsilon^6\right)^{\frac12} \lesssim   \|\varepsilon\|_{L^2} \, \big(\mathcal{N}(\varepsilon)+\mathcal{N}(\varepsilon)^2\big)
\lesssim C^{\star}|s|^{-\frac32+\frac{\theta}2} \, .
\end{equation}

Therefore, we conclude gathering \eqref{BS.4} (to control $( \varepsilon,Q)$), \eqref{eq:k111}, \eqref{eq:k113}, \eqref{eq:k112}, \eqref{eq:k12} and \eqref{eq:k13} that 
\begin{equation} \label{eq:k1}
|k_1| \lesssim  C^{\star}|s|^{-\frac32+\frac{\theta}2} +\left(\int \varepsilon^2 \varphi_y \right)^{\frac12}\, .
\end{equation}

\medskip

\noindent  \emph{Estimate for $k_2$.} We decompose $k_2$ as
\begin{displaymath}
\begin{split}
k_2&=\frac {\lambda_s}\lambda \int \Lambda \epsilon\big( D^1P_b+P_b\varphi-3Q^2P_b \big) \\
&=k_{2,1}+k_{2,2}+k_{2,3} \, ,
\end{split}
\end{displaymath}
and estimate each term separately. 

By the definition of $\Lambda \varepsilon$, we have 
\[ 
 \int \Lambda \varepsilon D^1P_b=\frac12\int \varepsilon D^1P_b+\int y\varepsilon_y D^1P_b \, .
\]
On the one hand, we see from \eqref{bd.PbRb} and \eqref{BS.1} that
\[ 
\left|\int \varepsilon D^1P_b\right|\lesssim \|D^{\frac12}\varepsilon\|_{L^2} \|D^{\frac12}P_b\|_{L^2} \lesssim \mathcal{N}(\varepsilon) |\ln |s||^{\frac12} \, .
\]
On the other hand, we have from the properties of $\mathcal{H}$
\[ 
\int y\varepsilon_y D^1P_b=-\int \mathcal{H}(y\varepsilon_y)P_b'=-\int y (\mathcal{H}\varepsilon_y) P_b'=-\int yD^1\varepsilon P_b' \, ,
\]
so that by \eqref{yPbprime}
\[ 
\left|\int y\varepsilon_y D^1P_b \right| \lesssim \|D^{\frac12}\varepsilon\|_{L^2} \|D^{\frac12}(yP_b')\|_{L^2} \lesssim \mathcal{N}(\varepsilon) \, .
\]
Hence, it follows from \eqref{bootstrap.4}, \eqref{BS.1} and \eqref{BS.2} that 
\[ 
|k_{2,1}| \lesssim (C^{\star})^2|s|^{-2+\theta}|\ln |s||^{\frac12} \, .
\]

After integrating by parts, we have
\[ 
 \int \Lambda \varepsilon P_b\varphi=-\frac12 \int \varepsilon P_b\varphi-
 \int y\varepsilon P_b'\varphi -\int y\varepsilon P_b\varphi_y\, .
\]
Moreover, 
\[ 
\left| \int \varepsilon P_b\varphi \right|+\left| \int y\varepsilon P_b'\varphi \right|
\lesssim \left(\int \varepsilon^2 \varphi\right)^{\frac12} \left(\left(\int P_b^2\varphi\right)^{\frac12} +\left(\int (yP_b')^2\varphi\right)^{\frac12}\right)
\lesssim \mathcal{N}(\varepsilon) |s|^{\frac{\theta}2} \, ,
\]
\[
\left| \int y\varepsilon P_b\varphi_y \right| \lesssim \left(\int (yP_b)^2 \varphi_y\right)^{\frac12}\left(\int \varepsilon^2 \varphi_y\right)^{\frac12} 
\]
and 
\[ 
\int (yP_b)^2 \varphi_y \lesssim \int_{y>0}\varphi_y+\int_{-2B|s|^{\theta}<y<0}|y|^2\varphi_y+\int_{-2|s|<y<-2B|s|^{\theta}}B\frac{y^2}{1+y^2}\lesssim |s|^{2\theta} \, .
\]
Hence, we deduce from \eqref{bootstrap.4} and \eqref{BS.1} and \eqref{BS.2} that 
\[ 
|k_{2,2}| \lesssim (C^{\star})^2|s|^{-2+\frac{3\theta}2} +C^{\star}|s|^{-1+\frac{3\theta}2}\left(\int \varepsilon^2 \varphi_y\right)^{\frac12} \, .
\]

Finally, we get easily 
\[ 
\left|\int \Lambda \varepsilon Q^2P_b \right| \lesssim \|\varepsilon\|_{L^2}  \|\Lambda (Q^2P_b)\|_{L^2} \lesssim \|\varepsilon\|_{L^2} \, ,
\]
so that, by using \eqref{BS.1},
\[ 
|k_{2,3}| \lesssim C^{\star}|s|^{-\frac32+\frac{\theta}2}  \, .
\]

Therefore, we conclude gathering those estimates that 
\begin{equation} \label{eq:k2}
|k_2| \lesssim  (C^{\star})^2|s|^{-2+\frac{3\theta}2} +C^{\star}|s|^{-1+\frac{3\theta}2}\left(\int \varepsilon^2 \varphi_y\right)^{\frac12} \, .
\end{equation}

\medskip

\noindent \emph{Estimate for $k_3$.} We split $k_3$ as 
\begin{displaymath}
\begin{split}
k_3&=(\frac {\lambda_s}\lambda+b) \int \Lambda Q \mathcal{L}_{\varphi}P_b
+(\frac {\lambda_s}\lambda+b)b \int \Lambda P_b \mathcal{L}_{\varphi}P_b\\
&=k_{3,1}+k_{3,2} \, ,
\end{split}
\end{displaymath}
and estimate each term separately. 

By using \eqref{LvarphiPb} and the identity $\mathcal{L}(\Lambda Q)=-Q$, we have 
(see the definition of $p_0=(P,Q)$ in \eqref{nlprofile.2})
\[k_{3,1}=(\frac {\lambda_s}\lambda+b)\left(-p_0+\int \Lambda Q\big(D^1(P(\chi_b-1))+P(\chi_b\varphi-1)-3Q^2P(\chi_b-1)\big)\right) \, . \]
Moreover, since $\Lambda Q \in \mathcal{Y}_2$, Lemma \ref{HY} implies that $\mathcal{H}(\Lambda Q) \in \mathcal{Y}_1$, and thus 
\[ 
\left|\int \Lambda Q D^1(P(\chi_b-1))\right| = \left|\int P(\chi_b-1) \partial_y\mathcal{H}(\Lambda Q) \right| \lesssim \int_{y <-|s|}\frac1{|y|^2} \lesssim |s|^{-1} \, .
\]
By using \eqref{phi_B}, $\Lambda Q\in \mathcal Y_2$, $P\in \mathcal Z$, we also have that 
\[
\begin{split}
&\left|\int \Lambda Q (P(\chi_b\varphi-1))\right|\\ & \lesssim \int_{y <-\frac{B}2|s|^{\theta}}\frac1{|y|^2}+|s|^{-2\theta}\left(\int_{-\frac{B}2|s|^{\theta}<y<0}\frac{|y|}{(1+|y|)^2}+\int_{y>0}\frac{|y|}{(1+|y|)^3}\right) \lesssim |s|^{-\theta} \, ,
\end{split}
\]
and 
\[
\left|\int \Lambda Q Q^2P(\chi_b-1))\right| \lesssim \int_{y<-|s|}\frac1{|y|^6} \lesssim |s|^{-5} \, .
\]
Hence, we deduce from \eqref{bootstrap.1} and \eqref{BS.2} that 
\[
\big|k_{3,1}+(\frac {\lambda_s}\lambda+b)p_0\big| \lesssim C^{\star}|s|^{-1-\frac{\theta}2}.
\]

Next, we look at $k_{3,2}$. We have that
\[
k_{3,2}=(\frac {\lambda_s}\lambda+b)b\int \big(\frac12 P_b+yP_b'\big) \big(D^1P_b+P_b\varphi-3Q^2P_b \big)
\]
By using \eqref{bd.PbRb} and \eqref{yPbprime}, we see that 
\[
\left|\int P_b D^1P_b \right| =\|D^{\frac12}P_b\|_{L^2}^2 \lesssim \ln|s|
\]
and 
\[
\left|\int yP_b' D^1P_b \right| \le \|D^{\frac12}(yP_b')\|_{L^2}\|D^{\frac12}P_b\|_{L^2} \lesssim |\ln|s||^{\frac12} \, .
\]
We also have that  (see the definition of $\varphi$ in \eqref{def.varphi})
\begin{equation}\label{encore.truc}
\left|\int P_b^2\varphi \right| 
\lesssim \int_{-2 |b|^{-1}<y<-10 B|s|^{\theta}} \frac {dy} {|y|} + \int_{y>-10 B|s|^{\theta}} P^2
\lesssim |s|^{\theta},
\end{equation}
\[
\left|\int yP_b'P_b\varphi \right| \lesssim \int |P'|+|b|\int |\chi'(|b|y)| \lesssim 1 \, 
\]
and 
\[
\left|\int \Lambda P_bQ^2P_b \right| \lesssim \int Q^2 \lesssim 1 \, .
\]
Hence, we deduce from \eqref{bootstrap.4} and \eqref{BS.1} that
\[
\big|k_{3,2}\big| \lesssim C^{\star}|s|^{-2+\frac{3\theta}2}
\]

Therefore, we conclude gathering those estimates that 
\begin{equation} \label{eq:k3}
\big|k_{3}+(\frac {\lambda_s}\lambda+b)p_0\big|  \lesssim   C^{\star} |s|^{-2+\frac{3\theta}2} \, .
\end{equation}
\medskip

\noindent \emph{Estimate for $k_4$.} We decompose $k_4$ as follows. 
\begin{displaymath}
\begin{split}
k_4&=(\frac {x_s}\lambda-1) \int Q' \mathcal{L}_{\varphi}P_b
+(\frac {x_s}\lambda-1)b \int P_b'\mathcal{L}_{\varphi}P_b
+(\frac {x_s}\lambda-1)\int \varepsilon_y\mathcal{L}_{\varphi}P_b \\
&=k_{4,1}+k_{4,2}+k_{4,3} \, .
\end{split}
\end{displaymath}
We estimate each term separately. 

The decomposition   \eqref{LvarphiPb} and the property $\mathcal{L}(Q')=0$ imply
\[k_{4,1}=(\frac {x_s}\lambda-1)\int Q'\big(D^1(P(\chi_b-1))+P(\chi_b\varphi-1)-3Q^2P(\chi_b-1)\big) \, . \]
Moreover, we have the following bounds 
\[\left| \int Q' D^1(P(\chi_b-1))\right| \lesssim \|Q\|_{L^2}\|(P(\chi_b-1))''\|_{L^2}\lesssim |s|^{-\frac32} \, ,
\]
\[\left| \int Q' P(\chi_b\varphi-1)\right| \lesssim \int_{y<-\frac{B}2|s|^{\theta}}\frac1{|y|^3}+|s|^{-2\theta}\int \frac{|y|}{(1+|y|)^3} \lesssim |s|^{-2\theta}\, ,
\]
thanks to \eqref{phi_B}, and
\[\left| \int Q' Q^2P(\chi_b-1)\right| \lesssim \int_{y < -|s|}\frac1{|y|^7} \lesssim |s|^{-6}\, .
\]
Thus, it follows from \eqref{BS.2} that
$|k_{4,1}| \lesssim C^{\star}|s|^{-1-\frac{3\theta}2}$.

By using that $\int P_b'D^1P_b=\int P_b'\mathcal{H}(P_b')=0$, we rewrite $k_{4,2}$ as
\[ k_{4,2}=(\frac {x_s}\lambda-1)b \int P_b' P_b\varphi-3(\frac {x_s}\lambda-1)b \int P_b'Q^2P_b \, .
\]
We also observe that 
\[\left| \int P_b' P_b\varphi \right| =\frac 12 \left| \int P_b^2\varphi_y \right| \lesssim 1  \quad \hbox{and}\quad 
\left| \int P_b' Q^2P_b \right| \lesssim \int \frac1{1+|y|^2}\lesssim 1 \, .  \]
Thus, it follows from \eqref{BS.1} and \eqref{BS.2} that
$|k_{4,2}| \lesssim C^{\star}|s|^{-2+\frac{\theta}2}$.

We rewrite $k_{4,3}$ as 
\[ k_{4,3}=(\frac {x_s}\lambda-1)\int \varepsilon_y\big(D^1P_b+P_b\varphi-3Q^2P_b \big) \, .
\]
Moreover, we observe that 
\[
\left| \int \varepsilon_yD^1P_b\right| \lesssim \|D^{\frac12}\varepsilon\|_{L^2} \|D^{\frac32}P_b\|_{L^2} \lesssim \mathcal{N}(\varepsilon) \, ,\]
\[
\begin{split}
 \left| \int \varepsilon_yP_b\varphi\right| & \le \left| \int \varepsilon P_b'\varphi\right|+
\left| \int \varepsilon P_b\varphi_y\right|\\
& \lesssim \left(\int \varepsilon^2\varphi \right)^{\frac12} \|P_b'\|_{L^2}+ \left(\int \varepsilon^2\varphi_y\right)^{\frac12} \left(\int \varphi_y\right)^{\frac12} \lesssim \mathcal{N}(\varepsilon) \, ,  
\end{split}
\]
and 
\[
\left| \int \varepsilon_yQ^2P_b\right| \lesssim \|D^{\frac12}\varepsilon\|_{L^2} \|D^{\frac12}(Q^2P_b)\|_{L^2} \lesssim \mathcal{N}(\varepsilon) \, .\]
Then, we deduce from \eqref{bootstrap.4} that 
 $ |k_{4,3}| \lesssim (C^{\star})^2|s|^{-2+\theta}$.

Therefore, we conclude gathering those estimates that 
\begin{equation} \label{eq:k4}
 |k_4| \lesssim (C^{\star})^2|s|^{-2+\theta}  \, .
\end{equation}

\medskip

\noindent \emph{Estimate for $k_5$.} By using \eqref{def.DQb}, we decompose $k_5$ as
\begin{displaymath}
k_5 = -(b_s+b^2) \int \frac{\partial Q_b}{\partial b} \mathcal{L}_{\varphi}P_b=-(b_s+b^2)\int \left( P_b+y\chi_b'P\right)\left(D^1P_b+P_b \varphi-3Q^2P_b \right) \, ,
\end{displaymath}
and estimate each term separately. 

First, we deduce from  \eqref{bd.PbRb} that 
\begin{displaymath}
\left|\int P_bD^1P_b \right|=\|D^{\frac12}P_b\|_{L^2}^2 \lesssim \ln|s|  \, ,
\end{displaymath}
and 
\begin{displaymath}
\begin{split}
\left|\int y\chi_b'PD^1P_b \right|&\le\|D^{\frac12}(y\chi_b'P)\|_{L^2}\|D^{\frac12}P_b\|_{L^2} \\
&\lesssim \|y\chi_b'P\|_{L^2}^{\frac12}\big\|\big(y\chi_b'P\big)'\big\|_{L^2}^{\frac12}
\|D^{\frac12}P_b\|_{L^2} \lesssim |\ln|s||^{\frac12}  \, .
\end{split}
\end{displaymath}
Next, in addition to \eqref{encore.truc},
$\left|\int P_b^2 \varphi \right| \lesssim |s|^{\theta} 
$, we have
\begin{displaymath}
\left|\int y\chi_b'PP_b \varphi \right| \lesssim |b| \int_{-2|s| \le y \le -|s|}\frac{|y|}{1+|y|}|\chi'(|b|y)| \lesssim 1 \, .
\end{displaymath}
Finally, we see easily that 
\begin{displaymath}
\left|\int \left( P_b+y\chi_b'P\right)Q^2P_b\right| \lesssim 1 \, .
\end{displaymath}
Therefore, we conclude gathering those estimates and using \eqref{BS.3} that 
\begin{equation} \label{eq:k5}
|k_5| \lesssim C^{\star}|s|^{-2+\frac{3\theta}2}\, .
\end{equation}
\medskip

\noindent \emph{Estimate for $k_6$.} Recall that 
\begin{displaymath}
k_6 =  \int \Psi_b \mathcal{L}_{\varphi}P_b=\int \Psi_b\left(D^1P_b+P_b \varphi-3Q^2P_b \right) \, .
\end{displaymath}
First, it follows from \eqref{bd.PbRb} and \eqref{lprofile.3}  that 
\begin{displaymath}
\left|\int \Psi_b D^1P_b \right| \le \|D^{\frac12}\Psi_b\|_{L^2}\|D^{\frac12}P_b\|_{L^2} \lesssim |s|^{-2}\ln|s| \, .
\end{displaymath}
Second, we deduce from \eqref{PsibVarphi} and \eqref{encore.truc} that 
\begin{displaymath}
\left|\int \Psi_b P_b \varphi \right| \le \left( \int \Psi_b^2\varphi \right)^{\frac12}\left( \int P_b^2\varphi \right)^{\frac12} \lesssim |s|^{-2+\theta} \, .
\end{displaymath}
Finally, \eqref{lprofile.2} yields 
\begin{displaymath}
\left|\int \Psi_b Q^2P_b   \right| \le \left( \int \Psi_b^2Q^2\right)^{\frac12}\left( \int P_b^2Q^2 \right)^{\frac12} \lesssim |s|^{-2} \, .
\end{displaymath}
Therefore, we conclude that 
\begin{equation} \label{eq:k6}
|k_6| \lesssim |s|^{-2+\theta}\, .
\end{equation}
\medskip

\noindent \emph{Estimate for $k_7$.} By using \eqref{varphi_s}, we decompose the first term of $k_7$ as
\begin{displaymath}
 \int \varepsilon  \varphi_s P_b=\theta s^{-1}|s|^{\theta}\int \varepsilon \frac{\phi'(\frac{y}{B}+|s|^{\theta})}{\phi(|s|^{\theta})}P_b-\theta s^{-1}|s|^{\theta}\int \varepsilon \frac{\phi(\frac{y}{B}+|s|^{\theta}) \phi'(|s|^{\theta})}{\phi^2(|s|^{\theta})}P_b \, .
\end{displaymath}
We deduce from \eqref{def.Dvarphi} that 
\begin{displaymath}
|s|^{-1+\theta}\Big|\int \varepsilon \frac{\phi'(\frac{y}{B}+|s|^{\theta})}{\phi(|s|^{\theta})}P_b \Big| \lesssim |s|^{-1+\theta}\left( \int \varepsilon^2 \varphi_y \right)^{\frac12} 
\end{displaymath}
and by \eqref{encore.truc}
\begin{displaymath}
\begin{split}
|s|^{-1+\theta}\left|\int \varepsilon \frac{\phi(\frac{y}{B}+|s|^{\theta}) \phi'(|s|^{\theta})}{\phi^2(|s|^{\theta})}P_b\right| &\lesssim |s|^{-1-\theta}\left( \int \varepsilon^2 \varphi \right)^{\frac12} \left( \int P_b^2 \varphi \right)^{\frac12}  \lesssim |s|^{-1-\frac{\theta}2} \mathcal{N}(\varepsilon) \, .
\end{split}
\end{displaymath}
Hence, it follows from \eqref{bootstrap.4} that
\begin{equation} \label{eq:k7.1}
\left| \int \varepsilon \varphi_s P_b \right| \lesssim |s|^{-1+\theta}\left( \int \varepsilon^2 \varphi_y \right)^{\frac12}+C^{\star}|s|^{-2} \, .
\end{equation}

To deal with the second term in $k_7$, we  use $\frac{\partial P_b}{\partial b}=-y\chi'(|b|y)P$, so that
\begin{displaymath}
\mathcal{L}_{\varphi}\left(\frac{\partial P_b}{\partial b}\right)=
D^1\left(\frac{\partial P_b}{\partial b}\right)-y\chi'(|b|y)P\varphi+3Q^2y\chi'(|b|y)P \, .
\end{displaymath}
We estimate each corresponding terms separately. First, we deduce by using \eqref{dPb_db} that
\begin{displaymath}
\left| \int \varepsilon D^1\left(\frac{\partial P_b}{\partial b}\right) \right| \lesssim 
\|D^{\frac12}\varepsilon\|_{L^2}\Big\| D^{\frac12}\left(\frac{\partial P_b}{\partial b}\right) \Big\|_{L^2} \lesssim |s| \mathcal{N}(\varepsilon) \, .
\end{displaymath}
Moreover, 
\begin{displaymath}
\left| \int \varepsilon \frac{\partial P_b}{\partial b} \varphi \right| \lesssim 
\left( \int \varepsilon^2 \varphi \right)^{\frac12} 
\left( \int |y \chi'(|b|y)|^2 \frac1{|y|} \right)^{\frac12} \lesssim |s| \mathcal{N}(\varepsilon) \, ,
\end{displaymath}
and 
\begin{displaymath}
\left| \int Q^2\varepsilon \frac{\partial P_b}{\partial b}  \right| \lesssim 
\left( \int \varepsilon^2 \varphi \right)^{\frac12} 
\left( \int |y \chi'(|b|y)|^2 \frac1{|y|^6} \right)^{\frac12} \lesssim |s|^{-\frac32} \mathcal{N}(\varepsilon) \, ,
\end{displaymath}
since $Q \in \mathcal{Y}_2$. Then it follows from \eqref{bootstrap.4} and \eqref{BS.5} that
\begin{equation} \label{eq:k7.2}
\left| b_s \int \varepsilon \mathcal{L}_{\varphi}\left(\frac{\partial P_b}{\partial b}\right)\right|
\lesssim (C^{\star})^2 |s|^{-2+\theta} \, .
\end{equation}

Therefore, we deduce gathering \eqref{eq:k7.1} and \eqref{eq:k7.2} that 
\begin{equation} \label{eq:k7}
|k_7| \lesssim |s|^{-1+\theta}\left( \int \varepsilon^2 \varphi_y \right)^{\frac12}+(C^{\star})^2 |s|^{-2+\theta} \, .
\end{equation}

 Finally, we conclude the proof of estimate \eqref{eq:K} gathering estimates \eqref{eq:k1}, \eqref{eq:k2}, \eqref{eq:k3}, \eqref{eq:k4}, \eqref{eq:k5}, \eqref{eq:k6}, \eqref{eq:k7} and taking $|S_0|$ large enough.
 \end{proof}

\subsection{Remaining terms}

\begin{lemma} For $|S_0|$ large enough possibly depending on $C^{\star}$ and for all $s\in \mathcal I^\star$,
\begin{equation}\label{eq:Z}  
|Z|\lesssim |s|^{-2+\theta},
\end{equation}
and
\begin{equation}\label{eq:dZ}  
\left|\frac {d Z}{ds}\right|\lesssim |s|^{-2}+|s|^{-\frac{3\theta}2}\left(\int \varepsilon^2 \varphi_{y} \right)^{\frac12}.
\end{equation}
\end{lemma}
\begin{proof}
Recall that $Z=Z_{1}+Z_{2}$, where
\[Z_1=\int \varepsilon^2 y \varphi_y \chi_1,\quad
Z_2   = - 2 \int  \Lambda Q(\varphi-1)\varepsilon \chi_2.\]

\noindent \textit{Proof of \eqref{eq:Z}}.
Indeed, 
\[
|Z_{1}|\lesssim \int_{y<-|s|^{\frac 23}} |\varepsilon|^2 |y|^{-1} \lesssim |s|^{-1-\frac 23}\lesssim |s|^{-2+\theta},
\]
and, since $\theta>\frac 35$,
\[
|Z_{2}|\lesssim \int_{y<-\frac B 4 |s|^{\theta}} |\varepsilon| |y|^{-2} \lesssim
\|\varepsilon\|_{L^2}  \left(\int_{y<-\frac B 4 |s|^{\theta}} |y|^{-4} \right)^{\frac 12}
\lesssim |s|^{-\frac 12 -\frac 32 \theta}\lesssim |s|^{-2+\theta}.
\]

\medskip
Now, we prove
\begin{equation}\label{eq:dZ1}  
\left|\frac {d Z_1}{ds}\right|\lesssim |s|^{-2}.
\end{equation}
and
\begin{equation}\label{eq:dZ2}  
\left|\frac {d Z_2}{ds}\right|\lesssim |s|^{-2}+|s|^{-\frac{3\theta}2}\left(\int \varepsilon^2 \varphi_{y} \right)^{\frac12}.
\end{equation}
which prove \eqref{eq:dZ}.

\medskip

\noindent \textit{Proof of \eqref{eq:dZ1}.} Observe from the definition of $Z_1$ that 
\[ \frac{dZ_1}{ds}=2\int \varepsilon \varepsilon_s y \varphi_y \chi_1+\int \varepsilon^2\partial_s(\varphi_y)\chi_1+\int \varepsilon^2\varphi_y\partial_s(\chi_1)\]
which implies by using \eqref{modulation.4} that 
\[  \frac{dZ_1}{ds}=z_{11}+z_{12}+z_{13}+z_{14}+z_{15}+z_{16}+z_{17} ,
\]
where
\begin{align*}
z_{11}&=2\int V_y \varepsilon y \varphi_y \chi_1, &
z_{12}&=2\frac{\lambda_s}{\lambda}\int \varepsilon \Lambda \varepsilon  y \varphi_y \chi_1, \\ 
z_{13}&=2\big(\frac{\lambda_s}{\lambda}+b\big)\int  \Lambda Q_b \varepsilon  y \varphi_y \chi_1, & 
z_{14}&=2\big(\frac{x_s}{\lambda}-1\big)\int  \big( Q_b+\varepsilon\big)_y \varepsilon y \varphi_y \chi_1,\\
z_{15}&=-2\big(b_s+b^2\big)\int  \frac{\partial Q_b}{\partial b} \varepsilon y \varphi_y \chi_1,&
z_{16}&=2\int  \Psi_b \varepsilon y \varphi_y \chi_1,\\
z_{17}&=\int \varepsilon^2\partial_s(\varphi_y)\chi_1+\int \varepsilon^2\varphi_y\partial_s(\chi_1)\, .
\end{align*}

First, we claim the following estimates
\begin{equation} \label{eq:dZ1.1} 
\|y\chi_1 \varphi_y\|_{L^{\infty}} \lesssim |s|^{-\frac23}, \quad \|(y\chi_1 \varphi_y)_y\|_{L^{\infty}} \lesssim |s|^{-\frac43}, \quad \|(y\chi_1 \varphi_y)_{yy}\|_{L^{\infty}} \lesssim |s|^{-2} , 
\end{equation}
and 
\begin{equation} \label{eq:dZ1.2} 
\|y\chi_1 \varphi_y\|_{L^2} \lesssim|s|^{-\frac13}, \quad \|(y\chi_1 \varphi_y)_y\|_{L^2} \lesssim  |s|^{-1},  \quad \|(y\chi_1 \varphi_y)_{yy}\|_{L^2} \lesssim  |s|^{-\frac53} \, ,
\end{equation}	
which follow directly from the definition of $\varphi$ in \eqref{def.varphi} and the definition of $\chi_1$ in \eqref{chi1chi2}.
	
\medskip 
	
\noindent \emph{Estimate for $z_{11}$.} By using the definition of $V$ in \eqref{defV}, we rewrite $z_{11}$ as 
\begin{align*}
z_{11} &= 2\int (D^1\varepsilon)_y \varepsilon y  \chi_1 \varphi_y+2\int \big((1-3Q^2)\varepsilon\big)_y \varepsilon y  \chi_1 \varphi_y
-6\int \big((Q_b^2-Q^2)\varepsilon\big)_y \varepsilon y  \chi_1 \varphi_y\\
&\quad -6\int\big(Q_b\varepsilon^2\big)_y \varepsilon y  \chi_1 \varphi_y
-2\int(\varepsilon^3)_y \varepsilon y  \chi_1 \varphi_y \\
&=z_{111}+z_{112}+z_{113}+z_{114}+z_{115} \, .
\end{align*}
We estimate each term separately.
	
First, we observe after integration by parts that 
\[
z_{111} =-  2\int (D^1\varepsilon)  \varepsilon_y (y  \chi_1 \varphi_y)
- 2\int (D^1\varepsilon) \varepsilon (y  \chi_1 \varphi_y)_y \, .
\]
Applying \eqref{m.1} with $a=y  \chi_1 \varphi_y$, from \eqref{eq:dZ1.1} and \eqref{BS.1}, one obtains
\[
\left| \int (D^1\varepsilon)  \varepsilon_y (y  \chi_1 \varphi_y)\right|\lesssim 
\| \varepsilon\|_{L^2}^2 \| (y  \chi_1 \varphi_y)_{yy}\|_{L^\infty} \lesssim |s|^{-3}.
\]
Applying \eqref{m.2} with $a=y  \chi_1 \varphi_y$, from \eqref{eq:dZ1.1}-\eqref{eq:dZ1.2} and \eqref{bootstrap.4}, \eqref{BS.1}, one obtains
\begin{align*}
\left| \int (D^1\varepsilon)  \varepsilon  (y  \chi_1 \varphi_y)_y\right| & 
\lesssim  \int |D^{\frac 12} \varepsilon|^2 |(y  \chi_1 \varphi_y)_y| + 
\|D^{\frac 12} \varepsilon\|_{L^2}^{\frac 32} \| \varepsilon\|_{L^2}^{\frac 12}
\|(y  \chi_1 \varphi_y)_{yy}\|_{L^2}^{\frac 34} \|(y  \chi_1 \varphi_y)_{y}\|_{L^2}^{\frac 14}
\\
& \lesssim (C^\star)^2 |s|^{-\frac {10}3 + \theta} + (C^\star)^{-\frac 32} |s|^{-3+\frac 34 \theta}
\lesssim |s|^{-2}.\end{align*}
Hence,
\begin{equation} \label{eq:dZ111}
|z_{111}|  
\lesssim |s|^{-2} \,.
\end{equation}

Second, we get after integration by parts that 
\begin{displaymath}
z_{112}= -6\int QQ'\varepsilon^2 y\chi_1 \varphi_y 
-\int (1-3Q^2) \varepsilon^2 (y \chi_1 \varphi_y)_y
\end{displaymath}
Thus it follows from \eqref{eq:dZ1.1} and \eqref{BS.1} that
\begin{equation} \label{eq:dZ112} 
|z_{112}| \lesssim \left( \|QQ'y\chi_1 \varphi_y \|_{L^{\infty}}
+\|(y\chi_1 \varphi_y)_y \|_{L^{\infty}}\right) \|\varepsilon\|_{L^2}^2
\lesssim |s|^{-\frac73} \, .
\end{equation}

Also integrating by parts, we have 
\begin{displaymath}
\begin{split}
z_{113}&=-3b\int (2QP_b+bP_b^2)_y \varepsilon^2  y \chi_1 \varphi_y
+3b\int (2QP_b+bP_b^2) \varepsilon^2 (y \chi_1 \varphi_y)_y 
\end{split}
\end{displaymath}
so that 
\begin{equation} \label{eq:dZ113}
|z_{113}| \lesssim |b|\big( \|(2QP_b+bP_b^2)_yy\chi_1 \varphi_y \|_{L^{\infty}}
+\|(2QP_b+bP_b^2)(y\chi_1 \varphi_y)_y \|_{L^{\infty}}\big) \|\varepsilon\|_{L^2}^2
\lesssim |s|^{-\frac{8}3} \, ,
\end{equation}
thanks to \eqref{BS.1} and \eqref{eq:dZ1.1}. 

Similarly, 
\begin{displaymath} 
z_{114}=-4 \int (Q+bP_b)_y \varepsilon^3 y \chi_1 \varphi_y
+2\int (Q+bP_b) \varepsilon^3 (y \chi_1 \varphi_y)_y \, .
\end{displaymath}
Hence, it follows from \eqref{bootstrap.4}, \eqref{BS.1} and \eqref{gnp} (with $p=3$) that
\begin{equation} \label{eq:dZ114}\begin{aligned}
|z_{114}| &\lesssim \big( \|(Q+bP_b)_yy\chi_1 \varphi_y \|_{L^{\infty}}
+\|(Q+bP_b)(y\chi_1 \varphi_y)_y \|_{L^{\infty}}\big) \int |\varepsilon|^3
 \\
 & \lesssim C^\star |s|^{-\frac 23} \|\varepsilon\|_{L^2}^2 \mathcal N(\varepsilon) \lesssim C^{\star}|s|^{-\frac83+\frac{\theta}2}  \lesssim |s|^{-2} .
\end{aligned}\end{equation}
Finally, integration by parts and \eqref{gnp} (with $p=4$) yield 
\begin{equation} \label{eq:dZ115}
|z_{115}| =\left|\frac12\int \varepsilon^4 (y\chi_1 \varphi_y)_y \right| 
\lesssim \|(y\chi_1 \varphi_y)_y \|_{L^{\infty}} 
\|D^{\frac12}\varepsilon\|_{L^2}^2\|\varepsilon\|_{L^2}^2   
\lesssim (C^{\star})^2|s|^{-\frac{13}3+\theta} \lesssim |s|^{-2} .
\end{equation}

Therefore, we deduce combining \eqref{eq:dZ111}-\eqref{eq:dZ115} that 
\begin{equation} \label{eq:dZ11}
|z_{11}| \lesssim |s|^{-2} \, .
\end{equation}

\medskip 
	
\noindent \emph{Estimate for $z_{12}$.} We have 
$\Lambda \varepsilon = \frac \varepsilon 2 + y \varepsilon_y$ and integrating by parts,
\begin{displaymath}
z_{12}=\frac{\lambda_s}{\lambda} \int \varepsilon^2 y\chi_1\varphi_y
-\frac{\lambda_s}{\lambda} \int \varepsilon^2 (y^2\chi_1\varphi_y)_y \, .
\end{displaymath}
Moreover, by \eqref{BS.1} and \eqref{BS.2}, $|\frac{\lambda_s}{\lambda}|\lesssim |\frac{\lambda_s}{\lambda}+b|+|b|\lesssim C^\star |s|^{-1+\frac \theta 2}$, and thus, by
 \eqref{eq:dZ1.1},
\begin{equation} \label{eq:dZ12}
|z_{12}| \lesssim   C^\star |s|^{-1+\frac \theta 2} \big( \|y\chi_1 \varphi_y \|_{L^{\infty}}
+\|(y^2\chi_1 \varphi_y)_y \|_{L^{\infty}}\big)\|\varepsilon\|_{L^2}^2   
\lesssim C^{\star}|s|^{-\frac{8}3+\frac{\theta}2} \lesssim |s|^{-2},
\end{equation}
since $\theta < \frac43$.
\medskip 
	
\noindent \emph{Estimate for $z_{13}$.} 
Since 
\begin{displaymath}
z_{13}=2\big(\frac{\lambda_s}{\lambda}+b\big)\int  \big(\Lambda Q+b\Lambda P_b\big) \varepsilon  y \varphi_y \chi_1 \, ,
\end{displaymath}
we deduce from \eqref{BS.1}, \eqref{BS.2} and \eqref{eq:dZ1.2} that
\begin{equation} \label{eq:dZ13}
\begin{split}
|z_{13}| &\lesssim  C^\star |s|^{-1+\frac \theta 2}\big( \|\Lambda Qy\chi_1 \varphi_y \|_{L^2}
+|b|\|\Lambda P_by\chi_1 \varphi_y \|_{L^2}\big)\|\varepsilon\|_{L^2}   
\\ & \lesssim C^{\star}|s|^{-1+\frac{\theta}{2}}
(|s|^{-\frac53}+|s|^{-\frac43})|s|^{-\frac12}
\lesssim C^{\star}|s|^{-\frac{17}6+\frac{\theta}2}  \lesssim |s|^{-2}\, .
\end{split}
\end{equation}	
\medskip 
	
\noindent \emph{Estimate for $z_{14}$.}
Integrating by parts, we have 
\begin{displaymath}
z_{14}=\big(\frac{x_s}{\lambda}-1\big)\left(2\int  \big( Q+bP_b\big)_y \varepsilon y \chi_1\varphi_y 
-\int  \varepsilon^2 ( y  \chi_1\varphi_y )_y\right),
\end{displaymath}
Moreover, observe from \eqref{BS.1} and $Q\in \mathcal Y_2$ that
\begin{displaymath} 
\left|\int  \big( Q+bP_b\big)_y \varepsilon y  \chi_1\varphi_y \right|
\lesssim \big\|\big( Q+bP_b\big)_y  y \chi_1\varphi_y  \big\|_{L^2} \|\varepsilon\|_{L^2}
\lesssim |s|^{-\frac73-\frac12} 
\end{displaymath}	
and from \eqref{BS.1} and \eqref{eq:dZ1.1} that
\begin{displaymath} 
\left|\int  \varepsilon^2 ( y  \chi_1\varphi_y )_y \right|
\lesssim \|( y  \chi_1\varphi_y )_y\|_{L^{\infty}} \|\varepsilon\|_{L^2}^2
\lesssim |s|^{-\frac73} \, .
\end{displaymath}
Hence, we deduce from  \eqref{BS.2} that
\begin{equation} \label{eq:dZ14}
|z_{14}| \lesssim 
 C^{\star}|s|^{-1+\frac{\theta}2}|s|^{-\frac73} \lesssim |s|^{-2} .
\end{equation}
\medskip 
	
\noindent \emph{Estimate for $z_{15}$.} Recalling \eqref{def.DQb}, we have 
\begin{displaymath}
z_{1,5}=-2\big(b_s+b^2\big)\int  \big( P_b+yP\chi_b'\big) \varepsilon y \chi_1\varphi_y \, .
\end{displaymath}
Moreover, 
\begin{displaymath}
\left|\int P_b  \varepsilon y \chi_1\varphi_y \right| \lesssim \|y \chi_1\varphi_y \|_{L^2}\|\varepsilon\|_{L^2} \lesssim |s|^{-\frac56} \, ,
\end{displaymath}
thanks to \eqref{eq:dZ1.2}, and	
\begin{displaymath}
\left|\int yP\chi_b'  \varepsilon y \chi_1\varphi_y \right| \lesssim 
\|\chi_b'\|_{L^2}\|\varepsilon\|_{L^2} \lesssim |s|^{-1} \, .
\end{displaymath}
Then, we deduce from \eqref{BS.3} that 
\begin{equation} \label{eq:dZ15}
|z_{15}| \lesssim 
 C^{\star}|s|^{-\frac {17}6+\frac{\theta}2} \lesssim |s|^{-2} .
\end{equation}

\medskip 
	
\noindent \emph{Estimate for $z_{16}$.} By using \eqref{lprofile.2}, \eqref{BS.1} and \eqref{eq:dZ1.1}, we get that
\begin{equation} \label{eq:dZ16}
|z_{16}| \lesssim \|y \chi_1\varphi_y \|_{L^{\infty}}\|\Psi_b\|_{L^2}\|\varepsilon\|_{L^2} 
\lesssim |s|^{-\frac83} \, .
\end{equation}
\medskip 
	
\noindent \emph{Estimate for $z_{17}$.} 
First, we compute $\partial_s(\chi_1)(s,y)=\frac23|s|^{-\frac53}y\chi'(y|s|^{-\frac23})$. Then, 
\begin{displaymath}
\left| \int \varepsilon^2\varphi_y\partial_s(\chi_1) \right| \lesssim |s|^{-\frac53} \|\varepsilon\|_{L^2}^2 \lesssim |s|^{-\frac83} \, .
\end{displaymath}
Second, arguing as in \eqref{varphi_s}, we get that	
\begin{displaymath}
\left| \int \varepsilon^2\partial_s(\varphi_y)\chi_1 \right| \lesssim |s|^{-1+\theta} 
\|\frac1{|y|^2}\chi_1\|_{L^{\infty}}\|\varepsilon\|_{L^2}^2 
+|s|^{-1-\theta}\|\frac1{|y|}\chi_1\|_{L^{\infty}}\|\varepsilon\|_{L^2}^2
\lesssim |s|^{-\frac{10}3+\theta}  \, .
\end{displaymath}	
Thus, we deduce that 
\begin{equation} \label{eq:dZ17}
|z_{17}| \lesssim |s|^{-\frac83} \, ,
\end{equation}		
for $|s|$ large enough since $\theta<\frac23$.	

\smallskip 
Therefore, we conclude the proof of \eqref{eq:dZ1} gathering \eqref{eq:dZ11}-\eqref{eq:dZ17}.
	
\medskip
	
\noindent \textit{Proof of \eqref{eq:dZ2}.} Observe from the definition of $Z_2$ that 
\[-\frac12 \frac{dZ_2}{ds}=\int \varepsilon_s \Lambda Q \chi_2 (\varphi-1)
+\int \varepsilon\Lambda Q \chi_2 \varphi_s+\int \varepsilon \Lambda Q (\chi_2)_s (\varphi-1) \, ,\]	
which implies by using \eqref{modulation.4} that 
\[ -\frac12 \frac{dZ_2}{ds}=z_{21}+z_{22}+z_{23}+z_{24}+z_{25}+z_{26}+z_{27}
\]
where
\begin{align*}
z_{21}&=\int V_y \Lambda Q \chi_2 (\varphi-1), &
z_{22}&=\frac{\lambda_s}{\lambda}\int \Lambda \varepsilon  \Lambda Q \chi_2 (\varphi-1), \\ 
z_{23}&=\big(\frac{\lambda_s}{\lambda}+b\big)\int  \Lambda Q_b \Lambda Q \chi_2 (\varphi-1), & 
z_{24}&=\big(\frac{x_s}{\lambda}-1\big)\int  \big( Q_b+\varepsilon\big)_y \Lambda Q \chi_2 (\varphi-1),\\
z_{25}&=-\big(b_s+b^2\big)\int  \frac{\partial Q_b}{\partial b} \Lambda Q \chi_2 (\varphi-1),&
z_{26}&=\int  \Psi_b \Lambda Q \chi_2 (\varphi-1),\\
z_{27}&=\int \varepsilon\Lambda Q \chi_2 \partial_s\varphi+\int \varepsilon \Lambda Q \partial_s(\chi_2) (\varphi-1)\, .
\end{align*}

First, we claim the following estimates
\begin{equation} \label{eq:dZ2.1} 
\|\Lambda Q \chi_2\|_{L^2} \lesssim|s|^{-\frac{3\theta}2}, \quad \|(\Lambda Q \chi_2)_y\|_{L^2} \lesssim|s|^{-\frac{5\theta}2}, \quad \text{and} \quad \|(\Lambda Q \chi_2)_{yy}\|_{L^2} \lesssim  |s|^{-\frac{7\theta}2}  ,
\end{equation}	
which follow directly from the fact that $\Lambda Q \in \mathcal{Y}_2$ and from the definition of  $\chi_2$ in \eqref{chi1chi2}. Recall that $B$ is a fixed universal constant chosen in \eqref{theta.def}.

\medskip 
	
\noindent \emph{Estimate for $z_{21}$.} 
By using the definition of $V$ in \eqref{defV}, we rewrite $z_{21}$ as 
\begin{align*}
z_{21} &= \int (D^1\varepsilon)_y \Lambda Q \chi_2 (\varphi-1)
+\int \big((1-3Q^2)\varepsilon\big)_y \Lambda Q \chi_2 (\varphi-1)
-3\int \big((Q_b^2-Q^2)\varepsilon\big)_y \Lambda Q \chi_2 (\varphi-1)\\
&\quad -3\int\big(Q_b\varepsilon^2\big)_y \Lambda Q \chi_2 (\varphi-1)
-\int(\varepsilon^3)_y \Lambda Q \chi_2 (\varphi-1) \\
&=z_{211}+z_{212}+z_{213}+z_{214}+z_{215} \, .
\end{align*}
We estimate each term separately.

\smallskip
First, we see integrating by parts that 
\begin{equation} \label{eq:Z21.1}
z_{211}=-\int (D^1 \varepsilon) (\Lambda Q \chi_2 )_y (\varphi-1)
-\int (D^1 \varepsilon) \Lambda Q \chi_2  \varphi_y \, .
\end{equation}
We deduce from \eqref{bootstrap.4} and \eqref{eq:dZ2.1} that
\begin{displaymath}
\begin{split}
\left|\int D^1 \varepsilon (\Lambda Q \chi_2 )_y (\varphi-1)\right|
&\lesssim \|D^{\frac12}\varepsilon\|_{L^2}\|D^{\frac12}\big((\Lambda Q \chi_2 )_y(\varphi-1)\big) \|_{L^2}\\ &
\lesssim \mathcal{N}(\varepsilon)\|(\Lambda Q \chi_2 )_y \|_{L^2}^{\frac12}
\|\big((\Lambda Q \chi_2 )_{y}(\varphi-1)\big)_y\|_{L^2}^{\frac12}
\\ & \lesssim C^{\star}|s|^{-1-2\theta} \, .
\end{split}
\end{displaymath}
In order to deal with the second term on the right-hand side of \eqref{eq:Z21.1}, we see integrating by parts again that 
\begin{displaymath} 
-\int D^1 \varepsilon \Lambda Q \chi_2  \varphi_y=
\int \mathcal{H} \varepsilon (\Lambda Q \chi_2)_y  \varphi_y
+\int \mathcal{H} \varepsilon \Lambda Q \chi_2  \varphi_{yy} \, .
\end{displaymath}
From \eqref{BS.1} and \eqref{eq:dZ2.1}, we get
\begin{displaymath}
\left| \int \mathcal{H} \varepsilon (\Lambda Q \chi_2)_y  \varphi_y\right|
\lesssim \|\varepsilon\|_{L^2}\|(\Lambda Q \chi_2)_y\|_{L^2} \lesssim |s|^{-\frac12-\frac{5\theta}2} \lesssim |s|^{-2} \, ,
\end{displaymath}
since $\theta>\frac35$. Moreover, 
\begin{displaymath} 
-\int \mathcal{H} \varepsilon \Lambda Q \chi_2  \varphi_{yy} =
\int \varepsilon \Lambda Q \chi_2  \mathcal{H}\varphi_{yy}
+\int \varepsilon [ \mathcal{H} , \Lambda Q \chi_2] \varphi_{yy} \, .
\end{displaymath}
From the definition of $\varphi$ and Lemma \ref{HY}, we have 
$| \mathcal{H}\varphi_{yy}|\lesssim \varphi_y$. Thus, thanks to \eqref{eq:dZ2.1},
\begin{displaymath}
\left| \int \varepsilon \Lambda Q \chi_2  \mathcal{H}\varphi_{yy}\right|
\lesssim \|\Lambda Q \chi_2\|_{L^2} 
\left( \int \varepsilon^2 \varphi_y \right)^{\frac12} \lesssim 
|s|^{-\frac{3\theta}2}\left( \int \varepsilon^2 \varphi_y \right)^{\frac12}  .
\end{displaymath}
We deduce from the Calder\'on commutator estimate (see \eqref{Calderon} with $l=0$ and $m=1$) that
\begin{displaymath}
\left| \int \varepsilon [ \mathcal{H} , \Lambda Q \chi_2] \varphi_{yy}\right|
\lesssim \|\varepsilon\|_{L^2} \big\|[ \mathcal{H} , \Lambda Q \chi_2] \varphi_{yy}  \big\|_{L^2}  
\lesssim \|\varepsilon\|_{L^2}\|(\Lambda Q \chi_2)_{y}\|_{L^{\infty}}\|\varphi_y\|_{L^2}\lesssim |s|^{-\frac12-3\theta} \, .
\end{displaymath}
Therefore, we conclude gathering those estimates that 
\begin{equation} \label{eq:dZ211}
|z_{211}| \lesssim |s|^{-\frac12-\frac{5\theta}2}+|s|^{-\frac{3\theta}2}\left( \int \varepsilon^2 \varphi_y \right)^{\frac12}  .
\end{equation}

Next, we see integrating by parts that 
\begin{displaymath}
z_{212}=-6\int QQ'\varepsilon \Lambda Q \chi_2 (\varphi-1)
-\int (1-3Q^2) \varepsilon (\Lambda Q \chi_2)_y (\varphi-1)
-\int (1-3Q^2) \varepsilon \Lambda Q \chi_2 \varphi_y  .
\end{displaymath}
Moreover, it follows from \eqref{BS.1} and \eqref{eq:dZ2.1} that 
\begin{displaymath}
\left| \int QQ'\varepsilon \Lambda Q \chi_2 (\varphi-1) \right|
\lesssim \|QQ'\Lambda Q \chi_2 \|_{L^2}\|\varepsilon\|_{L^2} \lesssim |s|^{-\frac12-\frac{13\theta}2} \, ,
\end{displaymath}
\begin{displaymath}
\left| \int (1-3Q^2) \varepsilon (\Lambda Q \chi_2)_y (\varphi-1)\right|
\lesssim \|(\Lambda Q \chi_2)_y \|_{L^2}\|\varepsilon\|_{L^2} \lesssim |s|^{-\frac12-\frac{5\theta}2} \, ,
\end{displaymath}
and
\begin{displaymath}
\left| \int (1-3Q^2) \varepsilon \Lambda Q \chi_2 \varphi_y\right|
\lesssim \|\Lambda Q \chi_2 \|_{L^2} \left( \int \varepsilon^2 \varphi_y \right)^{\frac12}\lesssim |s|^{-\frac{3\theta}2}\left( \int \varepsilon^2 \varphi_y \right)^{\frac12} \, .
\end{displaymath}
Hence, we deduce that 
\begin{equation} \label{eq:dZ212}
|z_{212}| \lesssim |s|^{-\frac12-\frac{5\theta}2}+|s|^{-\frac{3\theta}2}\left( \int \varepsilon^2 \varphi_y \right)^{\frac12}  ,
\end{equation}
for $|s|$ large enough.

Integrating by parts again, we get 
\begin{displaymath}
z_{213}=3b \int (2QP_b+bP_b^2)\varepsilon \big(\Lambda Q \chi_2 (\varphi-1)\big)_y \, ,
\end{displaymath}
so that 
\begin{equation} \label{eq:dZ213}
|z_{213}| \lesssim |b|\|\varepsilon\|_{L^2} \big\|(\frac1{|y|^2}+|b|)\big(\Lambda Q \chi_2 (\varphi-1)\big)_y  \big\|_{L^2} \lesssim |s|^{-\frac52-\frac{3\theta}2}  \, .
\end{equation}

Similarly, 
\begin{equation} \label{eq:dZ214}
\begin{split}
|z_{214}| &=3\left|\int (Q+bP_b)\varepsilon^2 \big(\Lambda Q \chi_2 (\varphi-1)\big)_y \right|  \\ & \lesssim \|\varepsilon\|_{L^2}^2 \big\|(\frac1{|y|^2}+|b|) \big(\Lambda Q \chi_2 (\varphi-1)\big)_y  \big\|_{L^{\infty}} \\
& \lesssim |s|^{-2-2\theta} 
\end{split}
\end{equation}
and 
\begin{displaymath}  
\begin{split}
|z_{215}| &=\left|\int \varepsilon^3 \big(\Lambda Q \chi_2 (\varphi-1)\big)_y \right|  \lesssim \|\varepsilon\|_{L^3}^3 \big\| \big(\Lambda Q \chi_2 (\varphi-1)\big)_y  \big\|_{L^{\infty}} \lesssim |s|^{-2\theta}\|\varepsilon\|_{L^2}^2\|D^{\frac12}\varepsilon\|_{L^2}
\end{split}
\end{displaymath}
thanks to \eqref{gnp} (with $p=3$), so that 
\begin{equation} \label{eq:dZ215}
|z_{215}|  \lesssim C^{\star}|s|^{-2-\frac{3\theta}2} \, .
\end{equation}

Therefore, we conclude gathering \eqref{eq:dZ211}-\eqref{eq:dZ215} that 
\begin{equation} \label{eq:dZ21}
|z_{21}| \lesssim |s|^{-\frac12-\frac{5\theta}2}+|s|^{-\frac{3\theta}2}\left( \int \varepsilon^2 \varphi_y \right)^{\frac12} .
\end{equation}

\medskip 
	
\noindent \emph{Estimate for $z_{22}$.} We see integrating by parts that 
\begin{displaymath} 
z_{22}=-\frac12\frac{\lambda_s}{\lambda}\int \varepsilon  \Lambda Q \chi_2 (\varphi-1)
-\frac{\lambda_s}{\lambda}\int  \varepsilon  y(\Lambda Q \chi_2)_y (\varphi-1)
-\frac{\lambda_s}{\lambda}\int  \varepsilon  y\Lambda Q \chi_2 \varphi_y
 \, .
\end{displaymath}
Hence, it follows from \eqref{BS.1}, \eqref{BS.2} and \eqref{eq:dZ2.1}  that
\begin{equation} \label{eq:dZ22}
\begin{split}
|z_{22}|& \lesssim C^{\star}|s|^{-1+\frac \theta 2}\|\left[\varepsilon\|_{L^2} 
\big(\|\Lambda Q \chi_2\|_{L^2}+\|y(\Lambda Q )_y\chi_2\|_{L^2}  \big)
+  \left( \int \varepsilon^2 \varphi_y \right)^{\frac12}\|y\Lambda Q \chi_2\|_{L^2}\right]  \\ &
\lesssim C^{\star}|s|^{-\frac32-\theta}+C^{\star}|s|^{-1}\left( \int \varepsilon^2 \varphi_y \right)^{\frac12}
\lesssim |s|^{-2} + |s|^{-\frac {3\theta}2} \left( \int \varepsilon^2 \varphi_y \right)^{\frac12}.
\end{split}
\end{equation}
\medskip 
	
\noindent \emph{Estimate for $z_{23}$.} Note that from the definition of $Q_b$, 
\begin{displaymath}
z_{23}=\big(\frac{\lambda_s}{\lambda}+b\big)\int  (\Lambda Q)^2 \chi_2 (\varphi-1)
+\big(\frac{\lambda_s}{\lambda}+b\big)b\int  \Lambda P_b \Lambda Q \chi_2 (\varphi-1)\, .
\end{displaymath}
Moreover, it follows from \eqref{BS.1} and \eqref{BS.2} that
\begin{displaymath}
\left|\big(\frac{\lambda_s}{\lambda}+b\big)\int  (\Lambda Q)^2 \chi_2 (\varphi-1)  \right|
\lesssim C^{\star}|s|^{-1+\frac \theta 2} \int_{y<-\frac{B}4|s|^{\theta}} \frac1{|y|^4} \lesssim 
C^{\star}|s|^{-1-\frac{5\theta}2}
\end{displaymath}
and 
\begin{displaymath}
\begin{split}
\left|\big(\frac{\lambda_s}{\lambda}+b\big)b\int  \Lambda P_b \Lambda Q \chi_2 (\varphi-1)  \right| &\lesssim |b|C^{\star}|s|^{-1+\frac \theta 2}
\left( \int_{y<-\frac{B}4|s|^{\theta}} \frac1{|y|^2} 
+|b|\int_{-2|b|^{-1}<y<-|b|^{-1}}\frac1{|y|}\right) \\ & 
\lesssim C^{\star}|s|^{-2-\frac{\theta}2} \, .
\end{split}
\end{displaymath}
Then, we deduce that 
\begin{equation} \label{eq:dZ23}
|z_{23}| \lesssim C^{\star}|s|^{-2-\frac{\theta}2} \, ,
\end{equation}
for $|s|$ large enough. 
\medskip 
	
\noindent \emph{Estimate for $z_{24}$.} We have after integrating by parts that 
\begin{displaymath} 
z_{24}=\big(\frac{x_s}{\lambda}-1\big)\left(\int  \big( Q+bP_b\big)_y \Lambda Q \chi_2 (\varphi-1)
-\int  \varepsilon (\Lambda Q \chi_2)_y (\varphi-1)-\int  \varepsilon \Lambda Q \chi_2 \varphi_y\right) \, .
\end{displaymath}
Moreover, we get from \eqref{BS.1} and \eqref{eq:dZ2.1} that
\begin{displaymath} 
\left|\int  \big( Q+bP_b\big)_y \Lambda Q \chi_2 (\varphi-1) \right|\lesssim 
  \int_{y<-\frac{B}4|s|^{\theta}} \frac1{|y|^5}+|b| \int_{y<-\frac{B}4|s|^{\theta}} \frac1{|y|^4}+|b|^2
  \int\frac{\chi'(|b|y)}{|y|^2} \lesssim |s|^{-4\theta} \, ,
\end{displaymath}
\begin{displaymath} 
\left|\int  \varepsilon (\Lambda Q \chi_2)_y (\varphi-1) \right|\lesssim 
  \|\varepsilon\|_{L^2} \|(\Lambda Q \chi_2)_y\|_{L^2} \lesssim |s|^{-\frac12-\frac{5\theta}2} 
\end{displaymath}
and 
\begin{displaymath} 
\left|\int  \varepsilon \Lambda Q \chi_2 \varphi_y\right| \lesssim \|\Lambda Q \chi_2\|_{L^2}
\left( \int \varepsilon^2 \varphi_y\right)^{\frac12} \lesssim |s|^{-\frac{3\theta}2}\left( \int \varepsilon^2 \varphi_y\right)^{\frac12} .
\end{displaymath}
Thus we conclude from \eqref{BS.2} that 
\begin{equation} \label{eq:dZ24}
|z_{24}| \lesssim C^{\star}|s|^{-\frac32-2\theta}+C^{\star}|s|^{-1-\theta} \left( \int \varepsilon^2 \varphi_y\right)^{\frac12}\lesssim |s|^{-2} .
\end{equation}
\medskip 
	
\noindent \emph{Estimate for $z_{25}$.} We recall from the definition of $\frac{\partial Q_b}{\partial b}$ in \eqref{def.DQb} that 
\begin{displaymath}
z_{25}=-\big(b_s+b^2\big)\int \big(P_b+yP\chi_b' \big) \Lambda Q \chi_2 (\varphi-1) \, .
\end{displaymath}
Moreover, 
\begin{displaymath}
\left| \int P_b \Lambda Q \chi_2 (\varphi-1) \right| \lesssim \int_{y<-\frac{B}4|s|^{\theta}}\frac1{|y|^2} \lesssim |s|^{-\theta} 
\end{displaymath}
and 
\begin{displaymath}
\left|\int yP\chi_b'  \Lambda Q \chi_2 (\varphi-1)  \right| \lesssim |b|\int \frac{\chi'(|b|y)}{|y|} \lesssim |s|^{-1} \, .
\end{displaymath}
Then, it follows from \eqref{BS.3} that 
\begin{equation} \label{eq:dZ25}
|z_{25}| \lesssim C^{\star}|s|^{-2-\frac{\theta}2}\lesssim |s|^{-2} \, .
\end{equation}
\medskip 
	
\noindent \emph{Estimate for $z_{26}$.} We deduce from \eqref{lprofile.2} and \eqref{eq:dZ2.1} that
\begin{equation} \label{eq:dZ26}
|z_{26}| \lesssim \|\Psi_b\|_{L^2} \|\Lambda Q \chi_2\|_{L^2} \lesssim |s|^{-\frac32-\frac{3\theta}2} \, .
\end{equation}
\medskip 
	
\noindent \emph{Estimate for $z_{27}$.}    First, we compute $\partial_s(\chi_2)(s,y)=\frac4B \theta|s|^{-\theta-1}y\chi'(\frac4By|s|^{-\theta})$. Then, 
\begin{displaymath}
\left| \int \varepsilon \Lambda Q \partial_s(\chi_2) (\varphi-1) \right| \lesssim |s|^{-1-\theta} \|\varepsilon\|_{L^2} \left(\int_{y<-\frac{B}4|s|^{\theta}} \frac1{|y|^2} \right)^{\frac12} \lesssim |s|^{-\frac32-\frac{3\theta}2}\, .
\end{displaymath}
Second, using \eqref{varphi_s} and \eqref{eq:dZ2.1}, we obtain
\begin{displaymath}
\begin{split}
\left| \int \varepsilon\Lambda Q \chi_2 \partial_s\varphi \right|
&\lesssim |s|^{-1+\theta}\|\Lambda Q\chi_2\|_{L^2} \left(\int \varepsilon^2\varphi_y\right)^{\frac12}+|s|^{-1-\theta} \|\Lambda Q\chi_2\|_{L^2}\|\varepsilon\|_{L^2} \\ &
\lesssim |s|^{-1-\frac{\theta}{2}}\left(\int \varepsilon^2\varphi_y\right)^{\frac12}
+|s|^{-\frac32-\frac{5\theta}{2}} \, .
\end{split}
\end{displaymath}
Hence, we deduce that 
\begin{equation} \label{eq:dZ27}
|z_{27}| \lesssim |s|^{-\frac32-\frac{3\theta}2}+|s|^{-1-\frac{\theta}{2}}\left(\int \varepsilon^2\varphi_y\right)^{\frac12} \, .
\end{equation}

\smallskip 

Therefore, we conclude the proof of \eqref{eq:dZ2} gathering \eqref{eq:dZ21}-\eqref{eq:dZ27}.
\end{proof}

\subsection{Coercivity lemma}
We state and prove two consequences of \eqref{coercivity} and \eqref{BS.4}.
\begin{lemma}\label{coer1}
There exits $\kappa>0$ such that 
\begin{equation}\label{coer2}
 \int \left[\big|D^{\frac 12} (\varepsilon \rho)\big|^{2}
 +   \varepsilon^2   \left(\frac {\rho^2+\varphi}2\right)
-3  Q^2 (\varepsilon  \rho)^2 \right]
 \geq \kappa  \int \left[\big|D^{\frac 12} (\varepsilon \rho)\big|^{2} + \varepsilon^2\rho^2\right]
+\mathcal O(|s|^{-2 +\theta})  ,
\end{equation}
where $\rho$ is defined in \eqref{def.rho}, and
\begin{equation}\label{eq:cF}
F  \geq \kappa  \mathcal N(\varepsilon)^2 + \mathcal O(|s|^{-2+\theta}).
\end{equation}
\end{lemma}
\begin{proof}
\emph{Proof of \eqref{coer1}}.
It is clear from the definitions of $\varphi$ and $\rho$ (see~\eqref{def.varphi} and~\eqref{def.rho}) 
and $\theta>\frac 35$ that $\rho^2\lesssim \varphi$.
More precisely, for $y>-2|s|^{\frac 35} $, we have
\[
1- \varphi(s,y) \leq 1- \varphi(s,-2|s|^{\frac 35} ) 
=\frac {\phi(|s|^\theta)-\phi(-\frac 2B |s|^{\frac 35} +|s|^\theta)}{\phi(|s|^\theta)}
\leq |s|^{-\frac 35},
\]
and thus 
\begin{equation}\label{bof}
\varphi(s)\geq  \rho^2 - |s|^{-\frac 35}.
\end{equation}
From the   definition of $\rho$, we have
\begin{align*}
\|(1-\rho) Q\|_{L^2}^2 & \lesssim \|(1-\rho^2) Q\|_{L^2}^2\lesssim
\|Q\|_{L^2(y<-|s|^{\frac 35})}^2 + 
\int \frac 1{(1+y^2)^2 }\frac {y^2 |s|^{-1}}{(1+y^2 |s|^{-1})}\\
& \lesssim |s|^{-\frac 95} + |s|^{-1}\int \frac 1{(1+y^2) }\frac {1}{(1+y^2 |s|^{-1})}\lesssim |s|^{-1}.
\end{align*}
Thus, by \eqref{BS.1}, \eqref{BS.4},
\begin{align}
|(\varepsilon \rho, Q)| 
&= |(\varepsilon , Q) - (\varepsilon,(1-\rho) Q)|
\lesssim |(\varepsilon,Q)|+  \|(1-\rho) Q\|_{L^2}\|\varepsilon\|_{L^2} \nonumber \\
& \lesssim (C^\star)^2 |s|^{-2+\theta} + |s|^{-\frac 12} \|\varepsilon\|_{L^2}
  \lesssim |s|^{-1}. \label{coer.100}
\end{align}
By a similar argument, using \eqref{BS.4} we obtain
\begin{equation} \label{coer.101}
|(\varepsilon \rho, Q')|+|(\varepsilon \rho, \Lambda Q)|\lesssim |s|^{-1} \, .
\end{equation}

Therefore, applying \eqref{coercivity} to $ \varepsilon	 \rho $, we find
\begin{align*}
 \int \left[\big|D^{\frac 12} (\varepsilon \rho)\big|^{2}
 +  \right. & \left. \varepsilon^2   \left(\frac {\rho^2+\varphi}2\right)
-3  Q^2 (\varepsilon  \rho)^2 \right] \\
&\geq  \int \left[\big|D^{\frac 12} (\varepsilon \rho)\big|^{2}
 +    \varepsilon^2 \rho^2-3  Q^2 (\varepsilon  \rho)^2 \right]
  +\mathcal O\left(|s|^{-\frac 35} \| \varepsilon\|_{L^2}^2\right)\\
 & \geq \ \kappa   \int\left[\big|D^{\frac 12} (\varepsilon \rho)\big|^{2} 
 +   \varepsilon^2  \rho^2 \right]   
+\mathcal O(|s|^{-\frac 85})  ,
\end{align*}
and the result follows since $\frac 85 >  2 -\theta$.

\medskip

\noindent \emph{Proof of \eqref{eq:cF}.} First, we observe that
\[\left| \left(\left( Q_{b}+\varepsilon\right)^4-Q_{b}^4- 4 Q_{b}^3 \varepsilon\right) -  6 Q^2 \varepsilon^2\right| \\
 \lesssim |s|^{-1} | \varepsilon|^2  + | \varepsilon|^3 + | \varepsilon|^4,
\]
and thus, using \eqref{gnp} with $p=4$ (which implies $\| \varepsilon\|_{L^4}^2\lesssim \| \varepsilon\|_{L^2}\| \varepsilon\|_{\dot H^{\frac 12}}\lesssim C^\star|s|^{-\frac 32 + \frac \theta 2}$),
\begin{align}
\left|\int\left(\left(Q_{b}+\varepsilon\right)^4-Q_{b}^4-4Q_{b}^3\varepsilon\right)-6Q^2\varepsilon^2\right| 
&\lesssim|s|^{-1}\|\varepsilon\|_{L^2}^2 + \|\varepsilon\|_{L^2}\|\varepsilon\|_{L^4}^2 + \|\varepsilon\|_{L^4}^4
\nonumber 
\\ &\lesssim|s|^{-2}+C^\star |s|^{-2+\frac \theta 2} +(C^\star)^2 |s|^{-3 + \theta }\lesssim |s|^{-2+ \theta}.
\label{paraF}
\end{align}
Next, using \eqref{e:g003}, we have
\[
\int |D^{\frac 12} \varepsilon|^2 \geq \int |D^{\frac 12} \varepsilon|^2 \rho^2 
\geq \int |D^{\frac 12} ( \varepsilon \rho)|^2 + \mathcal O(|s|^{-2 + \theta})
+\mathcal{O}(|s|^{-\frac1{20}}\|D^{\frac12}(\varepsilon\rho)\|_{L^2}^2).
\]
Thus, using also \eqref{bof}, \eqref{bof2} and \eqref{paraF}, we obtain
\begin{align*}
F & \geq \frac \kappa 4 \mathcal N(\varepsilon)^2 
+  \int \left[\left(1-\frac \kappa 4\right) |D^{\frac 12} ( \varepsilon \rho)|^2
+  \left(1-\frac \kappa 4\right) (\varepsilon \rho)^2 - 3 Q^2  (\varepsilon \rho)^2 \right]
\\ & \quad +\mathcal O(|s|^{-2 + \theta})+\mathcal{O}(|s|^{-\frac1{20}}\|D^{\frac12}(\varepsilon\rho)\|_{L^2}^2)\\
& \geq \frac \kappa 8 \mathcal N(\varepsilon)^2 + \mathcal O(|s|^{-2 + \theta}),
\end{align*}
applying \eqref{coercivity} on $ \varepsilon \rho$ as before.

For future reference, we claim the following bound
\begin{equation}\label{eq:bF}
|F|\lesssim   \mathcal N(\varepsilon)^2 +  |s|^{-2+\theta}.
\end{equation}
Indeed, \eqref{eq:bF} is a direct consequence of \eqref{paraF} and the estimate $\int Q^2\varepsilon^2 \lesssim \int \varepsilon^2 \varphi \lesssim \mathcal{N}(\varepsilon)^2$.
\end{proof}

\subsection{Closing  estimates on $\varepsilon$}
Let
\begin{equation}\label{defH}
H =\left(1-\frac {K}{p_{0}}\right) \frac F\lambda + G + \frac 1{2p_{0}}   K^2  -\frac 1{p_{0}}  \frac {KZ}\lambda.
\end{equation}

\begin{proposition}\label{FJH} 
 For $|S_0|$ large enough, possibly depending on $C^{\star}$ and for all $s\in \mathcal I^\star$, the following hold
\begin{itemize}
\item[(i)] Bound.
\begin{equation}\label{eq:bH}
|H(s)| \lesssim  \frac {\mathcal N(\varepsilon)^2} {\lambda} +|s|^{-1+\theta}.
\end{equation}
\item[(ii)] Coercivity. There exists $\kappa>0$ such that
\begin{equation}\label{eq:cH}
H(s) \geq \kappa \frac {\mathcal N(\varepsilon)^2} {\lambda} +\mathcal O(|s|^{-1+\theta }).
\end{equation}
\item[(iii)] Estimate of the time derivatives.
\begin{equation}\label{eq:H}
\frac {dH}{ds} \lesssim C^\star |s|^{-2+\theta }.
\end{equation}
\end{itemize}
\end{proposition}

\begin{proof}
\emph{Proof of \eqref{eq:bH}.}
Recall from \eqref{bg:G}, \eqref{bg:K} and \eqref{eq:Z} that
\begin{equation}\label{autres}  
|G| + K^2 + \frac {|KZ|}{\lambda}  \lesssim |s|^{-\frac 12} +|s|^{-2+2\theta} \lesssim |s|^{-\frac 12}.
\end{equation}
By \eqref{autres} and \eqref{eq:bF}, we obtain \eqref{eq:bH}.

\medskip

\noindent \emph{Proof of \eqref{eq:cH}.}
By \eqref{autres} and \eqref{eq:cF}, we observe that
\[
H(s) \geq \kappa \frac {\mathcal N(\varepsilon)^2} {\lambda} +\mathcal O(|s|^{-1+\theta}) +\mathcal O(|s|^{-\frac 12}),
\]
which implies \eqref{eq:cH}.

\medskip

\noindent \emph{Proof of \eqref{eq:H}.} 
First, from \eqref{eq:F} and~\eqref{eq:G}, we have
\begin{equation}\label{eq:F+G}\begin{aligned}
\frac {d}{ds}\left( \frac F\lambda + G\right)+\frac 1{4\lambda} \int \varepsilon^2 \varphi_y  &\le -2 \int \big|D^{\frac 12} (\varepsilon \rho)\big|^{2} -\int (\varepsilon  \rho)^2 -\int \varepsilon^2 \varphi
+ 6\int Q^2 (\varepsilon  \rho)^2  \\
&\quad -\left(\frac{\lambda_s}{\lambda} + b\right)  \left(\frac{F}\lambda {  -}  K + \frac Z\lambda\right)
 \\ &\quad  + \mathcal O(C^\star|s|^{-2+\theta})
 +  \mathcal O(|s|^{-\frac 1{20}} \|D^{\frac 12}(\varepsilon \rho)\|_{L^2}^2)\,.
 \end{aligned}
\end{equation}

Second,  from \eqref{eq:F} and  \eqref{BS.2}, \eqref{bg:K}, \eqref{eq:Z}, we have
\begin{equation}\label{eq:FFF}
\left| \frac {d}{ds}\left( \frac F\lambda\right) \right| 
\lesssim \frac 1{\lambda} \int \varepsilon^2 \varphi_y +  (C^\star)^2|s|^{-2+\frac {3\theta}2}\,.
\end{equation}
From~\eqref{eq:K} and then \eqref{bg:K}, \eqref{eq:FFF},  \eqref{eq:bF},
\begin{align}
 \left| \frac {d}{ds}\right. & \left.\left( \frac {KF}\lambda\right) + p_{0}\left(\frac{\lambda_s}{\lambda} + b\right)  \frac{F}\lambda\right| 
 \\ & \lesssim \left| K \frac {d}{ds}\left( \frac {F}\lambda\right)\right|
+ \left(\left( \int \varepsilon^2 \varphi_y \right)^{\frac12} + |s|^{  -2+\frac{3\theta}2}\right) \frac F\lambda  \nonumber \\
&\lesssim C^\star |s|^{-1+\theta}  \frac 1{\lambda} \int \varepsilon^2 \varphi_y + 
(C^\star)^3|s|^{-3+\frac {5\theta}2}
+ \left(\left( \int \varepsilon^2 \varphi_y \right)^{\frac12} +  (C^\star)^2|s|^{  -2+\frac{3\theta}2}\right) (C^\star)^2 |s|^{-1+\theta}\nonumber \\
& \leq \frac {p_{0}}{20 \lambda} \int \varepsilon^2 \varphi_y + \mathcal O(|s|^{-2+\theta}).
\label{eq:KF}
\end{align}

Therefore, combining \eqref{eq:F+G} and \eqref{eq:KF},
\begin{align}
\frac {d}{ds}\left( \Big(1-\frac {K}{p_{0}}\Big) \frac F\lambda + G\right)+\frac 1{5\lambda} \int \varepsilon^2 \varphi_y  &\le -2 \int \big|D^{\frac 12} (\varepsilon \rho)\big|^{2} -\int (\varepsilon  \rho)^2 -\int \varepsilon^2 \varphi
 \nonumber  \\
& \quad -\left(\frac{\lambda_s}{\lambda} + b\right)  \left( -  K + \frac Z\lambda\right)
\nonumber \\ & \quad  + \mathcal O(C^\star|s|^{-2+\theta})
 +  \mathcal O(|s|^{-\frac 1{20}} \|D^{\frac 12}(\varepsilon \rho)\|_{L^2}^2)\,. \label{eq:F+KF+G}
\end{align}

From~\eqref{eq:K} and then \eqref{bg:K},
\begin{equation}\label{K2}\begin{aligned}
  \frac 1{2p_{0}} \frac d{ds}  ( K^2 )  & 
 =   \frac {KK_{s}}{p_{0}} =- \left(\frac{\lambda_s}{\lambda} + b\right)  K
 +\mathcal O\left(K \left( \int \varepsilon^2 \varphi_y \right)^{\frac12}  \right)+\mathcal O\left(K(C^\star)^2 |s|^{  -2+\frac{3\theta}2} \right)  \\
 &  \leq- \left(\frac{\lambda_s}{\lambda} + b\right)  K + \frac 1{20\lambda}  \int \varepsilon^2 \varphi_y
 + \mathcal O\left(\lambda K^2 \right)+\mathcal O\left(K(C^\star)^2 |s|^{  -2+\frac{3\theta}2} \right) \\
 & \leq- \left(\frac{\lambda_s}{\lambda} + b\right)  K + \frac 1{20\lambda}  \int \varepsilon^2 \varphi_y
 +\mathcal O\left(|s|^{-2+\theta} \right).
 \end{aligned}\end{equation}

Moreover, from~\eqref{eq:K} and~\eqref{eq:dZ1},~\eqref{eq:dZ2},
\begin{equation}\label{KZ}\begin{aligned}
  -\frac 1{p_{0}}\frac d{ds} \left(  \frac {KZ}\lambda \right)  
& =  - \frac 1{p_{0}}\frac {K_{s}Z}\lambda - \frac 1{p_{0}}\frac {KZ_{s}}\lambda
 + \frac 1{p_{0}} \frac{\lambda_{s}}{\lambda} \frac {KZ}{\lambda}
 \\ & = \left(\frac{\lambda_s}{\lambda} + b\right)  \frac  Z\lambda  
   + \mathcal O\left(  \left( \int \varepsilon^2 \varphi_y \right)^{\frac12}\frac {Z}{\lambda}  \right)
   + \mathcal O\left( (C^\star)^2|s|^{-2+\frac {3\theta} 2} \frac {Z}{\lambda}\right)
\\& + \mathcal O\left(\frac {K}{\lambda} |s|^{-2}\right) + \mathcal O\left(\frac {K}{\lambda} |s|^{-\frac {3\theta}2}\left( \int \varepsilon^2 \varphi_y \right)^{\frac12}  \right)+ \mathcal O\left(\frac{\lambda_s}{\lambda} \frac {KZ}{\lambda}\right).
 \end{aligned}\end{equation}
Thus, using \eqref{BS.2}, and then \eqref{bg:K}, \eqref{eq:Z}, 
\begin{equation}\label{KZbis}\begin{aligned}
  -\frac 1{p_{0}}\frac d{ds} \left(  \frac {KZ}\lambda \right)  
  & \leq \left(\frac{\lambda_s}{\lambda} + b\right)  \frac  Z\lambda 
  + \frac 1{20\lambda}  \int \varepsilon^2 \varphi_y  
  + \mathcal O\left(  \frac {Z^2}{\lambda} \right) + \mathcal O\left(|s|^{-3\theta}\frac {K^2}{\lambda} \right)
\\&   \quad+ \mathcal O\left( (C^\star)^2|s|^{-2+\frac {3\theta} 2} \frac {Z}{\lambda} \right)
  + \mathcal O\left(\frac {K}{\lambda} |s|^{-2}\right)  + \mathcal O\left(C^\star|s|^{-1+\frac \theta 2} \frac {KZ}{\lambda}\right)\\
   & \leq \left(\frac{\lambda_s}{\lambda} + b\right)  \frac  Z\lambda 
  + \frac 1{20\lambda}  \int \varepsilon^2 \varphi_y  + \mathcal O(|s|^{-2+\theta }).
 \end{aligned}\end{equation}
 
 Thus, setting $H = \left(1-\frac {K}{p_{0}}\right) \frac F\lambda + G + \frac 1{2p_{0}}   K^2  -\frac 1{p_{0}}  \frac {KZ}\lambda$,
we see that  
 \begin{equation}\label{eq:F+G2}\begin{aligned}
\frac {d}{ds}H+\frac 1{10\lambda} \int \varepsilon^2 \varphi_y 
 &\le -2 \int \big|D^{\frac 12} (\varepsilon \rho)\big|^{2} -\int (\varepsilon  \rho)^2 -\int \varepsilon^2 \varphi
+ 6\int Q^2 (\varepsilon  \rho)^2   
\\ &\quad   
   + \mathcal O(C^\star|s|^{-2+\theta})
 +  \mathcal O(|s|^{-\frac 1{20}} \|D^{\frac 12}(\varepsilon \rho)\|_{L^2}^2)\,.
 \end{aligned}
 \end{equation}
Using \eqref{coer2}, we obtain
\begin{equation}\label{pourH}
\frac {d}{ds}H+\frac 1{10\lambda} \int \varepsilon^2 \varphi_y +
\frac \kappa  2 \int \left[\big(D^{\frac 12} (\varepsilon \rho)\big)^{2} + \varepsilon^2\rho^2\right]
\lesssim C^\star|s|^{-2 +\theta},
\end{equation}
which  implies \eqref{eq:H}.
\end{proof}

Observe by \eqref{epsilton.initial}, \eqref{bootstrap.2} and \eqref{eq:bH} that
\begin{equation}\label{bHn}
|H(S_{n})|\lesssim |S_{n}|^{-1+\theta}.
\end{equation}
Let $s\in \mathcal I^\star$.
Integrating \eqref{eq:H} on $[S_{n},s]$, we obtain
\[
H(s) - H(S_{n}) \lesssim C^\star |s|^{-1+\theta}.
\]
Therefore, by \eqref{bHn}, \eqref{eq:cH} and $\lambda(s) \sim |s|^{-1}$, we obtain
\[
\mathcal N(\varepsilon)^2 \leq C_{2}  C^\star |s|^{-2+\theta},
\]
where the constant $C_{2}$ above does not depend on $C^\star$ nor on $n$.
 Now, we fix the positive constant $C^{\star}$ so that $C^\star =4C_{2}$. Then, we deduce
\begin{equation}\label{closing.0}
\mathcal N(\varepsilon) \le \frac{C^{\star}}2|s|^{-1+\frac \theta 2} \, , 
\end{equation}
which strictly improves \eqref{bootstrap.4}.

\section{Parameters estimates}\label{S5}
In this section, we finish the proof of Proposition~\ref{bootstrapprop} by closing the estimates for the parameters $\mu_n(s)$, $\lambda_n(s),b_n(s)$ and $x_n(s)$. Indeed, we strictly improve estimates \eqref{bootstrap.1}, \eqref{bootstrap.2} and \eqref{bootstrap.3}. 

We consider an arbitrary $n \ge n_{0}$. As in Sect.~\ref{S4}, for the sake of simplicity, we will omit the subscript $n$ in this section and write $\mu$, $\lambda$, $b$, $x$ and $\varepsilon$ for $\mu_n$, $\lambda_n$, $b_n$, $x_n$ and $\varepsilon_n$. 
 Recall that the constant $C^\star$ has been fixed in Sect.~\ref{S4}, thus from now on, we omit to mention dependency in $C^\star$.

\subsection{Refined scaling control}
Recall that we set
\begin{equation}\label{defJ}
\rho(y)=\int_{-\infty}^y \Lambda Q(y') dy', \quad
J(s)= \int \varepsilon(s,y) \rho(y) \chi(-y|s|^{-\frac 23}) dy,
\end{equation}
and
\begin{equation}\label{defl0}
\mu(s) =|1-J(s)|^{\frac 1{p_{0}}} \lambda(s).
\end{equation}
Such a functional $J(s)$ was introduced in \cite{MaMejmpa,MaMe} in similar context (see also \cite{BoSoSt}).
It corresponds to the fact that the orthogonality $( \varepsilon, \Lambda Q)$ is not especially interesting for the estimate of $\lambda_{s}$. Indeed, from \eqref{BS.1} and \eqref{BS.2}, the best information one can  get is 
$|\frac{\lambda_s}{\lambda}|\lesssim C^\star |s|^{-1+\frac \theta 2}$. Thus, one needs to introduce a functional related to the cancellation $(\Lambda Q,Q)=0$.
Unlike for NLS-type equation, where such cancellation can be used easily (see e.g. \cite{RaSz}), the fact that $\rho\not \in L^2$ creates serious difficuty and imposes the use of cut-off term in the definition of $J$.
 We claim the following result.

\begin{lemma}
For $|S_0|$ large enough and for all $s\in \mathcal I^\star$,
\begin{equation}\label{bJ}
|J(s)|\lesssim  |s|^{-\frac 23 + \frac \theta2},
\end{equation}
\begin{equation}\label{dsJ}
\left|\frac{dJ}{ds}  -  p_{0}\left(\frac{\lambda_s}{\lambda}+b\right) \right| 
 \lesssim  |s|^{-\frac 43 + \frac \theta 2},
\end{equation}
\begin{equation}\label{dsl0}
\left|\frac{\mu_s}{\mu}+b\right| 
 \lesssim |s|^{-\frac 43 + \frac \theta 2}.
\end{equation}
\end{lemma}

\begin{remark} Since $\theta<\frac 23$, we see that from \eqref{dsl0}, $\left|\frac{\mu_s}{\mu}+b\right| \ll |s|^{-1}$, while
$|b|\sim |s|^{-1}$, thus this estimate is much more precise than $|\frac{\lambda_s}{\lambda} |\lesssim C^\star |s|^{-1+\frac \theta 2}$.
\end{remark}

\begin{proof}
First, we see by using the decay of $\rho$ and \eqref{bootstrap.4} that 
\begin{displaymath}
|J(s)|\lesssim \int_{y<0} \frac{|\varepsilon(s,y)|}{1+|y|} dy + \int_{0<y<2 |s|^{\frac 23}} |\varepsilon(s,y)| dy 
\lesssim \|\varepsilon(s)\|_{L^2}+ |s|^{\frac 13} \mathcal N(\varepsilon)\lesssim |s|^{-\frac 23 + \frac \theta2} \, ,
\end{displaymath}
which proves \eqref{bJ}.

\medskip

Now, we use    \eqref{modulation.4} to compute $\frac{dJ}{ds}$,
\begin{align*}
 \frac{dJ}{ds} & = \int \varepsilon_s \rho   \chi(-y|s|^{-\frac 23}) 
- \frac 23|s|^{-1} \int  y  |s|^{-\frac 23}  \chi'(-y|s|^{-\frac 23}) \rho \varepsilon  \\
 & = j_1 + j_2+ j_3+j_4+j_5+j_6+j_7\end{align*}
where
\begin{equation*} 
\begin{aligned}
& j_1=\int V_y \rho \chi(-y|s|^{-\frac 23}) ,
& j_2& =   \frac{\lambda_s}{\lambda} \int\Lambda \varepsilon \rho \chi(-y|s|^{-\frac 23}),\\
& j_3 =  (\frac{\lambda_s}{\lambda}+b) \int  \Lambda Q_b  \rho \chi(-y|s|^{-\frac 23}),
& j_4 &= 
   (\frac{x_s}{\lambda}-1)\int (Q_b+\varepsilon)_y  \rho \chi(-y|s|^{-\frac 23}) ,\\
& j_5 = 
     -(b_s+b^2)\int \frac{\partial Q_b}{\partial b}  \rho \chi(-y|s|^{-\frac 23}),
 &  j_6 &= \int  \Psi_b  \rho \chi(-y|s|^{-\frac 23}),\\
 & j_7 = -\frac 23|s|^{-1} \int   y|s|^{-\frac 23}  \chi'(-y|s|^{-\frac 23}) \varepsilon \rho  .
\end{aligned}
\end{equation*}

\medskip

\noindent \emph{Estimate for $j_{1}$.} 
\[
\begin{split}
j_1&=-\int \mathcal{L}\varepsilon \big(\rho \chi(-y|s|^{-\frac 23}) \big)_y
+3\int (Q_b^2-Q^2) \varepsilon \big(\rho \chi(-y|s|^{-\frac 23}) \big)_y
+\int \big(3Q_b\varepsilon^2+\varepsilon^3\big)\big(\rho \chi(-y|s|^{-\frac 23}) \big)_y \\ 
&=j_{1,1}+j_{1,2}+j_{1,3} \, .
\end{split}\]
First, using $\rho'=\Lambda Q$ and $\mathcal L\Lambda Q=-Q$,
\begin{align*}
 j_{1,1}& =
 - \int \varepsilon  \mathcal{L} (\Lambda Q  \chi(-y|s|^{-\frac 23})) + |s|^{-\frac23}\int \varepsilon \mathcal{L}(\rho  \chi'(-y|s|^{-\frac 23}))\\
 & =  \int \varepsilon Q+\int \varepsilon   \mathcal{L} (\Lambda Q  (1-\chi(-y|s|^{-\frac 23}))) + |s|^{-\frac 23} \int \varepsilon  \mathcal{L} (\rho  \chi'(-y|s|^{-\frac 23})) \, .
\end{align*}
Note from \eqref{BS.1} that
\[
\left| \int \varepsilon    D^1 (\Lambda Q  (1-\chi(-y|s|^{-\frac 23})) \right| 
\lesssim \|\varepsilon\|_{L^2} \|(\Lambda Q  (1-\chi(-y|s|^{-\frac 23}))_y\|_{L^2}
\lesssim   |s|^{-\frac {13}6} \, ,
\]
\begin{align*}
\left| \int \varepsilon    \Lambda Q (1-3Q^2) (1-\chi(-y|s|^{-\frac 23})  \right| 
&\lesssim \|\varepsilon\|_{L^2} \|  \Lambda Q (1-3Q^2)(1-\chi(-y|s|^{-\frac 23}))\|_{L^2}\\
&\lesssim \|\varepsilon\|_{L^2}  |s|^{-1}\lesssim |s|^{-\frac 32} \, ,
\end{align*}
\begin{align*}
\left|\int \varepsilon  D^1(\rho  \chi'(-y|s|^{-\frac 23}))\right| &
\lesssim \|	\varepsilon \|_{L^2} \|(\rho  \chi'(-y|s|^{-\frac 23}))_{y}\|_{L^2}
\lesssim \|	\varepsilon \|_{L^2}  |s|^{-\frac 13}\lesssim |s|^{-\frac 56} \, ,
\end{align*}
and by \eqref{bootstrap.4}
\begin{align*}
\left|\int \varepsilon   (1-3Q^2)(\rho  \chi'(-y|s|^{-\frac 23}))\right| &
\lesssim \mathcal N(	\varepsilon)\|\chi'(-y|s|^{-\frac 23}) \|_{L^2}
\lesssim \mathcal N(	\varepsilon)  |s|^{\frac 13}\lesssim |s|^{-\frac 23 + \frac \theta 2} \, .
\end{align*}
Thus, using also   $|\int \varepsilon Q|\lesssim |s|^{-2+\theta}
\lesssim |s|^{-\frac 43 + \frac \theta 2}$ (see \eqref{BS.4}), we deduce that

\[
|j_{1,1}|\lesssim |s|^{-\frac 43+ \frac \theta 2} \, .
\]
Next,
\begin{align*}
|j_{1,2}|& \lesssim |b| \|\varepsilon\|_{L^2}\|\big(\rho \chi(-y|s|^{-\frac 23}) \big)_y\|_{L^2}
\lesssim |s|^{-\frac 32} \, ,
\end{align*}
and using \eqref{gnp}
\begin{equation*}
|j_{1,3}|\lesssim \big(\mathcal N(\varepsilon)^2 + |b| \|\varepsilon\|_{L^2}^2+ \| \varepsilon\|_{L^2}^2  \|D^{\frac12} \varepsilon\|_{L^2}\big) \|(\rho \chi(-y|s|^{-\frac 23}))_{y} \big\|_{L^\infty}
\lesssim |s|^{-2+\theta} \lesssim |s|^{-\frac 43} \, .
\end{equation*}

Hence, we deduce combining those estimates that 
\begin{equation} \label{eq:j1}
|j_{1}| \lesssim |s|^{-\frac43+\frac{\theta}2} \, .
\end{equation}

 \medskip

\noindent \emph{Estimate for $j_{2}$.} 
\[
j_2 =  -\frac{\lambda_s}{\lambda} \int  \varepsilon   \Lambda (\rho \chi(-y|s|^{-\frac 23}))
=- \frac 12 \frac{\lambda_s}{\lambda}  J
 - \frac{\lambda_s}{\lambda} \int  \varepsilon y (\rho \chi(-y|s|^{-\frac 23}))_y.
\]
Note that  
\begin{align*}
\left| \int  \varepsilon y (\rho \chi(-y|s|^{-\frac 23}))_y\right| 
&\lesssim \left| \int  \varepsilon y \Lambda Q \chi(-y|s|^{-\frac 23})\right| 
+|s|^{-\frac 23} \left| \int  \varepsilon y \rho \chi'(-y|s|^{-\frac 23}) \right| \\
& \lesssim \|\varepsilon\|_{L^2}+|s|^{\frac 13} \mathcal N( \varepsilon)\lesssim |s|^{-\frac 23 + \frac \theta 2}.
\end{align*}
Thus, by   \eqref{bJ} and then \eqref{BS.2},
\begin{equation} \label{eq:j2}
\left|j_2 \right|\lesssim\left(\left|\frac{\lambda_s}{\lambda}+ b\right| +|b|\right)  |s|^{-\frac 23 + \frac \theta 2}  \lesssim |s|^{-\frac 53 + \theta}\lesssim |s|^{-\frac 43+ \frac \theta 2} \, .
\end{equation}

\medskip

\noindent \emph{Estimate for $j_{3}$.} 
For the term $j_3$, we first note from \eqref{nlprofile.2} that
\[ \int \rho \Lambda Q  =\frac 12 \Big(\lim_{y \to +\infty} \rho(y)\Big)^2 = \frac 18\left( \int Q\right)^2=p_{0} \, .
\]
By the decay properties of $Q$, $P$ and $\rho$,
\begin{equation}\label{avr}
\left|\int (\Lambda P_b ) \rho \chi(-y |s|^{-\frac 23}) \right|
\lesssim \int_{-2 |b|^{-1}<y<|s|^{\frac 23}} \frac {dy}{1+|y|} \lesssim |\ln|b||+|\ln|s|| \lesssim |\ln|s||\,.
\end{equation}
Thus,
\begin{align*}
\int (\Lambda Q_b) \rho \chi(-y|s|^{-\frac 23}) 
& = \int \rho  (\Lambda Q)  \chi(-y|s|^{-\frac 23})  
+ b \int (\Lambda P_b) \rho \chi(-y|s|^{-\frac 23}) 
\\&
= \int \rho \Lambda Q +\int \rho  \Lambda Q \big(\chi(-y|s|^{-\frac 23})-1\big) + \mathcal O(|s|^{-1}|\ln |s||)
\\&
= p_{0} + \mathcal O(|s|^{-\frac 23}) \, .
\end{align*}
Therefore, by~\eqref{BS.2},
\begin{equation} \label{eq:j3}
\left| j_{3} - p_{0}\left(\frac{\lambda_s}{\lambda}+ b\right) \right| \lesssim |s|^{-\frac 53 + \frac \theta2}
\lesssim |s|^{-\frac 43+ \frac \theta 2}.
\end{equation}

\medskip

\noindent \emph{Estimate for $j_{4}$.} 
Similarly, since $\int \rho Q'=-\int \rho' Q = -\int Q \Lambda Q =0$, we have
\[
\int Q_b' \rho \chi(-y|s|^{-\frac 23}) 
=   \mathcal O(|s|^{-\frac 23}).
\]
Moreover, 
\[
\int \varepsilon_y \rho \chi(-y|s|^{-\frac 23}) = -\int \varepsilon \rho' \chi(-y|s|^{-\frac 23})
+|s|^{-\frac 23}\int \varepsilon \rho \chi'(-y|s|^{-\frac 23}),
\]
and so
\[
\left| \int \varepsilon_y \rho \chi(-y|s|^{-\frac 23}) \right| \lesssim \| \varepsilon\|_{L^2}\lesssim |s|^{-\frac 12}.
\]
From \eqref{BS.2}, we obtain
\begin{equation} \label{eq:j4}
|j_{4}|\lesssim |s|^{-1+\frac \theta 2} |s|^{-\frac 12} \lesssim |s|^{-\frac 32 + \frac\theta 2} \, .
\end{equation}

\medskip

\noindent \emph{Estimate for $j_{5}$.} 
By \eqref{def.DQb} and the decay properties of $P$ and $\rho$ (as in \eqref{avr}), one has
\[
\left| \int \frac{\partial Q_b}{\partial b}  \rho \chi(-y|s|^{-\frac 23})\right|
\lesssim \int |P_{b}||\rho|\chi(-y|s|^{-\frac 23})
+\int_{-2|b|^{-1}<y<-|b|^{-1}} |\rho|
\lesssim \ln |s| \, .
\]
Thus, using \eqref{BS.5},
\begin{equation} \label{eq:j5}
|j_{5}|\lesssim |s|^{-2+\frac \theta 2} \ln|s|\lesssim |s|^{-\frac 43 + \frac \theta 2} \, .
\end{equation}
 
\medskip

\noindent \emph{Estimate for $j_{6}$.}  
Next, by \eqref{lprofile.2} and the properties of $\rho$,
\begin{equation} \label{eq:j6}
|j_{6}|=\left|  \int  \Psi_b \rho \chi(-y|s|^{-\frac 23})\right|
\lesssim \|\Psi_{b} \|_{L^2} \|\rho  \chi(-y|s|^{-\frac 23})\|_{L^2}
 \lesssim  \|\Psi_b\|_{L^2} |s|^{\frac 13} \lesssim |s|^{-\frac 76}\lesssim |s|^{-\frac 43 + \frac \theta 2} \, .
\end{equation}
 
\medskip

\noindent \emph{Estimate for $j_{7}$.} Note   by the definition of $\chi$,
\begin{equation} \label{eq:j7}
| j_{7}| \lesssim |s|^{-1}   \int_{|s|^{\frac 23}<y<2|s|^{\frac 23}} |\varepsilon|  
\lesssim |s|^{-\frac 23}\|\varepsilon\|_{{L^2}}\lesssim |s|^{-\frac 76} \lesssim |s|^{-\frac 43 + \frac \theta 2} \, .
\end{equation}

Therefore, combining estimates \eqref{eq:j1}-\eqref{eq:j7}, we obtain \eqref{dsJ}.

\medskip

Now, we prove \eqref{dsl0}.
Since $\mu =(1-J)^{\frac 1{p_{0}}} \lambda$, by direct computations 
and then \eqref{dsJ}, \eqref{BS.2} and \eqref{bJ}, we have
\begin{align*}
\left|p_{0}(1-J) \left(\frac {\mu_{s}}  {\mu} +b\right)\right|
& =\left|p_{0}(1-J)\left(\frac {\lambda_{s}}{\lambda}+b\right) -   {J_{s}}\right|\\
&\le \left| p_{0}\left(\frac {\lambda_{s}}{\lambda}+b\right) - {J_{s}}\right|
+ p_{0} \left|J \left(\frac {\lambda_{s}}{\lambda}+b\right)\right|\\
& \lesssim |s|^{-\frac 43 + \frac\theta 2} + |s|^{-\frac 53 + \theta}
\lesssim  |s|^{-\frac 43 +\frac \theta 2} \, ,
\end{align*}
which proves \eqref{dsl0}, since $\theta <\frac 23$.
 \end{proof}

\subsection{Closing parameter estimates} 
Let $S_{0}<0$ and $n_0\in \mathbb N$, with  $|S_0|$ and $n_{0}$ to be fixed large enough.
Set
\[
\alpha = \frac 16 - \frac \theta 4\in \left(0,\frac 1{60}\right) \quad \hbox{from   the condition \eqref{theta.def} on $\theta$.}
\]
Let $s \in  \mathcal I^\star$. From~\eqref{BS.2},
\begin{equation} \label{closing.1}
\Big| \frac{x_s}{\lambda}-1\Big|  \lesssim  |s|^{-1+\frac \theta 2} \,.
\end{equation}
Moreover, from \eqref{bJ}, \eqref{dsl0} and $b=-\lambda$,
\begin{equation}\label{closing.2b}
\left| \frac {\mu}{\lambda} -1 \right| \lesssim |J|\lesssim |s|^{-\frac 23 + \frac \theta2}\quad \text{and} \quad
\Big|  {\mu_s}-{\mu^2}\Big| \le |\mu|\left|\frac{\mu_s}{\mu}+b\right|+|\mu\lambda|\left|\frac{\mu}{\lambda}-1\right| \lesssim |s|^{-2-2\alpha} \,. 
\end{equation}

First, we observe from \eqref{bootstrap.1}, \eqref{BS.1} and \eqref{closing.2b} that  
\begin{equation} \label{closing.200}
\left| \lambda(s)-\frac1{|s|} \right| \le \lambda(s)\left|1-\frac{\mu(s)}{\lambda(s)}\right|+\left| \mu(s)-\frac1{|s|}\right|
\lesssim |s|^{-1-\alpha} < \frac 12 |s|^{-1-\frac{\alpha}{2}}\, ,
\end{equation}
 for $|S_0|$ large enough. Since $b=-\lambda$, this
strictly improves estimates \eqref{bootstrap.2} on $b(s)$ and $\lambda(s)$.

Second, we improve the estimate on $|x(s)+\ln(|s|)|$ in \eqref{bootstrap.3}. 
From \eqref{closing.1}, \eqref{BS.1} and \eqref{closing.200}, we have
\[
\Big|x_s+\frac 1{s}\Big|\leq   |x_s - \lambda (s)|+\left|\lambda(s)-\frac 1{|s|}\right| 
\lesssim   |s|^{-1-\alpha} \, .
\]
Integrating on $[S_n,s]$, using $x^{in}=-\ln(|S_n|)$, we obtain
\begin{equation}\label{closing.3}
|x(s) +\ln(|s|)|\lesssim   |s|^{-\alpha} 
\, .
\end{equation}
As before, this   strictly improves \eqref{bootstrap.3} on $x(s)$ for $|S_0|$ large enough.

\medskip

The last step of the proof is to strictly improve \eqref{bootstrap.1} on $\mu$ by adjusting the initial value of $\lambda$, i.e. $\lambda^{in}$ using a contradiction argument. See    \cite{CoMaMe} for a similar argument.
Note that  such an indirect argument is needed because the estimate \eqref{closing.2b} is relatively tight, which prevents us from choosing explicitly the value of $\lambda^{in}$.
Let 
\begin{equation}\label{corresp}
\mu^{in} = \mu(S_n)=(1-J(S_{n}))^{\frac 1{p_{0}}} \lambda^{in}.
\end{equation}
First, we prove that there exists at least a choice of $\mu^{in}\in [|S_n|^{-1}-|S_n|^{-1-\alpha},|S_n|^{-1}+|S_n|^{-1-\alpha}]$ that allows to strictly improve the bootstrap bound \eqref{bootstrap.1} on $\mathcal I^\star$ using \eqref{closing.2b}.
Second, since the dependency of $\mu^{in}$ on $\lambda^{in}$ in  \eqref{corresp} is implicit (recall from \eqref{decomposition.2} that $\varepsilon(S_n)$ and thus $J(S_n)$ depends on $\lambda^{in}$), we need to check that the the image of the map $$:\lambda^{in}\in [|S_n|^{-1}-|S_n|^{-1-\frac \alpha2},|S_n|^{-1}+|S_n|^{-1-\frac \alpha2}]\mapsto \mu^{in} $$
contains the interval $ [|S_n|^{-1}-|S_n|^{-1-\alpha},|S_n|^{-1}+|S_n|^{-1-\alpha}]$.

\medskip

Assume for the sake of contradiction that for any  $\mu^\sharp\in [-1,1]$, the choice $\mu^{in}=|S_n|^{-1} + \mu^\sharp |S_n|^{-1-\alpha}$ leads to $S^\star_n(\mu^\sharp) = S^\star_n<S_{0}$.
By \eqref{closing.0}, \eqref{closing.200} and \eqref{closing.3}, we have strictly improved   \eqref{bootstrap.2}, \eqref{bootstrap.3} and \eqref{bootstrap.4}. Thus, at $S_n^\star$, \eqref{bootstrap.1} is saturated, which means that  \begin{equation}\label{e:sat}
\Big| \mu(S^\star_n)- |S^\star_n|^{-1}\Big| = |S_n^{\star}|^{-1-\alpha}.
\end{equation}
Define the function $\Phi$ by 
\[
\Phi \, : \, \mu^\sharp\in [-1,1] \mapsto \Big( \mu(S^\star_n)- |S^\star_n|^{-1}\Big) |S_n^{\star}|^{1+\alpha}
\in \{-1,1\}.
\]
Set
\[
f(s) = \Big(\mu(s)+s^{-1}\Big)^2 (-s)^{2+2\alpha}.
\]
Then, using \eqref{closing.2b} and \eqref{bootstrap.1},
\begin{align*}
f'(s) & = 2 \Big(\mu(s)+s^{-1}\Big) \Big(\mu_s(s)-s^{-2}\Big)(-s)^{2+2\alpha}
- 2(\alpha+1)\Big(\mu(s)+s^{-1}\Big)^2 (-s)^{1+2\alpha}\\
& = 2 \Big(\mu(s)+s^{-1}\Big) \left( \mu^2(s) - s^{-2} + \mathcal O(|s|^{-2-2\alpha})\right)(-s)^{2+2\alpha}
\\ &\quad - 2(1+\alpha)\Big(\mu(s)+s^{-1}\Big)^2 (-s)^{1+2\alpha}\\
& = 2 \Big(\mu(s)+s^{-1}\Big) \left( \mu^2(s) - s^{-2} \right)(-s)^{2+2\alpha}
+\mathcal O(|s|^{-1-\alpha})
\\ & \quad - 2(1+\alpha)\Big(\mu(s)+s^{-1}\Big)^2 (-s)^{1+2\alpha}\\
& = 2 \Big(\mu(s)+s^{-1}\Big)^2 \left( \mu(s)  + (1+2\alpha) s^{-1}  \right)(-s)^{2+2\alpha} + \mathcal O(|s|^{-1- \alpha}).
\end{align*}
In particular,  by \eqref{e:sat},
\begin{equation}\label{closing.6}
f(S^\star_n) =1 \quad \text{and} \quad
f'(S^\star_n) = 4\alpha |S^\star_n|^{-1} + \mathcal O(|S^\star_n|^{-1-\alpha})
> 3\alpha |S^\star_n|^{-1},
\end{equation}
for $|S^\star_n|>|S_0|$ large enough.
It follows from the transversality property \eqref{closing.6} that  the map $\mu^\sharp\mapsto S^\star_n$ is continuous.
Indeed,  first, let $\mu^\sharp\in (-1,1)$ so that $S_n<S_n^\star$ and  let $0<\epsilon<S_n^\star-S_n$ small. By \eqref{closing.6}, there exists $\delta>0$ such that 
$f(S_n^\star+\epsilon)>1+\delta$ and, for all $s\in [S_n,S_n^\star-\epsilon]$, 
$f(s)<1-\delta$. By continuity of the flow for the mBO equation, there exists $\eta>0$ such for all $\tilde \mu^\sharp\in (-1,1)$ with $|\tilde\mu^\sharp-\mu^\sharp|<\eta$, the corresponding $\tilde f$ satisfies $|\tilde f(s)-f(s)|<\delta/2$
on $[S_n,S_n^\star+\epsilon]$. This has two consequences : first, for all $s\in [S_n,S_n^\star-\epsilon]$, $\tilde f(s)<1-\frac \delta 2$ and thus
$\tilde S_n^\star> S_n^\star-\epsilon$; second, $\tilde f(S_n^\star+\epsilon)>1+\frac \delta 2$ and thus
$\tilde S_n^\star\leq S_n^\star+\epsilon$, which proves continuity of $\mu^\sharp\mapsto S_n^\star$ on $(-1,1)$. 

Moreover, we see that 
for $\mu^\sharp=-1$ and $\mu^\sharp=1$, $f(S_n)=1$ and $f'(S_n)>0$ (see \eqref{closing.6}), and thus in this case $S_n^\star=S_n$. By similar arguments as before, $\mu^\sharp\mapsto S_n^\star$ is continuous on $[-1,1]$.

Therefore, the function $\Phi$ is also continuous from $[-1,1]$ to $\{1,-1\}$, but this is a contradiction with $\Phi(-1)=-1$ and $\Phi(1)=1$.

It follows that there exists at least a value of $\mu^\sharp\in (-1,1)$ such that 
$\mu^{in}=|S_n|^{-1} + \mu^\sharp |S_n|^{-1-\alpha}$ leads to 
$S_{n}^\star=S_{0}$. 

\medskip

As announced, now we check that this value of $\mu^{in}$ indeed    corresponds to a choice of $\lambda^{in}$
satisfying \eqref{Snstar}. This will finish the  proof of Proposition \ref{bootstrapprop}.
For this,
we set $\Omega_n=[n^{-1}-n^{-1-\frac \alpha2},n^{-1}+n^{-1-\frac \alpha2}]$ and
 we study  the map:
\[
\lambda^{in}  \in \Omega_n  \mapsto \mu^{in}=(1-J(S_{n}))^{\frac 1{p_{0}}} \lambda^{in}
=\left( 1- \int \varepsilon^{in}(y) \rho(y) \chi(-y n^{-\frac 23}) dy\right)^{\frac 1{p_{0}}} \lambda^{in},
\]
where from \eqref{decomposition.1} and \eqref{Pb}, $\varepsilon^{in} = -a^{in}Q + \lambda^{in}  \chi(y \lambda^{in}) P$. From the definition of $a^{in}$ in \eqref{ain} and \eqref{Ea}, it is clear that the map $:\lambda^{in} \mapsto a^{in}$
is independent of $n$,  smooth and bijective in a neighborhood of $0$ and $\frac{d a^{in}}{d\lambda^{in}}|_{\lambda^{in}=0 }=   {p_0}/{\int Q^2}>0$. Since
\begin{multline*}
\frac{d }{d\lambda^{in}}\int \varepsilon^{in} \rho \chi(-y n^{-\frac 23}) 
 \\= -\frac{d a^{in}} {d\lambda^{in}} \int Q \rho  \chi(-y n^{-\frac 23}) 
+   \int \chi(y\lambda^{in}) P \rho \chi(-y n^{-\frac 23})
  + \lambda^{in} \int y  \chi'(y\lambda^{in}) P \rho   ,
\end{multline*}
where $|y\rho(y)|\lesssim 1$ for $y<-1$,
we see that the map $:\lambda^{in}    \mapsto \mu^{in}(\lambda^{in})$ is $C^1$ on $\Omega_n$  and
that $\frac{d \mu^{in}}{d\lambda^{in}}|_{\lambda^{in}\in \Omega_n}\in [\frac 12,\frac 32]$ for all $n$ large enough.
Moreover, by the properties of $P$ and $\rho$, 
\[\mu^{in}|_{\lambda^{in}=n^{-1}}=n^{-1} (1+\mathcal O(n^{-1}\log n)),\]
\[\mu^{in}|_{\lambda^{in}=n^{-1}+n^{-1-\frac \alpha 2}}
\geq \mu^{in}|_{\lambda^{in}=n^{-1} }+ \frac 12 n^{-1-\frac \alpha 2} 
\geq n^{-1} + 2 n^{-1-\alpha},\]
for $n\geq n_0$, $n_0$ large enough.
Therefore, for $n$ large enough, the map $\mu^{in}$ is one-to-one from $\Omega_n$ to $\mu^{in}(\Omega_n)$ and 
 $\mu^{in}(\Omega_n)$ contains the interval $[n^{-1}-n^{-1-  \alpha  },n^{-1}+n^{-1-  \alpha  }]$.

\section{Compactness arguments}\label{S6}
Going back to the original variables $(t,x)$, we claim from Proposition~\ref{bootstrapprop} that there exist $n_{0}>0$ large  and $t_{0}>0$ small such that for any $n\geq n_{0}$, the solutions $\{u_{n}\}$ defined in 
Sect.~\ref{s:BS} satisfy, for all $t\in [T_{n},t_{0}]$,
\begin{equation}\label{CC.1}
\begin{aligned}
 & \| \varepsilon_n(t)\|_{L^2} \lesssim t^{\frac 12},  &&\| \varepsilon_n (t)\|_{\dot H^{\frac 12}} + \| \varepsilon_n(t)\|_{L^2(y>-t^{-\frac 35})}\lesssim t^{\frac 23 + 2\alpha}.	\\
 & |\lambda_{n}(t)-t|+ |b_{n}(t)+t| \lesssim t^{1+\frac \alpha 2},
 && \left|x_{n}(t)+|\ln t|\right| \lesssim t^{\frac \alpha 2}.
 \end{aligned}
 \end{equation}
(Recall that $\alpha = \frac 16 -\frac \theta 4 \in (0,\frac 1{60})$.)
Indeed, from \eqref{bootstrap.2}, we have $\lambda_{n}(s)=|s|^{-1} + \mathcal O(|s|^{-1-\frac \alpha 2})$ and so \eqref{sn} rewrites as
\[
t-T_{n} = \int_{S_{n}}^s \lambda_{n}^2(s') ds' = \left( |s|^{-1} - |S_{n}|^{-1}\right) + \mathcal O(|s|^{-1-\frac \alpha 2}). 
\]
Since $|S_{n}|^{-1}=T_{n}$, it follows that $t=|s|^{-1} + \mathcal O(|s|^{-1-\frac \alpha 2})$ as $|s|\to +\infty$ and, equivalently,
\begin{equation}\label{s.to.t}
|s|^{-1} = t+\mathcal O(t^{1+\frac \alpha 2}),\quad 
|s| = t^{-1} + \mathcal O( t^{-1+\frac \alpha 2}) \quad \hbox{as $t\downarrow 0$}.
\end{equation}
Thus, the estimates on $\lambda_{n}(t)$, $b_{n}(t)$ and $x_{n}(t)$ in \eqref{CC.1} follow directly from \eqref{bootstrap.2}, \eqref{bootstrap.3} and \eqref{s.to.t}. The estimates on $\| \varepsilon_{n}(t)\|_{L^2}$ and $\| \varepsilon_{n} (t)\|_{\dot H^{\frac 12}}$ follow from \eqref{bootstrap.4} and \eqref{BS.1}. Finally, from the definition of $\varphi$ in \eqref{def.varphi}, it follows that for $y>-|s|^{\theta}$, $\varphi(s,y)\gtrsim 1$. Thus, from \eqref{bootstrap.4}, we have
\[
\int_{y>-|s|^{\theta}} \varepsilon_{n}^2(s,y) dy \lesssim |s|^{-2+\theta}.
\]
Since $t^{-\frac 35} < |s|^{\theta}$ choosing $t_{0}$ small enough, we also obtain
\begin{equation}\label{CC.2}
\int_{y>-t^{-\frac 35}} \varepsilon_{n}^2(t,y) dy \lesssim t^{2-\theta}
=t^{\frac 43 + 4 \alpha}.
\end{equation}
This completes the proof of \eqref{CC.1}.

These estimates imply that the sequences $\{\lambda_{n}(t_{0})\}$, $\{x_{n}(t_{0}\}$ are bounded, and the sequence
$\{v_{n}\}$ defined by
\[
v_{n} = \lambda_{n}^{\frac 12}(t_{0}) u_{n}(t_{0},\lambda_{n}(t_{0})\cdot +x_{n}(t_{0}))
=Q_{b_{n}(t_{0})} + \varepsilon_{n}(t_{0})
\]
is bounded in $H^{\frac 12}$. Therefore, there exist  subsequences $\{v_{n_{k}}\}$, $\{\lambda_{n_{k}}(t_{0})\}$, $\{x_{n_{k}}(t_{0})\}$ and $v_\infty\in H^{\frac 12}$,
$\lambda_\infty>0$, $x_{\infty}$,
such that 
\begin{equation} \label{CC.2b}
v_{n_{k}}\underset{k \to \infty}{\rightharpoonup} v_\infty \ \text{weakly in} \ H^{\frac 12}, \quad \lambda_{n_{k}}(t_{0})\underset{k\to +\infty}{ \to} \lambda_\infty \quad \text{and}    \quad x_{n_{k}}(t_{0})\underset{k \to \infty}{\to} x_\infty \, .
\end{equation}

Let $v_{k}$ be the maximal solution of \eqref{mBO} such that $v_{k}(0)=v_{n_{k}}$ and let
$(\lambda_{k},x_{k},b_{k},\varepsilon_{k})$ be its decomposition as given by Lemma~\ref{modulation}. Then, by the scaling invariance \eqref{ulambda} and the uniqueness of the Cauchy problem,
\[ v_k(t,\cdot)=\lambda_{n_k}^{\frac12}(t_0)u_{n_k}\big(t_0+\lambda_{n_k}^2(t_0)t,\lambda_{n_k}(t_0)\cdot+x_{n_k}(t_0)\big), \quad \forall \, t \in \big[-\frac{t_0-T_{n_k}}{\lambda_{n_k}^2(t_0)},0\big] \, .\]
Hence, it follows from the uniqueness of the decomposition in Lemma \ref{modulation} that\begin{equation}\label{CC.3}\begin{aligned}
 & \lambda_{k}(t)=\frac {\lambda_{n_{k}}(t_{0}+\lambda_{n_{k}}^2(t_{0}) t)}{\lambda_{n_{k}}(t_{0})} \, ,
& x_{k}(t)&=\frac {x_{n_{k}}(t_{0}+\lambda_{n_{k}}^2(t_{0}) t)-x_{n_{k}}(t_{0})}{\lambda_{n_{k}}(t_{0})} \, ,\\
 & b_{k}(t)=b_{n_{k}}(t_{0}+\lambda_{n_{k}}^2(t_{0}) t) \, ,
& \varepsilon_{k}(t)&= \varepsilon_{n_{k}}(t_{0}+\lambda_{n_{k}}^2(t_{0}) t) \, .
 \end{aligned}\end{equation}
Now let $T_1>0$ be such that $T_1<\frac{t_0}{\lambda_\infty^2}$,
so that we may apply Lemma~\ref{le:weak} to $v_k(t)$ on $[-T_1,0]$,
since the conditions~\eqref{hypoweak} are fulfilled for $k$ large enough on such an interval.
We obtain that the solution $v(t)$ of~\eqref{mBO} such that $v(0)=v_\infty$
exists on $\left(-\frac{t_0}{\lambda_\infty^2},0\right]$,
and its decomposition $(\lambda_v,x_v,b_v,\varepsilon_v)$ satisfies, for all $t \in \left(-\frac{t_0}{\lambda_\infty^2},0\right] $,
\begin{equation} \label{CC.4}
\begin{aligned}
\lambda_k(t) \underset{k \to +\infty}{\to} \lambda_v(t), 
\ x_k(t) \underset{k \to +\infty}{\to} x_v(t), \ b_k(t) \underset{k \to +\infty}{\to} b_v(t), \  \varepsilon_k(t) \underset{k \to +\infty}{\rightharpoonup} \varepsilon_v(t) \ \text{weak in} \ H^{\frac12}.
\end{aligned}
\end{equation}

Then we define the solution $S(t)$ of~\eqref{mBO}, for all $t\in (0,t_0]$, by
\[
S(t,x) = \frac{1}{\lambda_\infty^{\frac12}}\, v\left(\frac{t-t_0}{\lambda_\infty^2},\frac{x-x_\infty}{\lambda_\infty}\right),
\]
and denote $( \lambda, x, b, \varepsilon)$ its decomposition. Once again, the uniqueness of the decomposition in Lemma \ref{modulation} yields
\begin{equation}\label{CC.5}
\begin{aligned}
 & \lambda(t)=\lambda_{\infty} \lambda_v\big(\frac{t-t_0}{\lambda_{\infty}^2}\big)\, ,
& x(t)&=x_{\infty}+\lambda_{\infty} x_v\big(\frac{t-t_0}{\lambda_{\infty}^2}\big) \, ,\\
 & b(t)=b_v\big(\frac{t-t_0}{\lambda_{\infty}^2}\big)  \, ,
& \varepsilon(t)&= \varepsilon_v\big(\frac{t-t_0}{\lambda_{\infty}^2}\big) \, .
 \end{aligned}
 \end{equation}

Therefore, we obtain gathering \eqref{CC.2b}-\eqref{CC.5}, for all  $t\in (0,t_0]$,
\begin{displaymath} 
\begin{aligned}
\lambda_{n_k}(t) \underset{k \to +\infty}{\to} \lambda(t), 
\ x_{n_k}(t) \underset{k \to +\infty}{\to} x(t), \ b_{n_k}(t) \underset{k \to +\infty}{\to} b(t), \  \varepsilon_{n_k}(t) \underset{k \to +\infty}{\rightharpoonup} \varepsilon(t) \ \text{weak in} \ H^{\frac12} ,
\end{aligned}
\end{displaymath}
which implies together with \eqref{CC.1} that
\begin{equation}\label{gbetanlim}\begin{aligned}
&\| \varepsilon(t)\|_{L^2} \lesssim t^{\frac 12},
&&\|\varepsilon(t)\|_{\dot H^{\frac 12}}+ \| \varepsilon(t)\|_{L^2(y>-t^{-\frac 35})}\lesssim t^{\frac 23+4\alpha},  \\
&| \lambda(t) - t| +| b(t)+t|\lesssim t^{1+\frac \alpha 2},
&& \left| x(t)+|\ln t|\right|\lesssim t^{\frac \alpha 2}. 
\end{aligned}\end{equation}
Since $\lambda(t)\to 0$ as $t\downarrow 0$, the solution $S(t)$ blows up at time $0$.
Moreover, by weak convergence and \eqref{inittrois}, we have
\[
\| S(t)\|_{L^2} = \|v_\infty\|_{L^2} \leq \liminf_{k\to +\infty} \|v_{n_k}\|_{L^2} = \liminf_{k\to +\infty} \|u_{n_k}(t_0)\|_{L^2}
= \liminf_{k\to +\infty} \|u_{n_k}(T_{n_k})\|_{L^2}  = \|Q\|_{L^2},
\]
and thus $\| S(t)\|_{L^2}=\|Q\|_{L^2}$
(recall that solutions with $\|u_0\|_{L^2}<\|Q\|_{L^2}$ are global and bounded in $H^{\frac 12}$).

To finish the proof of Theorem~\ref{th1}, we recall that
\begin{equation}\label{def.S}
S(t,x) =  \frac1{\lambda^{\frac12} (t)} \left(Q+b(t)P_{b(t)}+\varepsilon\right)
\left(t,\frac{x-x(t)}{\lambda(t)}\right) \, ,
\end{equation}
and so
\begin{equation}\label{def.DS}\begin{aligned}
D^{\frac 12} S(t,x) 
&=   \frac1{\lambda^{\frac12}(t)}D^{\frac 12}Q\left(t,\frac{x-x(t)}{\lambda(t)}\right)
+\frac {b(t)}{\lambda^{\frac12}(t)} D^{\frac 12} P_{b(t)}\left(t,\frac{x-x(t)}{\lambda(t)}\right) \\
&+\frac1{\lambda^{\frac12}(t)}D^{\frac 12} \varepsilon\left(t,\frac{x-x(t)}{\lambda(t)}\right)  \, ,
\end{aligned}\end{equation}
By \eqref{bd.PbRb} and \eqref{gbetanlim}, we have $\|b(t)P_{b(t)}\|_{L^2}\lesssim |b|^{\frac 12}\lesssim t^{\frac 12}$
and $\| \varepsilon\|_{L^2}\lesssim t^{\frac 12}$. Thus,
\begin{equation}\label{unun}
\left\| S(t) -  \frac1{\lambda^{\frac12} (t)}  Q 
\left(\frac{.-x(t)}{\lambda(t)}\right)\right\|_{L^2}\lesssim t^{\frac 12}.
\end{equation}
Similarly, it follows from  \eqref{bd.PbRb} and \eqref{gbetanlim} that
$\lambda^{-\frac 12} \|b D^{\frac 12}P_{b}\|_{L^2}\lesssim \lambda^{-\frac 12} |b| |\ln |b||^{\frac 12}\lesssim t^{\frac 12} |\ln t|^{\frac 12}$ and $\lambda^{-\frac 12} \|D^{\frac 12} \varepsilon\|_{L^2}\lesssim t^{\frac 16 + 4\alpha}$, and
thus
\begin{equation}\label{deuxdeux}
\left\| D^{\frac 12}\left[S(t) -  \frac1{\lambda^{1/2} (t)}  Q 
\left(\frac{.-x(t)}{\lambda(t)}\right)\right] \right\|_{L^2}\lesssim t^{\frac 16}.
\end{equation}
 Note to finish that \eqref{unun} and \eqref{deuxdeux} imply \eqref{d:S}.

\begin{remark}\label{rk.precise}
Note that by \eqref{gbetanlim}, we  get 
the estimate $|\frac{\lambda(t)}{t}-1|\lesssim t^{\frac \alpha 2}$, and thus
\[
\left\| S(t) -  \frac1{t^{\frac12} }  Q 
\left(\frac{.+|\ln t|)}{t}\right)\right\|_{L^2}\lesssim t^{\frac \alpha 2},
\]
but such an estimate cannot be established in $H^{\frac 12}$ because of the singularity.
This reflects a certain lack of precision of the ansatz $Q_{b}$. As observed in Remark~\ref{rk:profile}, it does not seem clear how to improve the ansatz without creating serious additional technical difficulties.
\end{remark}

\end{document}